\documentclass[reqno]{article}
\usepackage{authblk}
\pdfoutput=1
\usepackage{amsmath}
\usepackage{amsfonts,amsthm,amssymb}
\usepackage{mathabx}
\usepackage{bm}
\usepackage{amscd}
\usepackage[latin2]{inputenc}
\usepackage{t1enc}
\usepackage[mathscr]{eucal}
\usepackage{indentfirst}
\usepackage{graphicx}
\usepackage{graphics}
\usepackage{pict2e}
\usepackage{epic}
\numberwithin{equation}{section}
\usepackage[margin=3cm]{geometry}
\usepackage{epstopdf} 
\usepackage[colorlinks,linkcolor=blue]{hyperref}
\usepackage[capitalise,noabbrev]{cleveref}
\usepackage{todonotes}
\makeatletter
\def\thanks#1{\protected@xdef\@thanks{\@thanks
        \protect\footnotetext{#1}}}
\makeatother
\crefformat{equation}{(#2#1#3)}
\crefmultiformat{equation}{(#2#1#3)}{ and~(#2#1#3)}{, (#2#1#3)}{ and~(#2#1#3)}
\crefrangeformat{equation}{(#3#1#4) to~(#5#2#6)}
\usepackage[normalem]{ulem}

\allowdisplaybreaks

\theoremstyle{plain}
\newtheorem{Thm}{Theorem}[section]
\newtheorem*{Thm*}{Theorem}
\newtheorem{Lem}[Thm]{Lemma}

\newtheorem{Prop}[Thm]{Proposition}

\theoremstyle{definition}

\newtheorem{Rem}[Thm]{Remark}
\newtheorem{?}[Thm]{Problem}

\newcommand{\Do}{\mathbf{D}_{0}}

\newcommand{\p}{\partial}
\newcommand{\R}{\mathbb{R}}

\newcommand{\rhom}{\mathring{\rho}}

\newcommand{\phim}{\mathring{\phi}}
\newcommand{\psim}{\mathring{\psi}}
\newcommand{\varphim}{\mathring{\varphi}}

\newcommand{\Torus}{\mathbb{T}}

\newcommand{\dv}{\text{div}}

\newcommand{\lap}{\triangle}
\newcommand{\lapt}{\tilde{\triangle}}

\newcommand{\nablat }{\tilde{\nabla}}

\newcommand{\uv}{\mathbf{u}}

\newcommand{\pt}{\tilde{\partial}}

\newcommand{\psib}{\mathring{\psi} }

\newcommand{\B}{\mathcal{B}}

\newcommand{\ra}{\rangle}
\newcommand{\la}{\langle}

\newcommand{\Ylap}{\pt_Y{\lapt}^{-1}}
\newcommand{\Xlap}{\p_X{\lapt}^{-1}}
\newcommand{\XYlap}{\pt_{XY}{\lapt}^{-1}}

\newcommand{\XXlap}{\p_{X}^2{\lapt}^{-1}}

\newcommand{\abs}[1]{\left\lvert#1\right\rvert}

\newcommand{\norm}[1]{\left\lVert#1\right\rVert}

\begin{document}
	
	\begin{titlepage}
		\title{Nonlinear stability threshold for compressible Couette flow}

		\author{Feimin Huang$^{1,2}$}
		\author{Rui Li$^{1,2}$
		}
		\author{Lingda Xu$^{3}$
		}
		\affil{\begin{flushleft}
				\footnotesize\qquad\quad $ ^1 $ Academy of Mathematics and Systems Science, Chinese Academy of Sciences, Beijing 100190, China.\\
				\vspace{0.07cm}
				\qquad \quad$^2  $ School of Mathematical Sciences, University of Chinese Academy of Sciences, Beijing 100049, China. \\
				\vspace{0.07cm}
				\qquad\quad	 $ ^3 $ Department of Applied Mathematics, The Hong Kong Polytechnic University, Hong
				Kong, China.  \\
				\thanks{Emails:\ fhuang@amt.ac.cn (F. Huang), lirui202@mails.ucas.ac.cn (R. Li), lingda.xu@polyu.edu.hk (L. Xu)}
		\end{flushleft}}
		\date{}

	\end{titlepage}
	
	\maketitle

	\begin{abstract}
		This paper concerns the Couette flow for 2-D compressible Navier-Stokes equations (N-S) in an infinitely long flat torus $\Torus\times\R$. Compared to the incompressible flow, the compressible Couette flow has a stronger lift-up effect and weaker dissipation. To the best of our knowledge, there has been no work on the nonlinear stability in the cases of high Reynolds number until now and only linear stability was known in \cite{ADM2021,ZZZ2022}. 
		In this paper, we study the nonlinear stability of 2-D compressible Couette flow in Sobolev space at high Reynolds numbers. Moreover, we also show the enhanced dissipation phenomenon and stability threshold for the compressible Couette flow. 
		
		First, We decompose the perturbation into zero and non-zero modes and obtain two systems for these components, respectively. Different from  \cite{ADM2021,ZZZ2022},  we use the anti-derivative technique to study the zero-mode system. We introduce a kind of diffusion wave to remove 
		the excessive mass of the zero-modes and construct coupled diffusion waves along characteristics to improve the resulting time decay rates of error terms, and derive a new integrated system \cref{anti}. Secondly, we observe a cancellation with the new system \cref{anti} so that the lift-up effect is weakened. Thirdly, the large time behavior of the zero-modes is obtained by the weighted energy method and a weighted inequality on the heat kernel \cite{HLM2010}.
		In addition, with the help of the Fourier multipliers method, we can show the enhanced dissipation phenomenon for the non-zero modes by commutator estimates to avoid loss of derivatives. Finally, we complete the higher-order derivative estimates to close the a priori assumptions by the energy method and show the stability threshold.

	\end{abstract}
	
	\
	

	\

	
	\section{Introduction}
	
	We consider the isentropic Navier-Stokes system
	\begin{align}
		& \partial_t {\rho}+\operatorname{div} \boldsymbol{m}=0, \quad \text { for }(x, y) \in \mathbb{T} \times \mathbb{R}, t \geq 0,\label{equ-density} \\
		& \partial_t \boldsymbol{m}+\operatorname{div}(\boldsymbol{m} \otimes \boldsymbol{u})+ \frac{1}{M^2}\nabla P({\rho})=\mu\Delta \boldsymbol{u}+(\lambda +\mu)\nabla \operatorname{div}\boldsymbol{u}\label{equ-mom}, 
	\end{align}
	in an infinitely long flat torus with $\mathbb{T}=\mathbb{R} / \mathbb{Z}$ with the initial data $(\rho,\boldsymbol{m})(x,y,0)$. Here ${\rho}$ is the density of the fluid, $\boldsymbol{m}=\rho\boldsymbol{u}$ is the momentum with the velocity $\boldsymbol{u}$, $P({\rho})$ is the pressure, $M>0$ is the Mach number and $\mu, \mu+\lambda \geq 0$ are the shear and bulk viscosity coefficients, respectively. The shear viscosity is proportional to the inverse of the Reynolds number. 
	There has a special solution $\rho_E=1$, $u_E=(y,0)$ to \cref{equ-density}-\cref{equ-mom}, which is called Couette flow. There also has Couette flow with the same formula for the incompressible Navier-Stokes system. The Couette flow can be regarded as a monotonic lamilar shear flow. 
	
	The stability of lamilar shear flow at high Reynolds number 
	is a fundamental problem in the mathematical topics of the fluid mechanics for both the compressible and incompressible flows since Reynolds' early experiments \cite{Reynolds}.  In 1973, Romanov \cite{Romanov1973} firstly proved that Couette flow is 
	spectrally stable for all Reynolds numbers.
	However, experimental evidence \cite{Reynolds} hinted that shear flow is unstable and transitions to turbulence at high Reynolds numbers. Therefore the nonlinear term should play an important role in the stability issue and it is of great interest to quantify transition thresholds \cite{TTRD1993,BGM2019BAMS}. In order to suppress the generation of turbulence, how small does initial perturbation need to be with respect to the viscosity parameter?
	Recently, Bedrossian, Germain, and Masmoudi \cite{BGM2017ann} first studied the stability threshold in Sobolev space with an important index $\frac{3}{2}$, see also \cite{BGM2019BAMS} for a survey. In the Sobolev space with lower regularity, Wei and Zhang \cite{WZ2021} obtained similar works with index 1.  We refer to  \cite{BGM2020AMS, BH2020, BMV2016, CLWZ2020,CWZ2020,CWZ2019} for the related works on Gevrey space, inviscid limit, the boundary layer effects, and inhomogeneous flows. 
	Moreover, all of the above works involved enhanced dissipation, where the perturbation around shear flow decays to zero at a faster rate than that generated only by the viscous dissipation.  We refer to \cite{CKRZ2008, DH, RY 83, W2021} for the background of enhanced dissipation, and \cite{BM2015,IJ2020,IJ2023,MZ2019,WZZ2018} for the inviscid damping, where the perturbed velocity strongly converges to zero in $L^2$ with polynomial rates. 
	
	However, there is significantly less literature concerning the mathematical study of the stability of compressible Couette flow. Compared with the incompressible flow, the compressible flow has stronger lift-up effects. Indeed, a transient growth mechanism was numerically investigated by Hanifi et al \cite{HMRW1985}, see also \cite{FI2000}.  In 2011, Kagei \cite{Kagei2011 JFMF} proved the asymptotic stability of plane Couette flow in an infinite layer when Reynolds number and Mach number are small, see also \cite{kagei2011 JDE, Kagei2012} for the other shear flows. Later,  Li and Zhang \cite{LZ2017} studied the case of the Navier-slip boundary condition. Recently, Antonelli, Dolce, and Marcati \cite{ADM2021} showed linear stability and enhanced dissipation for 2-d Couette flow on $\mathbb{T} \times\mathbb{R} $ at high Reynolds number, and Zeng, Zhang, and Zi \cite{ZZZ2022} obtained 
	similar results in $\mathbb{T} \times\mathbb{R} \times\mathbb{T} $. 
	
	In this paper, we focus on the nonlinear stability threshold of Couette flow for the compressible Navier-Stokes system \eqref{equ-density} and \eqref{equ-mom} at high Reynolds number. That is, we consider a perturbation around the Couette flow by
	\begin{align}\label{per-va}
		{\rho}=1+\phi, \quad \boldsymbol{u}=u_E+\psi, \quad \boldsymbol{m}=u_E+\varphi, \quad \psi=(\psi_1,\psi_2), \quad \varphi=(\varphi_1,\varphi_2),
	\end{align}
	for $\rho$, $\uv$, $\boldsymbol{m}$, satisfying \cref{equ-density,equ-mom}. 
	The main difficulty comes from the so-called lift-up effect,  which means that the presence of $\psi_2$ in the equation of $\psi_1$ might make the $L^2$-{norm} of $\psi_1$ grow in time, see \cref{perturbation1}$_2$ below.  For the incompressible flow, the lift-up effect only occurs in 3-D, and the divergence-free essentially guarantees the uniform boundness of $\|\psi_1\|_{L^2}$, cf. \cite{BGM2017ann,BGM2020AMS,WZ2021}. However, for the compressible flow, the lift-up effect is very strong, and even for the linear case, the best work so far is that the $L^2$-norm of $\psi_1$ grows at a rate of $(1+t)^{3/4}$, cf. \cite{ZZZ2022}. 
	
	In this paper, we introduce a new energy method, which is based on the anti-derivative technique for
	the zero-modes and the construction of two kinds of diffusion waves, to prove the nonlinear stability of  Couette flow at high Reynolds number. Moreover, we also show the enhanced dissipation phenomenon and stability threshold with index $\frac{11}{3}$. The precise statement is stated in Theorem \ref{MT} and \ref{MT2} below.
	
	The strategy is as follows. 	We firstly decompose the perturbation $(\phi,\psi)$ into two parts: zero modes $(\phim,\psim_1,\psim_2)$ and non-zero modes $(\phi_\neq,\psi_{1\neq},\psi_{2\neq})$, and obtain two systems for $(\phim,\psim_1,\psim_2)$ and $(\phi_\neq,\psi_{1\neq},\psi_{2\neq})$ respectively. It is noted that the equations for $(\phim,\psim_2)$ and $\psim_1$ are decoupled, we can first consider the system \eqref{equ-rhou2} of $(\phim,\psim_2)$, and then solve the equation \eqref{psi1} for $\psim_1$.
	Secondly, we use the anti-derivative technique \cite{Goodman,HXXY2023,MN1985,YQ} to study the system \eqref{equ-rhou2} of $(\phim,\psim_2)$. Motivated by \cite{HW2007,Koike2023, Lliu1985, LZ1997}, we 
	introduce two kinds of diffusion waves to derive a new integrated system \cref{anti}, and capture
	the large time behavior of $(\phim, \psim_2)$. Thirdly, we obtain decay rates of the $L^2$ norms of zero modes $(\phim,\psim_2)$ by a weighted energy method and a weighted inequality on the heat kernel \cite{HLM2010}. 
	It remains to solve the equation \eqref{psi1} for $\psim_1$, and the key point is to control the lift-up effect generated by $\psim_2$. Fortunately, we observe a 
	cancellation in a combination of the equations \eqref{anti} and \eqref{psi1} 
	so that $\psim_2$ does not appear in the equation \cref{equ-A} of  a new quantity $\mathcal{A}$. That is, the lift-up effect generated by $\psim_2$ is weakened. 
	Thus, we can prove  $\psim_1$ is uniformly bounded in $L^\infty$ norm, and obtain the decay rates
	of $\partial_y \psim_1$ in $L^2$ norm, which was not obtained before even for the linearized cases.
	
	For the non-zero modes $(\phi_\neq,\psi_{1\neq},\psi_{2\neq})$, motivated by   and 
	\cite{ADM2021,BGM2017ann,ZZZ2022}, we design Fourier multipliers in terms of the structure of the compressible Navier-Stokes equation and obtain estimates of the commutator for the multipliers to avoid the loss of derivatives. With the help of the decay estimates obtained above for the zero modes $(\phim,\psim_1,\psim_2)$, we can show the enhanced dissipation for the non-zero modes $(\phi_\neq,\psi_{1\neq},\psi_{2\neq})$. Moreover, 
	we can also obtain the stability threshold with index $\frac{11}{3}$.
	
	\
	
	\noindent
	{{\bf{Notations:}} Throughout this paper, we use the following notations}
	
	\begin{itemize}
		
		\item Defining the zero and non-zero modes for an integrable function $f\in \Torus\times\R$,
		\begin{align*}
			\Do f:=\mathring{f}:=\int_{\Torus}fdx,\qquad {f}_{\neq}:=f-\Do f.
		\end{align*}
		\item $\delta>0$ denotes a small constant independent of $\mu$, $\lambda+\mu$ and $t$. 
		\item Given two quantities $A$ and $B$, we denote $A\lesssim B$ and $A \approx B$ if there exists a positive constant $C$ such that $A\leq CB$ and $C^{-1}B \leq A \leq CB$ respectively. The constant might dependent on $M$, but not on $\mu$, $\lambda+\mu$, $t$ and $\delta$.
		\item  For a vector $x$, we denote $\la x \ra:=(1+\abs{x}^2)^{\frac{1}{2}}$.
		\item The Fourier transform $\hat{f}(k,\eta)$ of a function $f(x,y)$ is defined by
		\begin{align*}
			\hat{f}(k,\eta)=\frac{1}{2\pi}\int \int_{\Torus\times \R} f(x,y)e^{-i(kx+\eta y)}dxdy.
		\end{align*}
		Then
		\begin{align*}
			f(x,y)=\frac{1}{2\pi} \sum\limits_{k\in \mathbb{Z}} \int_{\R} \hat{f}(k,\eta) e^{i(kx+\eta y)} d\eta.
		\end{align*}
		\item The Sobolev space $H^N(\Torus \times \R)$, $N\geq0$ is given by the norm 
		$$ \norm{f}_{H^N}^2=\norm{\la D \ra^N f}_{L^2}^2=\sum_{k\in \mathbb{Z}} \int \la k,\eta \ra^{2N} \abs{\hat{f}}^2 (k,\eta)d\eta. $$
		
		\item We simply write $\norm{f}:=\norm{f}_{L^2(\Torus\times\R)}$. For $f,g \in L^2$, we use $\la f, g \ra$ to denote the $L^2$ inner product of $f$ and $g$.
		
		\item $*$ denotes the convolution, i.e.,  $(f*g)(x):=\int_\Omega f(x-y)g(y)dy.$
	\end{itemize}
	\
	
	The rest of the present paper is organized as follows. In Section 2, we derive a new integrated system for the zero modes, introduce a new quantity $\mathcal{A}$ to weaken the lift-up effect, and state the main theorem. In Section 3, we study the system for the zero modes, and in Section 4, we study the system of non-zero modes and obtain the enhanced dissipation. Finally, in Section 5, we obtain the stability threshold of the Couette flow for the compressible Navier-Stokes equations. 
	
	\section{Ansatz and main results}
	\subsection{The perturbation system}
	
	The perturbation systems for $(\phi,\psi) $ around the homogeneous Couette flow read as follows
	\begin{align}\label{perturbation1}
		\left\{\begin{aligned}
			&\p_t \phi +y\p_x \phi + \dv \psi=-\dv(\phi\psi), \qquad \qquad \text { for }(x, y) \in \mathbb{T} \times \mathbb{R}, t \geq 0,\\
			&\p_t \psi+y\p_x\psi+\psi_{2}\mathbf{e}_1-\mu\lap\psi-(\lambda+\mu) \nabla \dv  \psi + \frac{1}{M^2}\nabla\phi\\
			&=-\psi\cdot\nabla\psi-\frac{\phi}{\phi+1}\big(\mu\lap\psi+(\lambda+\mu)\nabla\dv\psi\big)+\frac{\nabla\phi}{M^2}\frac{\phi}{\phi+1}-\frac{1}{M^2}\frac{(P'(\rho)-1)\nabla\phi}{\phi+1},
		\end{aligned}\right.
	\end{align}
with the initial data $(\phi,\psi)(x,y,0)$, where $P'(1)=1$, $\mathbf{e}_1:=(1,0)^{T}$.
	Applying $\Do$ to \cref{perturbation1}$_2$, one has
	\begin{align}\label{psi1}
		\p_t \psim_1-\mu\p_y^2\psim_1=-\psim_{2}-\psim_2\p_y\psim_1-\frac{\phim}{\phim+1}\mu\p_y^2\psim_1+F_1,
	\end{align}
	where
	\begin{align}\label{F1}
		\begin{aligned}
			F_1:=&-\Do\left[\frac{\phi}{\phi+1}\big(\mu\lap\psi_1+(\lambda+\mu)\p_x\dv\psi\big)\right]-\Do(\psi\cdot\nabla\psi_1)+\psim_2\p_y\psim_1\\
			&+\frac{\phim}{\phim+1}\mu\p_y^2\psim_1 +\Do\left[\frac{\p_x\phi}{M^2}\frac{\phi}{\phi+1}-\frac{1}{M^2}\frac{(P'(\rho)-1)\p_x\phi}{\phi+1}\right].
		\end{aligned}
	\end{align}
	
Note that the system \eqref{perturbation1} is not conserved, we instead study the conserved system for $(\phi,\varphi_2)$. Taking $\Do$ on the system \cref{equ-density} and \cref{equ-mom}$_2$ respectively, we obtain a new system for the zero mode $(\phim,\varphim_2)$ as follows, 
	\begin{align}
		\begin{cases}\label{equ-rhou2}
			\p_t \phim + \p_y \mathring{\varphi} _2 = 0, \\
			\p_t \mathring{\varphi} _2 + \frac{1}{M^2} \p_y\phim -( 2\mu +\lambda)\p_y^2 \mathring{\varphi} _2 = E_{y},
		\end{cases}
	\end{align}
	where
	\begin{align}\label{E}
		\begin{aligned}
			E=&\left[F_2-(2\mu+\lambda)\p_y\left(\frac{\varphim_2}{\rhom}-\int_\Torus\frac{\varphi_2}{\rho}dx\right)\right]-(2\mu+\lambda)\p_y\left(\mathring{\varphi} _{2}-\frac{\mathring{\varphi} _2}{\rhom}\right)\\
			&-\left[\frac{1}{M^2}\bigg(P(\rhom)-P(1)-P'(1)\phim\bigg)+ \frac{\mathring{\varphi} _2^2}{\rhom}\right]-\left(\psim_2-\frac{\varphim_2}{\rhom}\right)\mathring{\varphi} _2\\
			:=&E_1+E_2+E_3+E_4,\\
			F_2:=&\psim_2  \varphim_2 +\frac{P(\rhom)}{M^2}-\big(\int_{\Torus} \psi_2 \varphi_2 d x + \int_{\Torus} \frac{P(\rho)}{M^2} d x\big ),
		\end{aligned}
	\end{align}
with the initial data $(\phim,\varphim_2)(y,0)$.
After $(\phim,\varphim_2)$ is solved by \cref{equ-rhou2}, $\psim_1$ can be  obtained through the equation \eqref{psi1} with the initial value $\psim_1(y,0)$.
	\subsection{Construction of ansatz}
	
	We rewrite \cref{equ-rhou2} as
	\begin{align}\label{equ-green}
		w_t+\tilde{A}w_y=\tilde{B} w_{yy}+E_y \mathbf{e}_2,
	\end{align}
	where $\mathbf{e}_2:=(0,1)^{T}$ and
	\begin{align}
		w=\left(\begin{array}{c}\phim\\ \varphim_2 \end{array}\right), \quad \tilde{A}=\left(\begin{array}{cc}0 & 1 \\ {1}/{M^2} & 0\end{array}\right), \quad \tilde{B}=\left(\begin{array}{cc}0 & 0 \\ 0 & \bar{\mu}\end{array}\right),\quad \bar{\mu}:=2\mu+\lambda,\quad \bar{c}=1/M.
	\end{align}
	The eigenvalues of $\tilde{A}$ are $\sigma_1=-\bar{c}$, $\sigma_2=\bar{c}$, and the corresponding left and right eigenvectors are
	\begin{align}
		(l_1,l_2)=\frac{1}{2a}\left(\begin{array}{cc}-1&\frac{1}{\bar{c}}\\
			1& \frac{1}{\bar{c}}
		\end{array}\right), \quad	(r_1,r_2)=a\left(\begin{array}{cc}-1&1\\ \bar{c}&\bar{c}\end{array}\right),\qquad a:=\frac{2}{P''(1)\bar{c}+2\bar{c}},
	\end{align} 
	respectively. We distribute the solution $w$ in the right eigenvector directions, that is 
	\begin{align}\label{phi-psi}
		v=Lw, \quad \text{i.e.} \quad w=R v,
	\end{align}
	where $L=(l_1,l_2)^{t}$, $R=(r_1,r_2).$ Then we diagonalize \cref{equ-rhou2} as
	\begin{align}
		\begin{cases}
			v_{1t}-\bar{c}v_{1y}=\frac{\bar{\mu}}{2}v_{1yy}+\frac{\bar{\mu}}{2}v_{2yy}+\frac{1}{2a\bar{c}}E_y,\\
			v_{2t}+\bar{c}v_{2y}=\frac{\bar{\mu}}{2}v_{2yy}+\frac{\bar{\mu}}{2}v_{1yy}+\frac{1}{2a\bar{c}}E_y,
		\end{cases}
	\end{align}
with the initial data $(v_1,v_2)(y,0)$.
	
Next, we will try to find a suitable ansatz to capture the large time behavior of $(v_1,v_2)$. For this, we construct the following diffusion wave
	\begin{align}\label{equ-theta}
		\begin{cases}
			\theta_{1t}-\bar{c}\theta_{1y}+(\frac{\theta_1^2}{2})_y=\frac{\bar{\mu}}{2}\theta_{1yy},& y \in \mathbb{R}, t>0,\\
			\theta_{2t}+\bar{c}\theta_{2y}+(\frac{\theta_2^2}{2})_y=\frac{\bar{\mu}}{2}\theta_{2yy},& y\in \mathbb{R}, t>0,\\
			\int_{\R}\theta_i(y,0)dy=l_i\int_{\R}w(y,0)dy,\quad i=1,2,& y \in \mathbb{R}.
		\end{cases}
	\end{align}
	One can use the Hopf-Cole transformation to obtain the explicit expressions for $\theta_1$ and $\theta_2$
	\begin{align}\label{theta}
		\begin{aligned}
			&\theta_1(x_1, t)=\frac{\bar{\mu}^{\frac{1}{2}}}{\sqrt{2(1+t)}} \Gamma_1\left(\frac{x+\bar{c}(1+t)}{\sqrt{2\bar{\mu}(1+t)}}\right),\qquad\qquad\theta_2(x_1, t)=\frac{\bar{\mu}^{\frac{1}{2}}}{\sqrt{2(1+t)}} \Gamma_2\left(\frac{x-\bar{c}(1+t)}{\sqrt{2\bar{\mu}(1+t)}}\right),\\
			&\Gamma_i(y)=\frac{\left( e^ \frac{\eta_i}{\bar{\mu}}-1\right) \exp \left(-y^2\right)}{\sqrt{\pi}+\left(e^ \frac{\eta_i}{\bar{\mu}}-1\right) \int_y^{+\infty} \exp \left(-\xi^2\right) d \xi},\qquad \eta_i:=l_i\int_{\R}w(y,0)dy,\quad i=1,2.
		\end{aligned}
	\end{align} 
	
	Setting $\bar{v}_{i}:=v_{i}-\theta_i$, $i=1,2$, one has
	\begin{align}\label{barv}
		\begin{cases}
			\bar{v}_{1t}-\bar{c}\bar{v}_{1y}=\frac{\bar{\mu}}{2}\bar{v}_{1yy}+\frac{\bar{\mu}}{2}\bar{v}_{2yy}+\frac{1}{2a\bar{c}}{E}_y+(\frac{\theta_1^2}{2})_y+\frac{\bar{\mu}}{2}\theta_{2yy},\\
			\bar{v}_{2t}+\bar{c}\bar{v}_{2y}=\frac{\bar{\mu}}{2}\bar{v}_{2yy}+\frac{\bar{\mu}}{2}\bar{v}_{1yy}+\frac{1}{2a\bar{c}}{E}_y+(\frac{\theta_2^2}{2})_y+\frac{\bar{\mu}}{2}\theta_{1yy},\\
			\int_{\R}\bar{v}_i(y,0)dy=0,\quad i=1,2,
		\end{cases}& y \in \mathbb{R},\quad t>0.
	\end{align}
	In view of  \cref{E},  the leading term of ${E}_y$ is 
	\begin{align}\label{tildeE}
	\tilde{E}_{3y}:=-\frac{P''(1)}{2M^2}\left(\phim^2\right)_y-\left( \varphim_2^2 \right)_y.
    \end{align}
	Since $\phim=a\left(v_2-v_1\right)$ and $\varphim_2=a\bar{c}\left(v_1+v_2\right)$, the leading part of $\frac{1}{2a\bar{c}}\tilde{E}_{3y}+\left(\frac{\theta_i^2}{2}\right)_y$ is $-(\frac{\theta_{3-i}^2}{2})_y$, which is bounded by $(1+t)^{-\frac{3}{2}}e^{-\frac{(y\pm\bar{c}t)^2}{1+t}}$. This may not decay fast enough for the later estimates. To eliminate the bad term $-\left(\frac{\theta_{3-i}^2}{2}\right)_{y}$, we shall construct a coupled diffusion wave $\Xi_i$ to capture the leading term of $\bar{v}_i$. 
	Setting $\chi_i=\bar{v}_i-\Xi_i$, one has
	\begin{align}\label{chi11}
	\left\{\begin{aligned}
   &\chi_{1t}-\bar{c}\chi_{1y}-\frac{\bar{\mu}}{2}\chi_{1yy}-\frac{\bar{\mu}}{2}\chi_{2yy}-\frac{1}{2a\bar{c}}\left(E-\tilde{E}_3\right)_y=-\bigg\{\Xi_{1t}-\bar{c}\Xi_{1y}\\
   &\qquad\qquad-\frac{1}{2a\bar{c}}\tilde{E}_{3y}-\left(\frac{\theta_1^2}{2}\right)_y-\frac{\bar{\mu}}{2}\Xi_{1yy}\bigg\}+\frac{\bar{\mu}}{2} \bigg\{\Xi_{2yy} +\theta_{2yy}\bigg\}\\
   &\chi_{2t}+\bar{c}\chi_{2y}-\frac{\bar{\mu}}{2}\chi_{2yy}-\frac{\bar{\mu}}{2}\chi_{1yy}-\frac{1}{2a\bar{c}}\left(E-\tilde{E}_3\right)_y=-\bigg\{\Xi_{2t}+\bar{c}\Xi_{2y}\\
   &\qquad\qquad-\frac{1}{2a\bar{c}}\tilde{E}_{3y}-\left(\frac{\theta_2^2}{2}\right)_y-\frac{\bar{\mu}}{2} \Xi_{2yy}\bigg\}+ \frac{\bar{\mu}}{2}\bigg\{\Xi_{1yy} +\theta_{1yy}\bigg\}.
    \end{aligned}\right.
	\end{align}
	Have $v_i=\chi_i+\theta_i+\Xi_i$ in hand, we calculate  $\frac{1}{2a\bar{c}}\tilde{E}_{3y}+\left(\frac{\theta_i^2}{2}\right)_y$ again and find its leading error term is 
	$-\left[\frac{\theta_{3-i}^2}{2}+\theta_1\Xi_1+\theta_2 \Xi_2+\frac{2-P''(1)}{2+P''(1)}(\theta_1\Xi_2+\theta_2\Xi_1) \right]_y$.
	Therefore $\Xi_i$ is constructed as a solution of the following systems,
	\begin{align}\label{equ-xi}
		\begin{cases}
			\Xi_{1t}-\bar{c}  \Xi_{1y}+\left[\theta_2^2 / 2+\theta_1\Xi_1+\theta_2 \Xi_2+\frac{2-P''(1)}{2+P''(1)}(\theta_1\Xi_2+\theta_2\Xi_1)\right]_y-\frac{\bar{\mu}}{2} \Xi_{1yy}=0, & y \in \mathbb{R}, t>0, \\ 
			\Xi_{2t}+\bar{c} \Xi_{2y}+\left[\theta_1^2 / 2+\theta_1\Xi_1+\theta_2 \Xi_2+\frac{2-P''(1)}{2+P''(1)}(\theta_1\Xi_2+\theta_2\Xi_1)\right]_y-\frac{\bar{\mu}}{2}  \Xi_{2yy}=0, & y \in \mathbb{R}, t>0, \\ \Xi_1(y, 0)=\Xi_2(y, 0)=0, & y \in \mathbb{R}.\end{cases}
	\end{align}
	One can estimate $\Xi_i$ by Duhamel's principle, and  obtain the decay rate  $(1+t)^{-\frac{1}{2}}$ for $\norm{\Xi_i}_{L^2}$, which is indeed better than $(1+t)^{-\frac14}$ for $\norm{\theta_i}_{L^2}$,  see \cref{estimateontheta} and \ref{estimateonxi} below for details. 
     In order to eliminate another bad term $\frac{\bar{\mu}}{2}\left(\theta_{3-i}+\Xi_{3-i}\right)_{yy}$ in \cref{chi11}, set
	\begin{align}\label{tildev}
		&\tilde{v}_i:=\chi_i-\tilde{\mathcal{C}}_i=v_i-\mathcal{C}_i,\qquad i'=3-i,\\
        &\tilde{\mathcal{C}}_i:=(-1)^{i}\frac{\bar{\mu}}{4\bar{c}}\theta_{i'y}+(-1)^{i}\frac{\bar{\mu}}{4\bar{c}}\Xi_{i'y},\quad \mathcal{C}_i = \theta_i +\Xi_i+\tilde{\mathcal{C}}_i.\nonumber
	\end{align}
Then we have
	\begin{align}\label{equ-tildev}
		\begin{cases}
			\tilde{v}_{1t}-\bar{c}\tilde{v}_{1y}=\frac{\bar{\mu}}{2}\tilde{v}_{1yy}+\frac{\bar{\mu}}{2}\tilde{v}_{2yy}+\tilde{K}_{1y},\\
			\tilde{v}_{2t}+\bar{c}\tilde{v}_{2y}=\frac{\bar{\mu}}{2}\tilde{v}_{2yy}+\frac{\bar{\mu}}{2}\tilde{v}_{1yy}+\tilde{K}_{2y},
		\end{cases}
	\end{align}
	where
	\begin{align}\label{tildeK}
			\tilde{K}_{i}&:=\tilde{E}_{i}+\tilde{N},\\ \nonumber
			\tilde{E}_{i}&:=\frac{1}{2a\bar{c}}{E}_{2}+\left(\frac{1}{2a\bar{c}}{E}_{3}+\theta_1^2 / 2+\theta_2^2 / 2+\theta_1\Xi_1+\theta_2 \Xi_2+\frac{2-P''(1)}{2+P''(1)}(\theta_1\Xi_2+\theta_2\Xi_1)\right)\\
			&+(-1)^{i}\frac{\bar{\mu}}{4\bar{c}}\left(\theta_1^2 / 2+\theta_2^2 / 2+\theta_1\Xi_1+\theta_2 \Xi_2+\frac{2-P''(1)}{2+P''(1)}(\theta_1\Xi_2+\theta_2\Xi_1)-\frac{\bar{\mu}}{2}\theta_{iy}-\frac{\bar{\mu}}{2}\Xi_{iy}\right)_{y},\\ \nonumber
			\tilde{N}&:=\frac{1}{2a\bar{c}}(E_{1}+E_{4}).\nonumber
	\end{align}
	It should be noted that the decay rate of  $\tilde{K}_{iy}$ in \cref{equ-tildev} is improved to $(1+t)^{-2}e^{-\frac{(x\pm\bar{c}t)^2}{(1+t)}}+\p_y(\Xi^2)$ due to the cancellations in $\frac{1}{2a\bar{c}}{E}_{3}+\theta_1^2 / 2+\theta_2^2 / 2+\theta_1\Xi_1+\theta_2 \Xi_2+\frac{2-P''(1)}{2+P''(1)}(\theta_1\Xi_2+\theta_2\Xi_1)$. 

	In view of \eqref{equ-theta}-\eqref{equ-tildev}, one has $$\int_{-\infty}^\infty \tilde{v}_i(y,t)dy=0, \quad \int_{-\infty}^\infty \Xi_i(y,t)dy=0, \quad i=1,2.$$ 
	Motivated by \cite{Goodman, HXXY2023, MN1985}, we can define the anti-derivatives of $(\tilde{v}_1,\tilde{v}_2)$ as 
	\begin{align}
		\begin{array}{lll}
			\mathbf{\tilde{V}}:=\left(\tilde{V}_1,\tilde{V}_2\right):=&\int_{-\infty}^y\left(\tilde{v}_1(z,t),\tilde{v}_2(z,t)\right)dz,\quad\mathbf{\tilde{\Xi}}:=\left(\tilde{\Xi}_1,\tilde{\Xi}_2\right):=\int_{-\infty}^y\left({\Xi}_1(z,t),{\Xi}_2(z,t)\right)dz.\label{equ--anti}
		\end{array}
	\end{align}
	Then the systems \cref{equ-tildev} and \cref{equ-xi} become,  respectively,
	\begin{align}\label{anti}
		\mathbf{\tilde{V}}_t+A\mathbf{\tilde{V}}_y=B\mathbf{\tilde{V}}_{yy}+\mathbf{\tilde{K}},
	\end{align}
	where
	\begin{align}
		A=\left(\begin{array}{cc}-\bar{c}& 0 \\ 0 & \bar{c}\end{array}\right), \quad {B}=\left(\begin{array}{cc}\frac{\bar{\mu}}{2} & \frac{\bar{\mu}}{2} \\ \frac{\bar{\mu}}{2} & \frac{\bar{\mu}}{2}\end{array}\right),\quad \mathbf{\tilde{K}}=\left(\begin{array}{c}\tilde{K}_1\\ \tilde{K}_2\end{array}\right),
	\end{align}
	and
	\begin{align}\label{equ-Xi}
		\begin{cases}
			\tilde{\Xi}_{1t}-\bar{c}  \tilde{\Xi}_{1y}+{\left[\theta_2^2 / 2+\theta_1\Xi_1+\theta_2 \Xi_2+\frac{2-P''(1)}{2+P''(1)}(\theta_1\Xi_2+\theta_2\Xi_1)\right]}=\frac{\bar{\mu}}{2}  \tilde{\Xi}_{1yy}, & y \in \mathbb{R}, t>0, \\ 
			\tilde{\Xi}_{2t}+\bar{c} \tilde{\Xi}_{2y}+\left[\theta_1^2 / 2+\theta_1\Xi_1+\theta_2 \Xi_2+\frac{2-P''(1)}{2+P''(1)}(\theta_1\Xi_2+\theta_2\Xi_1)\right]=\frac{\bar{\mu}}{2}  \tilde{\Xi}_{2yy}, & y \in \mathbb{R}, t>0, \\ \tilde{\Xi}_1(y, 0)=\tilde{\Xi}_2(y, 0)=0, & y \in \mathbb{R}.\end{cases}
	\end{align}

	Next, we will introduce a new quantity to weaken the lift-up effect of the equation for $\psim_1$, which reads
	\begin{align}\label{equ-psi1}
		\p_t \psim_1-\mu\p_y^2\psim_1+\mathbf{\psim}_{2}=-\psim_2\p_y\psim_1-\frac{\phim}{\phim+1}\mu\p_y^2\psim_1+F_1,
	\end{align}
	and  the leading term of \cref{equ-psi1} is $\psib_2$. Note that
	\begin{align}\label{phi2}
		\psim_2+\phim\varphim_2-\frac{\phim^2\varphim_2}{\rhom}+\Do(\phi_\neq \psi_{2\neq})+\phim \Do \left(\frac{\varphi_2}{\rho}-\frac{\varphim_2}{\rhom}\right)= \varphim_2=a\bar{c}\left(\tilde{v}_1+\tilde{v}_2+\mathcal{C}_1+\mathcal{C}_2\right),
	\end{align} and
	\begin{align}\label{00000000}
		\bar{c}\left(\tilde{v}_1+\tilde{v}_2+\Xi_1+\Xi_2\right)=-\p_t(\tilde{V}_{2}-\tilde{V}_{1}+\tilde{\Xi}_2-\tilde{\Xi}_1)+\frac{\bar{\mu}}{2}\left(\Xi_{2y}-\Xi_{1y}\right)+\frac{\theta_2^2}{2}-\frac{\theta_1^2}{2}+\tilde{K}_{2}-\tilde{K}_{1},
	\end{align}
	there has a cancellation between \cref{equ-psi1}-\cref{00000000} so that 
	\begin{align}\label{equ-Ab}
		\begin{aligned}
			\p_t\bar{\mathcal{A}}- \mu \p_y^2\bar{\mathcal{A}}&-a\mu\p_y(\tilde{v}_2-\tilde{v}_1)=-\psim_2\p_y\psim_1-\mu\frac{\phim}{\phim+1}\p_y^2\psim_1-\frac{a\lambda}{2} \p_y(\Xi_2-\Xi_1)\\
			&+F_1+\frac{a}{2}(\theta_1^2-\theta_2^2)-a\bar{c}(\mathcal{C}_1+\mathcal{C}_2-\Xi_1-\Xi_2)-a (\tilde{K}_{2}-\tilde{K}_{1})\\
			&+\phim\varphim_2-\frac{1}{\rhom}\phim^2\varphim_2+\Do(\phi_{\neq}\psi_{2\neq})+\phim\Do\left(\frac{\varphi_2}{\rho}-\frac{\varphim_2}{\rhom}\right),
		\end{aligned}
	\end{align}
and $\psim_2$ is eliminated, where $\bar{\mathcal{A}}:=\psim_1-a(\tilde{V}_2-\tilde{V}_1+\tilde{\Xi}_2-\tilde{\Xi}_1)$, with the price that a new bad term $\mathcal{C}_1+\mathcal{C}_2-\Xi_1-\Xi_2$ bounded by 
$(1+t)^{-\frac{1}{2}}e^{-\frac{(y\pm\bar{c}t)^2}{1+t}}$,  comes out in the right hand side. To eliminate the bad term, let $\theta_A$ be the solution of the following equation 
	\begin{align}\label{thetaA}
		\left\{\begin{aligned}
			&{\p_t\theta_{A}}-\mu\p_y^2\theta_{A}=-a\bar{c}(\mathcal{C}_1+\mathcal{C}_2-\Xi_1-\Xi_2)-\frac{a\lambda}{2}\left(\Xi_2-\Xi_1\right)_y+\frac{a}{2}(\theta_1^2-\theta_2^2)\\
			&\qquad\qquad\qquad-a\bar{c}(\mathcal{C}_1+\mathcal{C}_2)\p_y\theta_A+a^2\bar{c}(\mathcal{C}_1+\mathcal{C}_2)(\mathcal{C}_2-\Xi_2-\mathcal{C}_1+\Xi_1)\\
			&\qquad\qquad\qquad-a\bar{c}(\tilde{v}_1+\tilde{v}_2)\p_y\theta_A+a^2\bar{c}(\tilde{v}_1+\tilde{v}_2)(\mathcal{C}_2-\mathcal{C}_1),\\
			&\theta_A(y,0)=0.
		\end{aligned}\right.
	\end{align}
	We expect $\theta_A$ is  the leading part of $\bar{\mathcal{A}}$. Denote $\mathcal{A}:=\bar{\mathcal{A}}-\theta_A$, then one has that 
	\begin{align}\label{equ-A}
		&\qquad\begin{aligned}
			\p_t{\mathcal{A}}- \mu \p_y^2{\mathcal{A}}-a\mu\p_y(\tilde{v}_2-\tilde{v}_1)=-\psim_{2}\p_y\mathcal{A}-\mu\frac{\phim}{\phim+1}\p_y^2\mathcal{A}+F_{A},
		\end{aligned}\\
		&\begin{aligned}\label{FA}
			{F_{A}:=}&a^2\bar{c}(\tilde{v}_1+\tilde{v}_2)(\Xi_1-\Xi_2)+a\mu\frac{\phim}{\phim+1}\left(\tilde{v}_1-\tilde{v}_2+\Xi_1-\Xi_2\right)_y-a\mu\frac{\phim}{\phim+1}\p_y^2\theta_A\\
			&-\frac{1}{\rhom}\phim^2\varphim_2+\Do(\phi_{\neq}\psi_{2\neq})+\phim\Do\left(\frac{\varphi_2}{\rho}-\frac{\varphim_2}{\rhom}\right)+F_1-a (\tilde{K}_{2}-\tilde{K}_{1})\\
			&+\left\{\frac{\phim \varphim_2}{\rhom}-\Do\left(\frac{\varphi_2}{\rho}-\frac{\varphim_2}{\rhom} \right) \right\}\Big(\p_y \theta_A + a\left(\tilde{v}_2-\tilde{v}_1-\Xi_1+\Xi_2 \right)\Big).
		\end{aligned}
	\end{align}
	\begin{Rem}
		Noted that the term $\psim_2$ in \eqref{equ-psi1} corresponds to  $-a\mu\p_y(\tilde{v}_2-\tilde{v}_1)$ in the new equation \eqref{equ-A}, thus the lift-up effect is weakened in  \eqref{equ-A}.
	\end{Rem}
	{\subsection{Main results}}
	Now we are ready to state the main results.
	\begin{Thm}\label{MT}
		There exists a sufficiently small constant $\delta_0>0$ such that if    $0<\mu\leqslant \delta_0$ and $\lambda+\mu$ is the same order of $\mu$, 
		$0<M\leqslant \frac{1}{50} \min\{(\lambda+2\mu)^{-\frac{1}{2}},\mu^{-\frac{1}{3}}\}$,
		and the initial data satisfies
		\begin{align}\label{initial}
			\begin{aligned}
				(\tilde{V}_1&,\tilde{V}_2)(y,0)\in H^{2}(\R),\qquad  (\phi,\varphi_2)(x,y,0)\in L^1\cap H^{5}(\Torus\times\R),\quad \psi_1(x,y,0)\in H^{5}(\Torus\times\R),\\
				&\norm{(\tilde{V}_1,\tilde{V}_2)(y,0)}_{H^{2}}+\norm{(\phi,\varphi_2)(x,y,0)}_{L^1}+\norm{(\phi,\psi_1,\psi_2)(x,y,0)}_{H^{5}}\leq C \mu^{\alpha},
			\end{aligned}
		\end{align}
		where $\alpha>\frac{11}{3}$, $C>0$ are constants, then the system \cref{perturbation1} admits a unique global  solution $(\phi,\psi_1,\psi_2)$. Moreover, it holds that  
		\begin{align}
			\begin{aligned}
				&\norm{\phi,\psi_2}_{L^{2}}\leq C_M \mu^{\alpha-\frac{1}{2}}(1+t)^{-\frac{1}{4}},\\ &\norm{\nabla\phi,\nabla\psi_2}_{L^{2}}\leq C_M \mu^{\alpha-1}(1+t)^{-\frac{3}{4}},\quad \norm{\psi_1}_{L^\infty}\leq C_M \mu^{\alpha-\frac{2}{3}}.
			\end{aligned}
		\end{align}
		where $C_M>0$ is a general constant depending on $M$ and independent of $\mu$. 
	\end{Thm}

	In order to prove  \cref{MT}, we first show
	\begin{Thm}\label{MT2}
		Under the same conditions as \cref{MT}, the system \eqref{anti},  \eqref{thetaA}-\eqref{equ-A} admits a unique global solution. Moreover,  it holds that for the zero modes, 
		\begin{align}\label{decay}
			\begin{aligned}
				&\norm{\tilde{V}_1,\tilde{V}_2}_{L^2}\leq C_M\mu^{\alpha}, \qquad \norm{\p_y^k\tilde{V}_1,\p_y^k\tilde{V}_2}_{L^2}\leq C_M\mu^{\alpha-\frac{k}{2}}(1+t)^{-\frac{k}{2}},\\ 
				&\norm{\theta_A}_{L^{\infty}}\leqslant C_M\mu^{\alpha},\qquad\norm{\mathcal{A}}_{L^2}\leq C_M\mu^{\alpha},	\qquad \norm{\mathcal{A}_y}_{L^2}\leq C_M\mu^{\alpha-\frac{1}{2}}(1+t)^{-\frac{1}{2}},
			\end{aligned}
		\end{align}
		where $k=1,2$. For non-zero modes, there is enhanced dissipation effect such that
		\begin{align}\label{enhanced dissipation}
			\norm{\nabla^{n+1} \phi_\neq}_{L^2}+\norm{\nabla^{n+1} \psi_\neq}_{L^2}\leqslant C_M\mu^{\alpha-\frac{2n+3}{6}} e^{-\frac{1}{128}\mu^{\frac{1}{3}}t}, \quad n=0,1.
		\end{align}	
	\end{Thm}
	
	\begin{Rem}
		
		Compared to the previous works for linearized case \cite{ADM2021,ZZZ2022},  the estimates are improved in \eqref{decay}. Indeed, the leading term of $\psim_1$ is $\theta_A$, which is shown uniformly bounded in $L^\infty$, and the time-decay rate of  $\psim_1-\theta_A\approx \mathcal{A}$ in $L^\infty$ is obtained in \eqref{decay}.
	\end{Rem}
	
	
%
	
Once we have Theorem \ref{MT2}, we are ready to prove \cref{MT}.	
\begin{proof}[\bf Proof of Theorem \ref{MT}]
	In terms of  \cref{phi-psi,tildev}, one has
	\begin{align*}
		&\norm{(\phi,\psi_2)}_{L^2}\leq\norm{(\phim,\psim_2)}_{L^2}+\norm{(\phi_\neq,\psi_{2\neq})}_{L^2}\\
        &\qquad\qquad\quad\lesssim\norm{(\tilde{v}_1,\tilde{v}_2)}_{L^2}+\norm{\mathcal{C}_i}_{L^2}+\norm{(\nabla\phi_\neq,\nabla\psi_{2\neq})}_{L^2},\\
		&\qquad\qquad\quad\lesssim\norm{(\p_y\tilde{V}_1,\p_y\tilde{V}_2)}_{L^2}+\norm{\theta_i}_{L^2}+\norm{\Xi_i}_{L^2}+\norm{(\nabla\phi_\neq,\nabla\psi_{2\neq})}_{L^2},\\
		&\norm{(\nabla\phi,\nabla\psi_2)}_{L^2}\leq\norm{(\p_y\phim,\p_y\psim_2)}_{L^2}+\norm{(\nabla\phi_\neq,\nabla\psi_{2\neq})}_{L^2}\\
		&\qquad\qquad\qquad\;\lesssim \norm{(\p_y^2\tilde{V}_1,\p_y^2\tilde{V}_2)}_{L^2}+\norm{\p_y\theta_i}_{L^2}+\norm{\p_y\Xi_i}_{L^2}+\norm{(\nabla\phi_\neq,\nabla\psi_{2\neq})}_{L^2}.
	\end{align*}
Acoording to \cref{MT2} and the estimations for the diffusion waves $\theta_i$ and $\Xi_i$ (see \cref{estimateontheta} and \ref{estimateonxi} below), we have
\begin{align*}
	\norm{\phi,\psi_2}_{L^{2}}\leq C_M \mu^{\alpha-\frac{1}{2}}(1+t)^{-\frac{1}{4}}, \quad \norm{\nabla\phi,\nabla\psi_2}_{L^{2}}\leq C_M \mu^{\alpha-1}(1+t)^{-\frac{3}{4}}.
\end{align*}

On the other hand, we obtain from Theorem \ref{MT2} and the Gagliardo-Nirenberg inequality that
\begin{align*}
&\norm{\psi_1}_{L^\infty}\leq\norm{\psim_1}_{L^\infty}+\norm{\psi_{1\neq}}_{L^\infty}\\
&\quad\lesssim \norm{\theta_A}_{L^\infty}+\norm{\mathcal{A}}_{L^\infty}+\norm{\tilde{V}_i}_{L^\infty}+\norm{\tilde{\Xi}_i}_{L^\infty}+\norm{\psi_{1\neq}}_{L^2}^{\frac{1}{2}}\norm{\nabla^2\psi_{1\neq}}_{L^2}^{\frac{1}{2}}\\
&\quad\lesssim \norm{\theta_A}_{L^\infty}+\norm{\mathcal{A}}_{L^\infty}+\norm{\tilde{V}_i}_{L^\infty}+\norm{\tilde{\Xi}_i}_{L^\infty}+\norm{\nabla\psi_{1\neq}}_{L^2}^{\frac{1}{2}}\norm{\nabla^2\psi_{1\neq}}_{L^2}^{\frac{1}{2}}\\
&\quad \lesssim C_M \mu^{\alpha-\frac{2}{3}}.
\end{align*}
Therefore we arrive at \cref{MT}.
\end{proof}
It remains to prove Theorem \ref{MT2}, which will be proved  by the Duhamel's principle for $\theta_A$ through the heat equation \eqref{thetaA} with source terms, and combining the local existence and a priori estimates in the following solution space
	\begin{align}
			X_{\bar{N}, N}&(T_0, T)=  \left\{(\phi, \psi, \tilde{V},\mathcal{A}):(\phi, \psi) \in C\left([T_0, T] ; H^5\right), \tilde{V} \in C\left([T_0, T] ; H^2\right),\right. \nonumber\\[1.5mm]
			&\mathcal{A}\in C\left([T_0, T] ; H^1\right),\quad \sup_{T_0 \leq t \leq T}\norm{\mathcal{A}}_{H^1}\leq N,\quad \p_y\mathcal{A}\in L^2\left(T_0, T ; H^1\right),\nonumber\\
			& \sup_{T_0 \leq t \leq T}\bigg[\norm{(\tilde{V}_1,\tilde{V}_2)}_{L^{2}}+\norm{(\phi,\psi_1,\psi_2)}_{H^{5}}\bigg]\leq N,\quad \inf _{x, t}(1+\phi) \geq \bar{N}, \\[1.5mm]
			&\p_y(\tilde{V}_2-\tilde{V}_1) ,\nabla^i\phi\in L^2\left(T_0, T ; L^2\right),  \left.\p_y(\tilde{V}_2+\tilde{V}_1),\nabla^i\psi \in L^2\left(T_0, T ; H^1\right),\quad i=1,..5.\right\}\nonumber
	\end{align}
	for $T_0,T, \bar{N}, N>0$. The local existence theorem is given as
	\begin{Thm}[Local existence theorem]\label{LET}
		Assume that
		\begin{align}\label{iniT}
			\norm{(\tilde{V}_1,\tilde{V}_2)(\cdot,T_0)}_{H^{2}}+\norm{(\phi,\psi_1,\psi_2)(\cdot,T_0)}_{H^{5}}\leq N,\quad \inf _{\Torus\times\R }(1+\phi)(\cdot,T_0) \geq \bar{N},
		\end{align}
		there exist positive constants $b$ and $t_0=t_0(\bar{N}, N)$ such that there exists a unique solution $(\phi, \psi, \tilde{V},\mathcal{A}) \in X_{\frac12\bar{N}, b N}\left(T_0, T_0+t_0\right)$ to \cref{perturbation1,anti,equ-A}.
		\end{Thm}
     The proof of \cref{LET} for small $t_0$ is standard and the details are omitted for easy reading.  Next, the a priori estimate will be carried out under the following a priori assumptions,
	\begin{align}\label{aps}
		\left\{\begin{aligned}
			&\sup_{T_0 \leq t \leq T_0+ t_0}\left\{\norm{\tilde{V}_1,\tilde{V}_2}_{L^2}+\mu^{\frac{1}{2}}(1+t)^{\frac{1}{2}}\norm{\tilde{v}_1,\tilde{v}_2}_{L^2}+\mu(1+t)\norm{\p_y\tilde{v}_1,\p_y\tilde{v}_2}_{L^2}\right\}\leq C_M \mu^{\alpha-\varepsilon},\\ 
			&\sup_{T_0 \leq t \leq T_0+t_0}\left\{\norm{\mathcal{A}}_{L^2}+ \mu^{\frac{1}{2}}(1+t)^{\frac{1}{2}}\norm{\mathcal{A}_y}_{L^2} \right\} \le C_M \mu^{\alpha-\varepsilon},  \\ 
			&\sup_{T_0 \leq t \leq T_0+t_0}e^{\frac{1}{128}\mu^{\frac{1}{3}}t}\bigg(\norm{\nabla^{n+1}\phi_{\neq}}_{L^2}+\norm{\nabla^{n+1}\psi_{\neq}}_{L^2}\bigg)\leq C_M \mu^{\alpha-\varepsilon-\frac{2n+3}{6}},\quad n=0,1,\\ 
			&\sup_{T_0 \leq t \leq T_0+t_0}\norm{\nabla^{k+1}\psi}_{L^2}+\norm{\nabla^{k+1}\phi}_{L^2}\leq 2C_M \mu^{\alpha-\frac{1}{2}-\frac{k}{2}},\quad  1\leq k\leq 4, 
	   \end{aligned}\right.
		\end{align}
	where the positive constant $\varepsilon$ can be arbitrarily small and the constant $C_M$ will be determined later. 
	\begin{Prop}[A priori estimates]\label{ape}
		Assume that $(\tilde{V}_1,\tilde{V}_2,\mathcal{A})$ is the unique solution for \cref{anti,equ-A} with the initial data \cref{iniT} in the interval $[T_0,T_0+t_0]$, satisfying the a priori assumptions \cref{aps}, it holds that
		\begin{align}\label{decaya}
			&\sup_{T_0 \leq t \leq T_0+t_0}\norm{\tilde{V}_1,\tilde{V}_2}_{L^2}\leq C_M \mu^{\alpha},\quad\sup_{T_0 \leq t \leq T_0+t_0}\norm{\p_y^k\tilde{v}_1,\p_y^k\tilde{v}_2}_{L^2}\leq C_M\mu^{\alpha}[\mu(1+t)]^{-\frac{k+1}{2}}, \nonumber\\
			&\sup_{T_0 \leq t \leq T_0+t_0}\norm{\mathcal{A}}_{L^2}\leq C_M\mu^{\alpha},	\qquad\sup_{T_0 \leq t \leq T_0+t_0}\norm{\mathcal{A}_y}_{L^2}\leq C_M\mu^{\alpha}[\mu(1+t)]^{-\frac{1}{2}},
		\end{align}
		where $k=0,1$. In addition,	 there has enhanced dissipation effect for non-zero mode, i.e., 
		\begin{align}\label{enhanced dissipationa}
			\sup_{T_0 \leq t \leq T_0+t_0}\left[\norm{\nabla^{n+1} \phi_\neq}_{L^2}+\norm{\nabla^{n+1} \psi_\neq}_{L^2}\right]\leq C_M\mu^{\alpha-\frac{2n+3}{6}} e^{-\frac{1}{128}\mu^{\frac{1}{3}}t}, \quad n=0,1.
		\end{align}	
		For higher-order derivatives, it holds that
		\begin{align}\label{ho}
			\sup_{T_0 \leq t \leq T_0+t_0}\norm{\nabla^{k+1}\psi}_{L^2}+\norm{\nabla^{k+1}\phi}_{L^2}\leq  (1+\frac{k}{10})C_M\mu^{\alpha-\frac{1}{2}-\frac{k}{2}},\quad  1\leq k\leq 4.
		\end{align}
	\end{Prop}
	\begin{Rem}
		Thanks \cref{decaya,enhanced dissipationa,ho}, one can close the a priori assumptions \cref{aps} as $\mu$ is small.
	\end{Rem}
	In terms of \cref{LET} and \cref{ape},   one can prove \cref{MT2} by the  continuity argument, cf. \cite{BGM2017ann}, together with the estimation for $\theta_A$ (see \cref{estimateonta} below). Thus, the rest of the present paper focuses on the proof of \cref{ape} and \cref{estimateonta}.

\vspace{0.3cm}
\section{Estimates on zero mode}
To estimate the zero mode, we firstly give the estimates of  $\theta_i$, $i=1,2$, which will be used frequently.
\begin{Lem}\label{estimateontheta}
For the initial data \eqref{initial},  it holds that for $i=1,2$, 
	\begin{align}
		\norm{\p_y^k \theta_i}_{L^p}\lesssim \mu^{\alpha}\left[\bar{\mu}(1+t)\right]^{-\frac{1}{2}+\frac{1}{2p}-\frac{k}{2}}, \quad  k\geq0,\ p\geq 1.
	\end{align}
\end{Lem}
\begin{proof}
Note that $\abs{\eta_i}=O(1)\mu^{\alpha}$, \cref{estimateontheta} can be directly derived by the formulas \cref{equ-theta,theta}. The proof is then omitted. 	
\end{proof}

\subsection{Estimates on coupled diffusion waves.}

It is known that the Green's function of the equation $v_t+\bar{c}v_x=\frac{\mu}{2}v_{xx} $ is  $$H(x,t;\bar{c},\mu)=\frac{1}{\sqrt{2\pi \mu t}}e^{-\frac{(x-\bar{c}t)^2}{2\mu t}}.$$
Let
$H(\bar{c},\mu):=H(y-s,t-\tau;\bar{c},\mu).$ 
By Duhamel's principle, one has from \eqref{equ-Xi} that 
\begin{align}\label{XI}
	\tilde{\Xi}_i(y, t)=-\int_0^t \int_{-\infty}^{\infty} H(\sigma_i,\frac{\bar{\mu}}{2}) \left(\theta_{i'}^2 / 2+\theta_1\Xi_1+\theta_2 \Xi_2+\frac{2-P''(1)}{2+P''(1)}(\theta_1\Xi_2+\theta_2\Xi_1)\right)(s, \tau) d s d \tau,
\end{align}
where $i'=3-i$. By \cref{XI}, we have
\begin{Lem}\label{estimateonxi}
	Under the same conditions of \cref{MT},  it holds that for $i=1,2$, 
	\begin{align}
		\norm{\tilde{\Xi}_i}_{L^\infty}\lesssim \mu^{2\alpha-\frac{1}{2}}(\bar{\mu}(1+t))^{-\frac{1}{4}},\quad \norm{\p_y^{k+1}\tilde{\Xi}_i}_{L^p}\lesssim\mu^{2\alpha}\left(\bar{\mu}(1+t)\right)^{-\frac{3}{4}+\frac{1}{2p}-\frac{k}{2}}, \quad  k\geq0,\ p\geq 1.
	\end{align}
\end{Lem}
\begin{proof}
	We only consider the case $i=1$ since the case $i=2$ can be treated similarly. 
	First, we introduce the following useful formula
	\begin{align}\label{semi-group}
		&-\frac{\left(s-\sigma(1+\tau)\right)^2}{4 \mu_1(1+\tau)}-\frac{\left(y-s-\sigma'(t-\tau)\right)^2}{4 \mu_2(t-\tau)}\nonumber \\
		= & -\frac{\mu_1(1+\tau)+\mu_2(t-\tau)}{4 \mu_1 \mu_2(1+\tau)(t-\tau)}\left[s-\sigma(1+\tau)-\frac{\mu_1(1+\tau)\left(y-\sigma'(t-\tau)-\sigma(1+\tau)\right)}{\mu_1(1+\tau)+\mu_2(t-\tau)}\right]^2\\
		& - \frac{\left[y-\sigma(1+\tau)-\sigma'(t-\tau)\right]^2}{4\left[\mu_1(1+\tau)+\mu_2(t-\tau)\right]},\nonumber
	\end{align}
	for given constants $\sigma,\ \sigma'$ and $\mu_1,\ \mu_2>0$.  Let
	\begin{align}
		\begin{aligned}
			\tilde{M}(t):=\sup _{0 \leq \tau \leq t, i=1,2} \left\{  \left\|\tilde{\Xi}_i(\cdot, \tau)\right\|_{L^{\infty}}\mu^{\frac{1}{2}}\left(\mu(1+t)\right)^{\frac{1}{4}}+\norm{\p_y^{k+1}\tilde{\Xi}_i}_{L^{p}}\left(\mu(1+\tau)\right)^{\frac{k}{2}+\frac{3}{4}-\frac{1}{2p}}\right\}, 
		\end{aligned}
	\end{align}
	then a direct computation on \eqref{XI} and \eqref{semi-group} yields
	\begin{align}\label{I123}
		\begin{aligned}
			\norm{\tilde{\Xi}_1}_{L^\infty}
			=&\norm{\int_0^{t} \int_{-\infty}^{\infty} H(\sigma_1,\frac{\bar{\mu}}{2})\left(\frac{\theta_{2}^2}{2}+\theta_1\Xi_1+\theta_2 \Xi_2+\frac{2-P''(1)}{2+P''(1)}(\theta_1\Xi_2+\theta_2\Xi_1)\right)(s, \tau) d s d \tau}_{L^\infty}
			\\
			\lesssim&\mu^{2\alpha}\norm{\int_0^{t}  \frac{1}{\sqrt{\bar{\mu}^2(1+t)(1+\tau)}}\exp\left\{-\frac{\left[\bar{y}_1-(\sigma_{2}-\sigma_{1})(1+\tau)\right]^2}{\bar{\mu}(1+t)}\right\} d \tau}_{L^{\infty}}\\
			&+\mu^{\alpha}\tilde{M}(t)\norm{\int_0^{t}  \frac{1}{\bar{\mu}^{\frac{5}{4}}(1+t)^{\frac{1}{2}}(1+\tau)^{\frac{3}{4}}}\exp\left\{-\frac{\left[\bar{y}_1-(\sigma_{2}-\sigma_{1})(1+\tau)\right]^2}{\bar{\mu}(1+t)}\right\} d \tau}_{L^{\infty}}\\
			&+\mu^{\alpha}\tilde{M}(t)\norm{\int_0^{t}  \frac{1}{\bar{\mu}^{\frac{5}{4}}(1+t)^{\frac{1}{2}}(1+\tau)^{\frac{3}{4}}}\exp\left\{-\frac{\bar{y}_1^2}{\bar{\mu}(1+t)}\right\} d \tau}_{L^{\infty}}\\
			\lesssim& \mu^{2\alpha}\norm{I_1}_{L^\infty}+\mu^{\alpha}\tilde{M}(t)(\norm{I_2}_{L^\infty}+\norm{I_3}_{L^\infty}),
		\end{aligned}
	\end{align}
	where $\bar{y}_1:=y-\sigma_1(1+t)$ and $\sigma_i=(-1)^i \bar{c}$. When $\abs{\bar{y}_1}\leq \sqrt{\bar{\mu}(1+t)}$, one has
	\begin{align}\label{I11}
		\begin{aligned}
			I_1 \lesssim \int_0^t \frac{1}{\bar{\mu}\sqrt{(1+t)(1+\tau)}} \exp \left\{-\frac{\bar{c}^2(1+\tau)^2}{ C\bar{\mu}(1+t)}\right\} d \tau \lesssim \bar{\mu}^{-\frac{1}{2}}\left[\bar{\mu}(1+t)\right]^{-\frac{1}{4}} \exp\left\{-\frac{\bar{y}_1^2}{\bar{\mu}(1+t)}\right\}.
		\end{aligned}
	\end{align}
We can also get the same estimate when $\bar{y}_1<-\sqrt{\bar{\mu}(1+t)}$. For $\bar{y}_1>\sqrt{\bar{\mu}(1+t)}$, we have
	\begin{align}\label{I12}
			I_1= & \int_0^{\frac{\bar{y}_1}{4 \bar{c}}-1}+\int_{\frac{\bar{y}_1}{4\bar{c}}-1}^t  \frac{1}{\bar{\mu}\sqrt{(1+t)(1+\tau)}} \exp \left\{-\frac{\left(\bar{y}_1-2\bar{c}(1+\tau)\right)^2}{C\bar{\mu}(1+t)}\right\} d \tau \\ \nonumber
			\leq & \int_0^{\frac{\bar{y}_1}{4 \bar{c}}-1} \frac{1}{\bar{\mu}\sqrt{(1+t)(1+\tau)}} \exp \left\{-\frac{\bar{y}_1^2 / 4+\left(\bar{y}_1 / 2-2\bar{c}(1+\tau)\right)^2}{C\bar{\mu}(1+t)}\right\} d \tau \\ \nonumber
			& +\int_{\frac{\bar{y}_1}{4 \bar{c}}-1}^t \frac{1}{\bar{\mu}\sqrt{(1+t)\left(\bar{y}_1 /\left(4\bar{c}\right)\right)}} \exp \left\{-\frac{\left(\bar{y}_1-2\bar{c}(1+\tau)\right)^2}{C\bar{\mu}(1+t)}\right\} d \tau \\ \nonumber
			\lesssim & \frac{1}{\bar{\mu}^{\frac{1}{2}}(\bar{\mu}(1+t))^{1 / 4}} \exp \left\{-\frac{\bar{y}_1^2}{4C \bar{\mu}(1+t)}\right\}+\frac{1}{\sqrt{\bar{\mu}\bar{y}_1}}.\nonumber
	\end{align}
	Thus, $\norm{I_1}_{L^\infty}\lesssim \bar{\mu}^{-\frac{3}{4}}(1+t)^{-\frac{1}{4}}.$ And $I_2,I_3$ can be directly bounded as $\norm{I_2,I_3}_{L^{\infty}}\lesssim \bar{\mu}^{-\frac{5}{4}}(1+t)^{-\frac{1}{4}}$. So we have
	\begin{align}\label{Xi_1}
		\norm{\tilde{\Xi}_1}_{L^\infty}\lesssim \mu^{2\alpha}\bar{\mu}^{-\frac{3}{4}}(1+t)^{-\frac{1}{4}}+\mu^{\alpha}\bar{\mu}^{-\frac{5}{4}}(1+t)^{-\frac{1}{4}}\tilde{M}(t).
	\end{align}
	For the higher order derivatives,  one has that for $k\geq0$,
	\begin{align}\label{pk1xi}
		\begin{aligned}
			\p^{k+1}_y\tilde{\Xi}_1
			=&-\int_0^{\frac{t}{2}} \int_{-\infty}^{\infty} \p^{k+1}_yH(\sigma_1,\frac{\bar{\mu}}{2})\left[\theta_1\Xi_1+\theta_2 \Xi_2+\frac{2-P''(1)}{2+P''(1)}(\theta_1\Xi_2+\theta_2\Xi_1)\right](s, \tau) d s d \tau\\
			&-\int_{\frac{t}{2}}^t \int_{-\infty}^{\infty} H(\sigma_1,\frac{\bar{\mu}}{2}) \p_s^{k+1}\left[\theta_1\Xi_1+\theta_2 \Xi_2+\frac{2-P''(1)}{2+P''(1)}(\theta_1\Xi_2+\theta_2\Xi_1)\right]d s d \tau\\
			&-\int_0^{t} \int_{-\infty}^{\infty} H(\sigma_1,\frac{\bar{\mu}}{2})\p^{k+1}_s\left(\frac{\theta_{2}^2}{2}\right)d s d \tau\\
			:=&I_4+I_5+I_6.
		\end{aligned}
	\end{align}
	Using \cref{theta,XI}, we get
	\begin{align}
		\norm{I_4+I_5}_{L^p}\lesssim \mu^{\alpha} \tilde{M}(t)\bar{\mu}^{-1}\left[\bar{\mu}(1+t)\right]^{-\frac{k}{2}-\frac{3}{4}+\frac{1}{2p}}.
	\end{align}
	The estimate on $I_6$ is a bit complicated.  Assume $t>4$ without loss of generality, then $t/2>\sqrt{t}.$
	\begin{align}
		\int_0^{t} \int_{-\infty}^{\infty} H(\sigma_1,\frac{\bar{\mu}}{2})\p^{k+1}_s\left(\frac{\theta_{2}^2}{2}\right)d s d \tau=\left(\int_0^{\sqrt{t}} +\int_{\sqrt{t}}^t\right) \int_{-\infty}^{\infty} H(\sigma_1,\frac{\bar{\mu}}{2})\p^{k+1}_s\left(\frac{\theta_{2}^2}{2}\right)d s d \tau=I_{6a}+I_{6b}.
	\end{align}
One has
	\begin{align}
		\begin{aligned}
			\norm{I_{6a}}_{L^p}\lesssim&\mu^{2\alpha}(\bar{\mu}(1+t))^{-\frac{k}{2}-1}\norm{\int_0^{\sqrt{t}}(\bar{\mu}(1+\tau))^{-\frac{1}{2}}e^{-\frac{\left[\bar{y}_1-(\sigma_2-\sigma_{1})(1+\tau)\right]^2}{(\bar{\mu}+\varepsilon)(1+t)}}d\tau}_{L^p}\\
			\lesssim&\mu^{2\alpha}\left[\bar{\mu}(1+t)\right]^{-\frac{k}{2}-\frac{3}{4}+\frac{1}{2p}}.
		\end{aligned}
	\end{align}
On the other hand, one can observe a cancellation in $I_{6b}$, i.e., 
	\begin{align}\label{cancellation}
			&\qquad \left(\sigma_1-\sigma_{2}\right)\int_{t^{1 / 2}}^t \int_{-\infty}^{\infty} H(\sigma_1,\frac{\bar{\mu}}{2}) \partial_s^{k+1} \left(\frac{\theta_{2}^2}{2}\right) d s d \tau \\ \nonumber
			& =- \int_{t^{1 / 2}}^t \int_{-\infty}^{\infty} H(\sigma_1,\frac{\bar{\mu}}{2})   \cdot\left[-\partial_\tau+\partial_\tau-\sigma_1 \partial_s+\sigma_{2} \partial_s+\frac{\bar{\mu}}{2} \partial_s^2-\frac{\bar{\mu}}{2} \partial_s^2\right] \p_s^k(\frac{\theta_{2}^2}{2}) d s d \tau \\ \nonumber
			& =\int_{t^{1 / 2}}^t \int_{-\infty}^{\infty} L_{\sigma_1} H(\sigma_1,\frac{\bar{\mu}}{2}) \p_s^k(\frac{\theta_{2}^2}{2}) - H(\sigma_1,\frac{\bar{\mu}}{2})\p_s^k L_{\sigma_2} (\frac{\theta_{2}^2}{2}) d s d\tau  \\ \nonumber
			& \quad- \int_{-\infty}^{\infty} H(y-s, t-t^{1/2},\sigma_1,\frac{\bar{\mu}}{2}) \p_s^k(\frac{\theta_{2}^2}{2})\left(s, t^{1 / 2}\right) d s +  \p_s^k(\theta_{2}^2)(y, t)\\ \nonumber
			&:=I_{6b0}+I_{6b1}+ I_{6b2} + \p_s^k(\theta_{2}^2)(s, t), \nonumber
	\end{align}
	where $L_{\sigma}:=\p_\tau+\sigma\p_s-\frac{\bar{\mu}}{2}\p_s^2$. Since $L_{\sigma_1}H(\sigma_1,\frac{\bar{\mu}}{2})=0$, we have $I_{6b0}^i=0$. By \cref{equ-theta}, one has
	\begin{align}
		L_{\sigma_{2}}(\theta_{2}^2)=O(1)\bar{\mu}\theta_{2}(\p_y\theta_{2})^2+O(1)\theta_{2}\p_y(\theta_{2})^2.
	\end{align}
	A direct calculation yields,
	\begin{align}
		\begin{aligned}
			\norm{I_{6b1}}_{L^p}\lesssim&\left[\bar{\mu}(1+t)\right]^{-\frac{k}{2}}\int_{\frac{t}{2}}^t \norm{L_{\sigma_{2}}(\theta_{2}^2)}_{L^1}\left[\mu(t-\tau)\right]^{-\frac{1}{2}+\frac{1}{2p}} d\tau\\
			&+\mu^{2\alpha}\left[\bar{\mu}(1+t)\right]^{-\frac{k}{2}-\frac{1}{2}}\norm{\int_{\sqrt{t}}^{\frac{t}{2}}\left[\bar{\mu}(\tau+1)\right]^{-\frac{3}{2}} e^{-\frac{(\bar{y}_1
						-(\sigma_{2}-\sigma_{1})(\tau+1))^2}{\bar{\mu}(t+1)}} d \tau}_{L^p}\\
			\lesssim &\mu^{2\alpha}\left[\bar{\mu}(1+t)\right]^{-\frac{k}{2}-\frac{3}{4}+\frac{1}{2p}},
		\end{aligned}
	\end{align}
	where we have used \cref{semi-group}. Also, we have  
	\begin{align}
		\begin{aligned}
			\norm{I_{6b2}}_{L^p} &\lesssim \mu^{2\alpha}\left[\bar{\mu}(t+1)\right]^{-\frac{k+1}{2}-\frac{1}{2}} \norm{\int_{-\infty}^{\infty} e^{-\frac{(y-s-\sigma_1(t-\sqrt{t}))^2}{\bar{\mu}(t-\sqrt{t})}-\frac{(s-\sigma_{2}(\sqrt{t}+1))^2}{\bar{\mu}(\sqrt{t}+1)}} d s}_{L^p} \\
			& \lesssim \mu^{2\alpha}\left[\bar{\mu}(t+1)\right]^{-\frac{k}{2}-\frac{3}{4}+\frac{1}{2p}},
		\end{aligned}
	\end{align}
	and $\norm{\p_s^k(\theta_{2}^2)(s, t)}\lesssim\mu^{2\alpha}\left[\bar{\mu}(t+1)\right]^{-\frac{k}{2}-\frac{3}{4}+\frac{1}{2p}}$. So we obtain that
	\begin{align}\label{Xi_2}
	\norm{\p^{k+1}_y\tilde{\Xi}_1}_{L^p}\lesssim\mu^{\alpha} \tilde{M}(t)\bar{\mu}^{-1}\left[\bar{\mu}(1+t)\right]^{-\frac{k}{2}-\frac{3}{4}+\frac{1}{2p}}+\mu^{2\alpha}\left[\bar{\mu}(t+1)\right]^{-\frac{k}{2}-\frac{3}{4}+\frac{1}{2p}}. 
	\end{align}
   By the same arguments, \eqref{Xi_2} holds for $\Xi_2$. Then combining \cref{Xi_1} and  \cref{Xi_2}, one has $\tilde{M}(t)\lesssim (\mu^{2\alpha}+\mu^{\alpha-1}\tilde{M}(t))$. Note that $\alpha>11/3$, one has $\tilde{M}(t)\lesssim \mu^{2\alpha}$ as $\mu$ is small. Substituting $\tilde{M}(t)\lesssim \mu^{2\alpha}$ into \eqref{Xi_1} and \eqref{Xi_2} yields Lemma \ref{estimateonxi}.
\end{proof}

\subsection{Energy Estimates}
This subsection is devoted to the energy estimates for $(\p_y^k \tilde{V}_1,\p_y^k\tilde{V}_2)$, where $k=0,1,2$. We first provide an estimate for $\norm{\tilde{V}_1(\cdot,t),\tilde{V}_2(\cdot,t)}$.
\begin{Lem}\label{1or}
	Under the same conditions as \cref{MT}, it holds that 
	\begin{align}
		\begin{aligned}
			&\norm{\tilde{V}_1(\cdot,t),\tilde{V}_2(\cdot,t)}^2+\bar{\mu}\int_0^T\norm{\p_y\tilde{V}_1,\p_y\tilde{V}_2}^2dt\lesssim \mu^{2\alpha}.
		\end{aligned}
	\end{align}
\end{Lem}
\begin{proof}
	Recalling \cref{tildeK,anti}, $\mathbf{\tilde{V}}$ satisfies
	\begin{align}\label{AA}
		\mathbf{\tilde{V}}_t+A\mathbf{\tilde{V}}_y=B\mathbf{\tilde{V}}_{yy}+\mathbf{\tilde{K}},
	\end{align}
	where
	\begin{align}
		A=\left(\begin{array}{cc}-\bar{c}& 0 \\ 0 & \bar{c}\end{array}\right), \quad {B}=\left(\begin{array}{cc}\frac{\bar{\mu}}{2} & \frac{\bar{\mu}}{2} \\ \frac{\bar{\mu}}{2} & \frac{\bar{\mu}}{2}\end{array}\right),\quad \mathbf{\tilde{K}}=\left(\begin{array}{c}\tilde{K}_1\\ \tilde{K}_2\end{array}\right).
	\end{align}
	To capture the viscous effect, we denote $(\mathbb{V}_1,\mathbb{V}_2):=(\tilde{V}_2-\tilde{V}_1,\tilde{V}_2+\tilde{V}_1)$, then one has
	\begin{align}\label{hatv}
		\begin{cases}
			\mathbb{V}_{1t}+\bar{c}\mathbb{V}_{2y}=\tilde{K}_{2}-\tilde{K}_{1},\\
			\mathbb{V}_{2t}+\bar{c}\mathbb{V}_{1y}=\bar{\mu}\mathbb{V}_{2yy}+\tilde{K}_{1}+\tilde{K}_{2}.
		\end{cases}
	\end{align}
	From \cref{tildeK}, one has
	\begin{align}\label{K2-K1}
			\tilde{K}_2-\tilde{K}_1=&\frac{\bar{\mu}}{2\bar{c}}\p_y\left[\theta_1^2 / 2+\theta_2^2 / 2+\theta_1\Xi_1+\theta_2 \Xi_2+\frac{2-P''(1)}{2+P''(1)}(\theta_1\Xi_2+\theta_2\Xi_1) \right]\\ \nonumber
			&-\frac{\bar{\mu}^2}{8\bar{c}}\p_y^2(\theta_1+\theta_2+\Xi_1+\Xi_2),
		\end{align}
 which implies from \cref{estimateontheta} and \cref{estimateonxi}, that for $k\geq0$, 
	\begin{align}\label{estimateon2-1}
		\norm{\p_y^k\left(\tilde{K}_2-\tilde{K}_1\right)}_{L^1}\lesssim\mu^{\alpha+2}\left[\bar{\mu}(1+t)\right]^{-1-\frac{k}{2}},\quad\norm{\p_y^{k}\left(\tilde{K}_2-\tilde{K}_1\right)}&\lesssim \bar{\mu}^{\alpha+2}\left[\bar{\mu}(1+t)\right]^{-\frac{5}{4}-\frac{k}{2}}.
	\end{align}
	The estimate of $\tilde{K}_1+\tilde{K}_2$ is complicated.  The formulas \cref{tildeK} and \eqref{E}  give that 
	\begin{align}\label{K2+K1}
		\abs{\tilde{K}_1+\tilde{K}_2-\frac{E_2}{a\bar{c}}}=O(1)\abs{\tilde{E}^{(1)}+\tilde{E}^{(2)}+\tilde{N}},
	\end{align}
where 	$E_2=-(2\mu+\lambda)\p_y\left(\frac{\varphim_2\phim}{\rhom}\right)$, 
	\begin{align}\label{E2}
			\frac{\varphim_2\phim}{\rhom}&=a^2 \bar{c}\frac{(\tilde{V}_{1y}+\tilde{V}_{2y}+\mathcal{C}_1+\mathcal{C}_2)(\tilde{V}_{2y}-\tilde{V}_{1y}+\mathcal{C}_2-\mathcal{C}_1)}{\tilde{V}_{2y}-\tilde{V}_{1y}+\mathcal{C}_2-\mathcal{C}_1+1}\\ \nonumber
			&=O(1)\sum_{i,j=1}^2 \left(\abs{\tilde{V}_{iy}\tilde{V}_{jy}}+ \abs{\tilde{V}_{iy}\theta_i}+\abs{\theta_i\theta_j}\right).
	\end{align}
	and
	\begin{align}\label{e1}
    \abs{\tilde{E}^{(1)}}=&O(1)\Bigg|\sum_{i=1}^2\left[\mathcal{C}_i^2-\theta_i^2-2\left(\theta_1\Xi_1+\theta_2 \Xi_2\right)\right.\\\notag
		&\left.\qquad\qquad+\frac{2-P''(1)}{2(2+P''(1))}\mathcal{C}_{i}\mathcal{C}_{i'}-\frac{2-P''(1)}{2+P''(1)}(\theta_1\Xi_2+\theta_2\Xi_1)\right]\\\notag
		&+\frac{\bar{\mu}}{4\bar{c}}\left(\theta_1^2 / 2+\theta_2^2 / 2+(\theta_1 +\theta_2)(\Xi_1+ \Xi_2)+\frac{\bar{\mu}}{2}\theta_{iy}+\frac{\bar{\mu}}{2}\Xi_{iy}\right)_y\Bigg|,\\\label{ee}
		\abs{\tilde{E}^{(2)}}=&O(1)\sum_{i,j=1}^2|\tilde{V}_{iy}\tilde{V}_{jy}+\tilde{V}_{iy}C_i|=O(1)\sum_{i,j=1}^2\left(|\tilde{V}_{iy}\tilde{V}_{jy}|+|\tilde{V}_{iy}\theta_{j}|+|\tilde{V}_{iy}\Xi_j|\right),\\\label{e2}
		\abs{\tilde{N}}=&O(1)\abs{\Do(\phi_\neq^2+\phi_\neq\psi_{2\neq})}+O(1)\abs{\p_y\Do(\phi_\neq^2+\phi_\neq\psi_{2\neq})}.
	\end{align}
    From \eqref{tildev}, the worst terms in $\tilde{E}^{(1)}$ are $ \theta_i \cdot \p_y \theta_i$ and  $\Xi_1 \cdot \Xi_2$.
	By \cref{theta},  \cref{estimateontheta}, \cref{estimateonxi} and the a priori assumption \cref{aps}, one has
	\begin{align}
		&\norm{\p_y^{k}\tilde{E}^{(1)}}_{L^1}\lesssim {\mu}^{2\alpha+1}\left[\bar{\mu}(1+t)\right]^{-1-\frac{k}{2}},\qquad \norm{\p_y^k\tilde{N}}_{L^1}\lesssim {\mu}^{2\alpha-\frac{k}{3}-1-\varepsilon}e^{-\frac{1}{64}\mu^{\frac{1}{3}}t},\label{ENL1}\\
		&\norm{\p_y^k\tilde{E}^{(1)}}\lesssim{\mu}^{2\alpha+1}\left[\bar{\mu}(1+t)\right]^{-\frac{5}{4}-\frac{k}{2}},\qquad\quad \norm{\p_y^k\tilde{N}}\lesssim{\mu}^{2\alpha-\frac{k}{3}-1-\varepsilon}e^{-\frac{1}{64}\mu^{\frac{1}{3}}t}.\label{ENL2}
	\end{align}
	\begin{flushleft}
		\textbf{Basic estimate}
	\end{flushleft}
	Multiplying \cref{hatv}$_1$ by $\mathbb{V}_1$ and  \cref{hatv}$_2$ by $\mathbb{V}_2$, adding them up and  integrating the resulting equations on $(0,t)\times\R$, one has
	\begin{align}
		\begin{aligned}\nonumber
			\norm{\mathbb{V}_1,\mathbb{V}_2}^2+\bar{\mu}	\int_0^T\norm{\mathbb{V}_{2y}}^2dt&\lesssim \norm{\mathbb{V}_1(\cdot,0),\mathbb{V}_2(\cdot,0)}^2+ \abs{\int_0^T\int_{\R}(\tilde{K}_2-\tilde{K}_1)\mathbb{V}_1+(\tilde{K}_1+\tilde{K}_2)\mathbb{V}_2dydt}.
		\end{aligned}
	\end{align} 
By \cref{aps,estimateon2-1}, we have
\begin{align}
&	\int_0^T \int_\R \abs{(\tilde{K}_2-\tilde{K}_1)\mathbb{V}_1} dy dt
	\lesssim\int_0^T\norm{\tilde{K}_2-\tilde{K}_1}_{L^1}\norm{\mathbb{V}_1}^{\frac{1}{2}}\norm{\p_y\mathbb{V}_1}^{\frac{1}{2}}dt\\ \nonumber
	\lesssim& \int_0^T\bar{\mu}^{-1-\varepsilon}\norm{\tilde{K}_2-\tilde{K}_1}^{\frac{4}{3}}_{L^1}\norm{\mathbb{V}_1}^{\frac{2}{3}}+\delta\bar{\mu}\norm{\p_y\mathbb{V}_1}^{2}dt
	 \lesssim \mu^{2\alpha+\frac{1}{3}-\frac53\varepsilon}+\int_0^T\delta\bar{\mu}\norm{\p_y\mathbb{V}_1}^{2}dt.
\end{align}
Recall \cref{K2+K1},  $\abs{(\tilde{K}_1+\tilde{K}_2-\frac{E_2}{a\bar{c}})\mathbb{V}_2}\lesssim\abs{\left(\tilde{E}^{(1)}+\tilde{E}^{(2)}+\tilde{N}\right)\mathbb{V}_2} $. From \eqref{ENL1} and \eqref{ENL2}, we have
	\begin{align}\label{a}
		&	\int_0^T\int_R \abs{(\tilde{E}^{(1)}+\tilde{N})\mathbb{V}_2}dydt \lesssim\int_0^T\norm{\tilde{E}^{(1)}+\tilde{N}}_{L^1}\norm{\mathbb{V}_2}^{\frac{1}{2}}\norm{\p_y\mathbb{V}_2}^{\frac{1}{2}}dt\\ \nonumber
			\lesssim& \int_0^T\bar{\mu}^{-1-\varepsilon}\norm{\tilde{E}^{(1)}+\tilde{N}}^{\frac{4}{3}}_{L^1}\norm{\mathbb{V}_2}^{\frac{2}{3}}+\delta\bar{\mu}\norm{\p_y\mathbb{V}_2}^{2}dt\lesssim \mu^{\frac{10}{3}\alpha-\frac{8}{3}-\varepsilon}+\int_0^T\delta\bar{\mu}\norm{\p_y\mathbb{V}_2}^{2}dt.
	\end{align}
   By the a priori assumption \cref{aps} and \eqref{ee}, we have
    \begin{align}\label{b}
			&\int_0^T\int_{\R}|\tilde{E}^{(2)}\cdot\mathbb{V}_2|dydt\lesssim \int_0^T\delta\bar{\mu} \norm{(\p_y\mathbb{V}_1,\p_y\mathbb{V}_2)}^2dt\\ 
			&\qquad+{\mu}^{2\alpha-1}\int_0^T\int_{\R}\left[\bar{\mu}(1+t)\right]^{-1}e^{\frac{-(y\pm\bar{c}(1+t))^2}{\bar{\mu}(1+t)}}\mathbb{V}_2^2 dydt+{\mu}^{4\alpha-\varepsilon-1}, \nonumber
    \end{align}
where we have used the fact that $$\sum_{i,j=1}^2|\tilde{V}_{iy}\theta_j|\lesssim  \sum_{i,j=1}^2 \mu^{\alpha}(\bar{\mu}(1+t))^{-\frac{1}{2}}e^{-\frac{\abs{y-(-1)^j\bar{c}(1+t)}^2}{\bar{\mu}(1+t)}}\abs{\tilde{V}_{iy}}$$ due to \cref{theta} and $\abs{\eta_i}=O(1)\mu^{\alpha}$.
	Also, we have
	\begin{align}\label{c}
		\begin{aligned}
			\abs{\int_0^T\int_\R E_2\cdot\mathbb{V}_2dydt}\lesssim& \sum_{i,j,k=1,2}\int_0^T\int_{\R}|\mathbb{V}_{iy}\mathbb{V}_{jy}\mathbb{V}_{ky}|+|\theta_{i}\theta_{j}\mathbb{V}_{ky}|+|\theta_{i}\mathbb{V}_{jy}\mathbb{V}_{ky}|dydt\\
			\lesssim& \delta\bar{\mu}\int_0^T\norm{\mathbb{V}_{1y},\mathbb{V}_{2y}}^2dt+\mu^{4\alpha-3-\varepsilon},
		\end{aligned}
	\end{align}
	where we have used the a priori assumption \cref{aps}. 
	Then it holds that 
	\begin{align}\label{v1v2}
		\begin{aligned}
		\norm{\mathbb{V}_1,\mathbb{V}_2}^2+\bar{\mu}	\int_0^T\norm{\mathbb{V}_{2y}}^2dt
		\lesssim& \norm{(\mathbb{V}_1,\mathbb{V}_2)(\cdot,0)}^2+\delta\bar{\mu}\int_0^T \norm{\p_y\mathbb{V}_1}^2dt+\mu^{2\alpha+\frac{1}{3}-\varepsilon}\\
		&+\mu^{2\alpha-1}\int_0^T\int_{\R} \left[\bar{\mu}(1+t)\right]^{-1}e^{\frac{-(y\pm\bar{c}(1+t))^2}{\bar{\mu}(1+t)}}\mathbb{V}_2^2 dydt,
	\end{aligned}
	\end{align} 
since
 $\frac{10}{3}\alpha-\frac{8}{3}>2\alpha+\frac{1}{3}, 4\alpha-3>2\alpha+\frac{1}{3}$ for $\alpha>\frac{11}{3}$.
	\begin{flushleft}
		\textbf{Estimate on $\norm{\mathbb{V}_{1y}}^2$}
	\end{flushleft}

	Multiplying \cref{hatv}$_2$ by $\mathbb{V}_{1y}$, one has
	\begin{align}
	\bar{c}\mathbb{V}_{1y}^2+(\mathbb{V}_2\mathbb{V}_{1y})_t+\mathbb{V}_{2y}\mathbb{V}_{1t}-\bar{\mu}\mathbb{V}_{2yy}\mathbb{V}_{1y}=(\tilde{K}_1+\tilde{K}_2)\mathbb{V}_{1y}+(\cdots)_y,
\end{align}	
which implies from \cref{hatv}$_1$ (i.e. $\mathbb{V}_{1t}=-\bar{c}\mathbb{V}_{2y}+\tilde{K}_2-\tilde{K}_1, \mathbb{V}_{2yy}=-\frac{\mathbb{V}_{1yt}}{\bar{c}}+\frac{1}{\bar{c}}(\tilde{K}_2-\tilde{K}_1)_y$) that 
	\begin{align}\label{V1y0}
&	\bar{c}\mathbb{V}_{1y}^2+(\mathbb{V}_2\mathbb{V}_{1y}+\frac{\bar{\mu}}{2\bar{c}}\mathbb{V}_{1y}^2)_t \notag\\
=&\bar{c}\mathbb{V}_{2y}^2+(\tilde{K}_1-\tilde{K}_2)\mathbb{V}_{2y}+\frac{\bar{\mu}}{\bar{c}}(\tilde{K}_2-\tilde{K}_1)_y\mathbb{V}_{1y}+(\tilde{K}_1+\tilde{K}_2)\mathbb{V}_{1y}+(\cdots)_y.
\end{align}	
Integrating \cref{V1y0} on $\R\times [0,T]$ gives that 
		\begin{align}\label{V1y}
			&\int_{\R}\mathbb{V}_2\mathbb{V}_{1y}+\frac{\bar{\mu}}{2\bar{c}}\mathbb{V}_{1y}^2dy+	\int_0^T\bar{c}\norm{\mathbb{V}_{1y}}^2dt\\ \nonumber
			\lesssim&\int_0^T \norm{\tilde{K}_{2}-\tilde{K}_{1}}^2 dt+\bar{\mu}\int_0^T \norm{\tilde{K}_{2y}-\tilde{K}_{1y}}^2 dt+\int_0^T\bar{\mu}\norm{\mathbb{V}_{1y}}^2+\norm{\mathbb{V}_{2y}}^2dt\\\nonumber
			&+\abs{\int_{0}^T\int_{\R}\left(\tilde{K}_{1}+\tilde{K}_{2}\right)\mathbb{V}_{1y}dydt}+\norm{(\tilde{V}_{1y},\tilde{V}_{2})(\cdot,0)}^2.\nonumber
	\end{align}
Using \cref{ENL1}, \cref{ENL2} and the a priori assumption \cref{aps}, we have
	\begin{align}
	&\begin{aligned}
		\int_0^T \norm{\tilde{K}_2-\tilde{K}_1}^2 + \bar{\mu}  \norm{\tilde{K}_{2y}-\tilde{K}_{1y}}^2 dt \lesssim \mu^{2\alpha+\frac{3}{2}},
	\end{aligned}\\
	&\begin{aligned}\label{nmnm}
		\int_0^T \int_\R \abs{(\tilde{E}^{(1)}+\tilde{N})\mathbb{V}_{1y}}dy \lesssim \delta \int_0^T\norm{\mathbb{V}_{1y}}^2 dt + \int_0^T \norm{\tilde{E}^{(1)},\tilde{N}}^2dt \lesssim \delta \int_0^T\norm{\mathbb{V}_{1y}}^2 dt+\mu^{4\alpha-\frac{7}{3}-\varepsilon},
	\end{aligned}\\
	&\begin{aligned}
		\int_0^T\int_{\R}|\tilde{E}^{(2)}\mathbb{V}_{1y}|dydt\lesssim& \sum_{i.j=1,2} \int_0^T\int_{\R}|\mathbb{V}_{iy}\mathbb{V}_{jy}\mathbb{V}_{1y}|+|\mathcal{C}_{i}\mathbb{V}_{jy}\mathbb{V}_{1y}|dydt
		\lesssim \delta\bar{\mu}\int_0^T\norm{\mathbb{V}_{1y},\mathbb{V}_{2y}}^2dt,
	\end{aligned}\\
	&\begin{aligned}
		\bigg|\int_0^T\int_{\R} {E}_2\mathbb{V}_{1y}dydt \bigg|\lesssim&\sum_{i,j=1,2} \int_0^T\int_{\R}(|\mathbb{V}_{iy}\mathbb{V}_{jy}|+\abs{\theta_{i}\theta_{j}}+|\theta_{i}\mathbb{V}_{jy}|)|\mathbb{V}_{1yy}|dydt\\
		\lesssim&\sum_{i,j=1}^2\int_0^T\norm{\mathbb{V}_{iy}}^2\norm{\mathbb{V}_{1yy}}_{L^\infty}+\delta\mu\norm{\mathbb{V}_{iy}}_{L^2}^2+\mu^{-1-\epsilon}\norm{\theta_i}_{L^\infty}^2\norm{\mathbb{V}_{1yy}}_{L^2}^2\\
		&\qquad\quad\;\;+\norm{\theta_i}_{L^\infty}\norm{\theta_j}_{L^2}\norm{\mathbb{V}_{1yy}}_{L^2} dt\\
		\lesssim&\delta\bar{\mu}\int_0^T\norm{\mathbb{V}_{1y},\mathbb{V}_{2y}}^2dt+\mu^{3\alpha-\frac{7}{4}-\varepsilon}.
	\end{aligned}
\end{align}
	Then we arrive at
	\begin{align}\label{V1y1}
		\begin{aligned}
			&\int_{\R}\mathbb{V}_2\mathbb{V}_{1y}+\frac{\bar{\mu}}{2\bar{c}}\mathbb{V}_{1y}^2dy+	\int_0^T\bar{c}\norm{\mathbb{V}_{1y}}^2dt
			\lesssim{\mu}^{2\alpha}+\int_0^T \norm{\mathbb{V}_{2y}}^2 dt.
		\end{aligned}
	\end{align}
	\begin{flushleft}
	\textbf{Estimate on $	\int_0^T\int_{\R}\left[\bar{\mu}(1+t)\right]^{-1}e^{-\frac{(y\pm\bar{c}(1+t))^2}{\bar{\mu}(1+t)}}\mathbb{V}_2^2dydt$}
\end{flushleft}

	We only consider the case of $-\bar{c}$ since the case $\bar{c}$ can be treated similarly. Note that $\mathbb{V}_2=\tilde{V}_1+\tilde{V}_2$, we go back the system \eqref{AA} and estimate $\int_0^T\int_{\R}[\bar{\mu}(1+t)]^{-1}e^{-\frac{(y-\bar{c}(1+t))^2}{\bar{\mu}(1+t)}}(\tilde{V}_1^2+\tilde{V}_2^2)dydt$. From $\eqref{AA}_1$, $\tilde{V}_1$ propagates backward with speed $-\bar{c}$, while the weight function $
	e^{-\frac{(y-\bar{c}(1+t))^2}{\bar{\mu}(1+t)}}$ travels forward with speed $\bar{c}$, we can expect stronger estimate for $\tilde{V}_1$ than $\int_0^T\int_{\R}[\bar{\mu}(1+t)]^{-1}e^{-\frac{(y-\bar{c}(1+t))^2}{\bar{\mu}(1+t)}}\tilde{V}_1^2dydt$. Indeed, 
	set
	\begin{align}\label{q}
		\eta_1=\exp \left(\int_{-\infty}^{y} h(z, t) d z\right), \quad h(t,z)=\frac{1}{\sqrt{2 \pi \bar{\mu}(1+t)}} \exp \left(-\frac{\left(z-\bar{c}(1+t)\right)^2}{2 \bar{\mu}(1+t)}\right),
	\end{align} where $h$ satisfies
	$$
	h_t+\bar{c}h_y=\frac{\bar{\mu}}{2} \p_y^2h,\qquad
	1 \leq \eta_1 \leq e,
	$$
	and
	$$
	\p_t\eta_{1}=\eta_1 \int_{-\infty}^{y} h_t(z, t) d z=\eta_1\left(\frac{\bar{\mu}}{2} \p_yh-\bar{c} h\right), \quad \p_y\eta_{1}=\eta_1 h .
	$$
	Multiplying \cref{AA}$_1$ by $\eta_1 \tilde{V}_1$, we can get
	$$
	\begin{aligned}
		& \p_t\left(\eta_1 \frac{\tilde{V}_1^2}{2}\right)-\left(\p_t\eta_{1}-\bar{c} \p_y\eta_{1}\right) \frac{\tilde{V}_1^2}{2} =\frac{\bar{\mu}}{2}\p_y^2 \tilde{V}_{1} \eta_1 \tilde{V}_1+\frac{\bar{\mu}}{2}\p_y^2\tilde{V}_2\eta_1\tilde{V}_1+\tilde{K}_1 \eta_1 \tilde{V}_1+(\cdots)_y .
	\end{aligned}
	$$
	Note that
	$$
	\begin{aligned}
		\p_t\eta_{1}-\bar{c}\p_y \eta_{1} & =-2\bar{c} \eta_1 h+\eta_1 \frac{\bar{\mu}}{2} \p_yh,
	\end{aligned}
	$$
	we can get
	\begin{align}\label{qq}
		\begin{aligned}
			& \left(\eta_1 \frac{\tilde{V}_1^2}{2}\right)_t+2\bar{c}\eta_1 h \tilde{V}_1^2=\eta_1 \bar{\mu} \p_yh\frac{\tilde{V}_1^2}{4}+\frac{\bar{\mu}}{2}\p_y^2 \tilde{V}_{1} \eta_1 \tilde{V}_1+\frac{\bar{\mu}}{2}\p_y^2\tilde{V}_2\eta_1\tilde{V}_1+\tilde{K}_1 \eta_1 \tilde{V}_1+(\cdots)_y.
		\end{aligned}
	\end{align}
Since $\p_yh\approx (1+t)^{-1}$ and $h\approx (1+t)^{-\frac12}$, $\bar{\mu}\p_yh \tilde{V}_1^2$ could be controlled by $h\tilde{V}_1^2$. 
Also, we have 
	\begin{align}\label{qqq}
		\begin{aligned}
&\abs{\int_\R \eta_1(\p_y^2\tilde{V}_1+\p_y^2\tilde{V}_2)\tilde{V}_1dy}\lesssim \norm{\left(\p_y\tilde{V}_1,\p_y\tilde{V}_2\right)}^2+ \norm{h\tilde{V}_1}^2.
		\end{aligned}
	\end{align}
Similar to \cref{a}-\cref{c}, one has
	\begin{align}\label{22222}
		\bigg|\int_0^T\int_{\R}\tilde{K}_1\tilde{V}_1\eta_1dydt\bigg|\lesssim &\mu^{2\alpha-1}\int_0^T\int_{\R}\left[\bar{\mu}(1+t)\right]^{-1}e^{-\frac{(y\pm\bar{c}(1+t))^2}{\bar{\mu}(1+t)}}\tilde{V}_1^2 dydt\nonumber\\
		&+\bar{\mu}\int_0^T \bigg(\norm{\p_y\tilde{V}_1}^2+\norm{\p_y\tilde{V}_2}^2\bigg)dt+\mu^{\frac{10}{3}\alpha-\frac{8}{3}-\varepsilon}.
	\end{align}
Thus, integrating \cref{qq} on $\R\times [0,T]$ gives that
	\begin{align}\label{new1}
	\norm{\tilde{V}_1(\cdot,T)}^2+	\int_0^T\int_{\R}h\tilde{V}_1^2dydt\lesssim& \norm{\tilde{V}_1(\cdot,0)}^2+\bar{\mu}\int_0^T\norm{\left(\p_y\tilde{V}_1,\p_y\tilde{V}_2\right)}^2dt\\ \nonumber
        &+\mu^{2\alpha-1}\int_0^T\int_{\R}\left[\bar{\mu}(1+t)\right]^{-1}e^{-\frac{(y+\bar{c}(1+t))^2}{\bar{\mu}(1+t)}}\tilde{V}_1^2 dydt +\mu^{\frac{10}{3}\alpha-\frac{8}{3}-\varepsilon}.
	\end{align}
From the formula \eqref{q} of $h$, we exactly obtain an estimate of $\int_0^T\int_{\R}\left[\bar{\mu}(1+t)\right]^{-\frac12}e^{-\frac{(y-\bar{c}(1+t))^2}{\bar{\mu}(1+t)}}\tilde{V}_1^2dydt$.

The estimate of $\int_0^T\int_{\R}\left[\bar{\mu}(1+t)\right]^{-1}e^{-\frac{(y-\bar{c}(1+t))^2}{\bar{\mu}(1+t)}}\tilde{V}_2^2dydt$ is subtle since both $\tilde{V}_2$ and the weight function $
e^{-\frac{(y-\bar{c}(1+t))^2}{\bar{\mu}(1+t)}}$ propatate forward with the same speed $\bar{c}$. 
Setting
	$$
	n(y, t)=\int_{-\infty}^y h(z, t) d z,
	$$
	it is easy to check that
	\begin{align}\label{0000000}
	\p_tn=\frac{\bar{\mu}}{2} \p_yh-\bar{c} h, \quad 0<n<1 .
	\end{align}
	We will use a weighted inequality in Huang-Li-Matsumura \cite{HLM2010} to estimate the above-weighted integral. Following Lemma 1 in \cite{HLM2010}, we have
	\begin{Lem}\label{Ptii} 
		For any $T>0$, if $V(t, y)$ satisfies
		$$
		V(t,y) \in L^{\infty}\left(0, T ; L^2({\R})\right), \quad \p_yV \in L^2\left(0, T ; L^2({\R})\right), \quad \p_tV \in L^2\left(0, T ; H^{-1}({\R})\right),
		$$
		it holds that 
		$$
		\begin{aligned}
			\frac{1}{2}\int_0^T \int_{\R} h^2 V^2 d y d t \leq  & \frac{1}{\bar{\mu}}\int_{\R} V^2(y, 0) d y+2 \int_0^T\left\|\p_yV\right\|^2 d t \\
			& +\frac{2}{\bar{\mu}}\left(\int_0^T<\p_t V, V n^2>d t-\bar{c} \int_0^T \int_{\R} nhV^2 d y d t\right).
		\end{aligned}
		$$
	\end{Lem}
	\begin{proof}
		From \eqref{0000000}, it is strightforward to check that
		\begin{align*}
			&\frac{2}{\bar{\mu}}\left(\int_0^T<\p_t V, V n^2>d t-\bar{c} \int_0^T \int_{\R} nhV^2  d y d t\right)\\
        =&\frac{1}{\bar{\mu}}\left(\int_\R (Vn)^2(y,T)dy-\int_\R (Vn)^2(y,0)dy\right)-\int_0^T \int_\R \p_yhnV^2 dydt\\
		\geq& -\frac{1}{\bar{\mu}}\int_\R (Vn)^2(y,0)dy +2 \int_0^T \int_\R hnV V_y dydt + \int_0^T \int_\R V^2 h^2 dydt\\
		\geq & -\frac{1}{\bar{\mu}}\int_\R (Vn)^2(y,0)dy - 2 \int_0^T \norm{V_y}^2 dt + \frac{1}{2}\int_0^T \int_\R V^2h^2dydt,
		\end{align*}
which completes the proof of \cref{Ptii}. 
	\end{proof}
	Multiplying \cref{AA}$_2$ by $n^2\tilde{V}_2$, we get
	\begin{align}\label{new11}
	&	\int_0^T<  \p_ t \tilde{V}_{2}, \tilde{V}_2 n^2>d t-\bar{c} \int_0^T \int_{\R} \tilde{V}_2^2 n h d y d t \nonumber\\
		=&\int_0^T \int_{\R}  \frac{\bar{\mu}}{2}\p_y^2\tilde{V}_2\tilde{V}_2n^2+\frac{\bar{\mu}}{2}\p_y^2\tilde{V}_1\tilde{V}_2n^2+\tilde{K}_2\tilde{V}_2n^2d y d t\nonumber\\
		\lesssim&\bar{\mu}\int_0^T\norm{\p_y\tilde{V}_1,\p_y\tilde{V}_2}^2dt+\delta\bar{\mu}\int_0^T\int_{\R} h^2\tilde{V}_2^2dydt+\int_0^T\int_{\R}\tilde{K}_2\tilde{V}_2n^2dydt.
	\end{align}
	Simialr to \cref{a}-\cref{c}, one has
	\begin{align}\label{new12}
		\int_0^T\int_{\R}\tilde{K}_2\tilde{V}_2n^2dydt\lesssim &\mu^{2\alpha-1}\int_0^T\int_{\R}\left[\bar{\mu}(1+t)\right]^{-1}e^{-\frac{(y\pm\bar{c}(1+t))^2}{\bar{\mu}(1+t)}}\tilde{V}_2^2 dydt\nonumber\\
		&+\bar{\mu}\int_0^T \norm{\p_y\tilde{V}_1}^2+\norm{\p_y\tilde{V}_2}^2dt+\mu^{\frac{10}{3}\alpha-\frac{8}{3}-\varepsilon}.
	\end{align}
	Combining  \cref{Ptii}, \eqref{new11} and \eqref{new12}, one has
	\begin{align}\label{h2V}
			\int_0^T \int_{\R} h^2 \tilde{V}^2_{2} d y d t\lesssim& \bar{\mu}^{-1}\norm{\tilde{V}_2(\cdot,0)}^2+\int_0^T\norm{\left(\p_y\tilde{V}_1,\p_y\tilde{V}_2\right)}^2dt+\mu^{\frac{10}{3}\alpha-\frac{11}{3}-\varepsilon}\\\nonumber
		&+\mu^{2\alpha-2}\int_0^T\int_{\R}\left[\bar{\mu}(1+t)\right]^{-1}e^{-\frac{(y+\bar{c}(1+t))^2}{\bar{\mu}(1+t)}} \tilde{V}_2^2 dydt.
	\end{align}
Therefore, from \eqref{new1} and \eqref{h2V}, we can obtain that 
	\begin{align}\label{new13}
		\int_0^T\int_{\R}h\tilde{V}_1^2dydt+\bar{\mu}\int_0^T\int_{\R}h^2\tilde{V}_2^2dydt\lesssim& \norm{(\tilde{V}_1,\tilde{V}_2)(\cdot,0)}^2+\bar{\mu}\int_0^T\norm{\left(\p_y\tilde{V}_1,\p_y\tilde{V}_2\right)}^2dt\\ \nonumber
	&+\mu^{2\alpha-1}\int_0^T\int_{\R}\tilde{h}^2(\tilde{V}_1^2+\tilde{V}_2^2) dydt +\mu^{\frac{10}{3}\alpha-\frac{8}{3}-\varepsilon}.
\end{align}
where $\tilde{h}(t,y)=\frac{1}{\sqrt{2 \pi \bar{\mu}(1+t)}} \exp \left(-\frac{\left(y+\bar{c}(1+t)\right)^2}{2 \bar{\mu}(1+t)}\right)$. We can also have the same estimate for the case $\bar{c}$, that is, 
	\begin{align}\label{33333}
		\int_0^T\int_{\R}\tilde{h}\tilde{V}_2^2dydt+\bar{\mu}\int_0^T\int_{\R}\tilde{h}^2\tilde{V}_1^2dydt\lesssim& \norm{(\tilde{V}_1,\tilde{V}_2)(\cdot,0)}^2+\bar{\mu}\int_0^T\norm{\left(\p_y\tilde{V}_1,\p_y\tilde{V}_2\right)}^2dt\\ \nonumber
		&+\mu^{2\alpha-1}\int_0^T\int_{\R}h^2(\tilde{V}_1^2+\tilde{V}_2^2) dydt +\mu^{\frac{10}{3}\alpha-\frac{8}{3}-\varepsilon}.
	\end{align}
Therefore, we conclude that
	\begin{align}\label{new14}
\bar{\mu}\int_0^T\int_{\R}\left[\bar{\mu}(1+t)\right]^{-1}e^{-\frac{(y\pm\bar{c}(1+t))^2}{\bar{\mu}(1+t)}}\mathbb{V}_2^2dydt\lesssim& \norm{(\mathbb{V}_1,\mathbb{V}_2)(\cdot,0)}^2+\bar{\mu}\int_0^T\norm{\left(\p_y\mathbb{V}_1,\p_y\mathbb{V}_2\right)}^2dt\\ \nonumber
	&+\mu^{\frac{10}{3}\alpha-\frac{8}{3}-\varepsilon}.
\end{align}
Taking $\cref{v1v2}+\cref{V1y1}\times\bar{\mu}$ and using \cref{new14}, one can prove 
	\begin{align}\label{new15}
\norm{\mathbb{V}_1,\mathbb{V}_2}^2+\bar{\mu}^2\norm{\mathbb{V}_{1y}}^2+\bar{\mu}\int_0^T\norm{\p_y\mathbb{V}_1,\p_y\mathbb{V}_2}^2dt+	\bar{\mu}\int_0^T\int_{\R}(h^2+\tilde{h}^2)\mathbb{V}_2^2dydt\lesssim \mu^{2\alpha},
\end{align}
where we have used $\norm{(\mathbb{V}_1,\mathbb{V}_2)(\cdot,0)}\lesssim \mu^\alpha$. This yields \cref{1or}.
\end{proof}
Then we provide an estimate for $\norm{\p_y^k\tilde{V}_1(\cdot,t),\p_y^k\tilde{V}_2(\cdot,t)}$,  $k=1,2$. 
\begin{Lem}\label{2or}
	Under the same assumptions as \cref{MT}, it holds that
	\begin{align*}			&\norm{\p_y^k\tilde{V}_1(\cdot,t),\p_y^k\tilde{V}_2(\cdot,t)}^2\lesssim \mu^{2\alpha}[\mu(1+t)]^{-k},  \quad k=1,2,	\\
			&\norm{\p_y^k\tilde{V}_1(\cdot,t),\p_y^k\tilde{V}_2(\cdot,t)}^2+  \bar{\mu}\int_{0}^T\norm{\p_y^{k+1}\tilde{V}_1,\p_y^{k+1}\tilde{V}_2}
			^2dt\lesssim \mu^{2\alpha+\frac{1-k}{2}}.
	\end{align*}
\end{Lem}
\begin{proof}
	Applying $\p_y^k$ to \cref{hatv},  and then multiplying the resulting equations by $(\p_y^k\mathbb{V}_1,\p_y^k\mathbb{V}_2)$, one has
	\begin{align}
			&\frac{d}{dt}\norm{\p_y^k\mathbb{V}_1(\cdot,t),\p_y^k\mathbb{V}_2(\cdot,t)}^2+\bar{\mu}\norm{\p_y^{k+1}\mathbb{V}_2}^2\nonumber\\
			\lesssim& \int_{\R}{\p_y^k(\tilde{K}_2-\tilde{K}_1)\p_y^{k}\mathbb{V}_1}+\p_y^k(\tilde{K}_1+\tilde{K}_2)\p_y^{k}\mathbb{V}_2dy.
	\end{align}
	According to \cref{K2+K1}, $\tilde{K}_1+\tilde{K}_2-\frac{E_2}{a\bar{c}}\lesssim\abs{\tilde{E}^{(1)}+\tilde{E}^{(2)}+\tilde{N}}$, it holds that 
	\begin{align}
    &\begin{aligned}\label{999}
     \int_\R &\abs{\p_y^k(\tilde{E}^{(1)}+\tilde{N})\p_y^k\mathbb{V}_2} dy \lesssim \delta\bar{\mu}\norm{\p_y^{k+1}\mathbb{V}_2}^2 + \bar{\mu}^{-1}\norm{\p_y^k(\tilde{E}^{(1)}+\tilde{N})}_{L^1}^{\frac{4}{3}}\norm{\p_y^k \mathbb{V}_2}^{\frac{2}{3}} \\
    &\lesssim \delta\bar{\mu}\norm{\p_y^{k+1}\mathbb{V}_2}^2 + \mu^{\frac{10}{3}\alpha+\frac{1}{3}}[\bar{\mu}(1+t)]^{-k-\frac{4}{3}}+\mu^{\frac{10}{3}\alpha - \frac{4}{9}k -\frac{7}{3}} [\bar{\mu}(1+t)]^{-\frac{k}{3}}e^{-\frac{1}{48}\mu^{\frac{1}{3}}t},
	\end{aligned}\\
		&\begin{aligned}\label{9999}
			&\int_{\R}\abs{\p_y^k\tilde{E}^{(2)}\p_y^{k}\mathbb{V}_{2}}dy\lesssim \sum_{i,j=1,2}\int_{\R}\abs{\p_y^{k}(\mathbb{V}_{iy}\mathbb{V}_{jy}+\mathcal{C}_{i}\mathbb{V}_{jy})\p_y^{k}\mathbb{V}_{2}}dy\\
			&\quad\lesssim  \delta\bar{\mu}\norm{\p_y^{k+1}\mathbb{V}_{2}}^2 +\bar{\mu}^{-1}\sum_{i,j=1,2}\norm{\p_y^k\mathbb{V}_j}^2\norm{\mathbb{V}_{iy}}_{L^\infty}^2+(k-1)\sum_{i,j=1,2}\norm{\p_y^k\mathbb{V}_{i} }\norm{\p_y^k\mathbb{V}_j}_{L^4}^2\\
			&\qquad+\bar{\mu}^{-1}\sum_{i,j=1,2} \norm{\p_y^{k-1}C_i}_{L^\infty}^2\norm{\mathbb{V}_{jy}}^2+\bar{\mu}^{-1}\sum_{i,j=1,2} \norm{C_i}_{L^\infty}^2\norm{\p_y^k \mathbb{V}_j}^2\\
			&\quad\lesssim  \delta\bar{\mu}\norm{\p_y^{k+1}\mathbb{V}_{2}}^2 +\bar{\mu}^{-1}\sum_{i,j=1,2}\left[\norm{\p_y^k\mathbb{V}_j}^2\norm{\mathbb{V}_{iy}}_{L^\infty}^2+(k-1)\norm{\p_y^k\mathbb{V}_{i} }^{\frac{4}{3}}\norm{\p_y^k\mathbb{V}_j}^2\right]\\
			&\qquad+\bar{\mu}^{-1}\sum_{i,j=1,2} \norm{\p_y^{k-1}C_i}_{L^\infty}^2\norm{\mathbb{V}_{jy}}^2+\bar{\mu}^{-1}\sum_{i,j=1,2} \norm{C_i}_{L^\infty}^2\norm{\p_y^k \mathbb{V}_j}^2\\
			&\quad\lesssim\delta\bar{\mu}\norm{\p_y^{k+1}\mathbb{V}_{2}}^2+{\mu}^{2\alpha-\varepsilon-\frac{3}{2}}\left[\bar{\mu}(1+t)\right]^{-1}\norm{\p_y^k(\mathbb{V}_1,\mathbb{V}_2)}^2\\[1mm]
			&\qquad\;+(k-1){\mu}^{2\alpha-1}\left[\bar{\mu}(1+t)\right]^{-2}\norm{\p_y^{k-1}(\mathbb{V}_{1},\mathbb{V}_2)}^2+(k-1)\mu^{\frac{10}{3}\alpha-1}[\mu (1+t)]^{-\frac{10}{3}},
		\end{aligned}\\
		&\begin{aligned}\label{99999}
			\Bigg|\int_{\R}\p_y^k{E}_2&\p_y^{k}\mathbb{V}_{2}dy\Bigg|\lesssim\sum_{i,j=1,2} \sum_{d=0}^{k}\int_{\R}\abs{\left[\p_y^{1+d}\mathbb{V}_{1}\p_y^{k-d+1}\mathbb{V}_{2}+\p_y^{k}\big(\mathcal{C}_{i}\mathcal{C}_{j}\big)+\p_y^d\mathcal{C}_{i}\p_y^{1+k-d}\mathbb{V}_{j}\right]\p_y^{k+1}\mathbb{V}_{2}}dy\\
			&\qquad\lesssim\delta\bar{\mu}\norm{\p_y^{k+1}\mathbb{V}_{1},\p_y^{k+1}\mathbb{V}_{2}}^2+{\mu}^{4\alpha-\frac{3}{2}-\varepsilon}\left[\bar{\mu}(1+t)\right]^{-\frac{3}{2}-k}\\
			&+{\mu}^{2\alpha-2-\varepsilon}\left[\bar{\mu}(1+t)\right]^{-1}\norm{\p_y^{k}(\mathbb{V}_1,\mathbb{V}_2)}^2+(k-1){\mu}^{2\alpha}\left[\bar{\mu}(1+t)\right]^{-2}\norm{\p_y^{k-1}(\mathbb{V}_1,\mathbb{V}_2)}^2,
		\end{aligned}\\
	&\begin{aligned}\label{999999}
		&\abs{\int_{\R}\p_y^k(\tilde{K}_2-\tilde{K}_1)\p_y^{k}\mathbb{V}_1dy}\lesssim \delta\mu\norm{\p_y^{k+1}\mathbb{V}_1}^2+\mu^{2\alpha+3}\left[\bar{\mu}(1+t)\right]^{-\frac{3}{2}-k}.
	\end{aligned}
	\end{align}
	where we have used \cref{estimateon2-1}, \cref{ENL1}, \cref{ENL2} and the a priori assumption \cref{aps}.
	Then we have
	\begin{align}\label{oneone}
			&\frac{d}{dt}\norm{\p^k_y\mathbb{V}_1(\cdot,t),\p^k_y\mathbb{V}_2(\cdot,t)}^2+\bar{\mu}\norm{\p_y^{k+1}\mathbb{V}_2}^2\\ \nonumber
			\lesssim&\delta\bar{\mu}\norm{\p_y^{k+1}\mathbb{V}_{1}}^2 +{\mu}^{2\alpha-2}\left[\bar{\mu}(1+t)\right]^{-1}\norm{\p_y^k(\mathbb{V}_1,\mathbb{V}_2)}^2\\ \nonumber
			&\qquad+(k-1){\mu}^{2\alpha-1}\left[\bar{\mu}(1+t)\right]^{-2}\norm{\p_y^{k-1}(\mathbb{V}_{1},\mathbb{V}_2)}^2+{\mu}^{2\alpha+3}\left[\bar{\mu}(1+t)\right]^{-\frac{3}{2}-k}.
	\end{align}

	Multiplying $\p_y^k$\cref{hatv}$_2$ by $\p_y^{k+1}\mathbb{V}_{1}$, one has
	\begin{align}
	&\bar{c}(\p_y^{k+1}\mathbb{V}_{1})^2+(\p_y^k\mathbb{V}_2\p_y^{k+1}\mathbb{V}_{1})_t+\p_y^{k+1}\mathbb{V}_{2}\p_y^k\mathbb{V}_{1t}-\bar{\mu}\p_y^{k+2}\mathbb{V}_{2}\p_y^{k+1}\mathbb{V}_{1}\\
	&\qquad\qquad=\p_y^k(\tilde{K}_1+\tilde{K}_2)\p_y^{k+1}\mathbb{V}_{1}+(\cdots)_y,\notag
\end{align}	
which implies from \cref{hatv}$_1$$\Big($i.e. $\p_y^k\mathbb{V}_{1t}=-\p_y^{k}(\bar{c}\mathbb{V}_{2y}+\tilde{K}_1-\tilde{K}_2),\; \p_y^{k+2}\mathbb{V}_{2}=-\p_y^{k+1}\left(\frac{\mathbb{V}_{1t}}{\bar{c}}-\frac{1}{\bar{c}}(\tilde{K}_2-\tilde{K}_1)\right)$$\Big)$ that 
	\begin{align}\label{V1y11}
&	\bar{c}(\p_y^{k+1}\mathbb{V}_{1})^2+(\p_y^k\mathbb{V}_2\p_y^{k+1}\mathbb{V}_{1}+\frac{\bar{\mu}}{2\bar{c}}(\p_y^{k+1}\mathbb{V}_{1})^2)_t \notag\\
=&\bar{c}(\p_y^{k+1}\mathbb{V}_{2})^2+\p_y^k(\tilde{K}_1-\tilde{K}_2)\p_y^{k+1}\mathbb{V}_{2}+\frac{\bar{\mu}}{\bar{c}}\p_y^{k+1}(\tilde{K}_2-\tilde{K}_1)\p_y^{k+1}\mathbb{V}_{1}\\
&+\p_y^k(\tilde{K}_1+\tilde{K}_2)\p_y^{k+1}\mathbb{V}_{1}+(\cdots)_y. \notag
\end{align}	
Integrating \cref{V1y11} on $\R$ gives that 	
	\begin{align}\label{V1yy}
			&\p_t\left(\int_{\R}\p_y^{k}\mathbb{V}_2\p_y^{k+1}\mathbb{V}_{1}+\frac{\bar{\mu}}{2\bar{c}}\abs{\p_y^{k+1}\mathbb{V}_{1}}^2dy \right)+	\bar{c}\norm{\p_y^{k+1}\mathbb{V}_{1}}^2\\
			=&\int_{\R}\p^{k}_y\left(\tilde{K}_1-\tilde{K}_2\right)\p_y^{k+1}\mathbb{V}_{2}+\p^{k}_y\left[\frac{\bar{\mu}}{\bar{c}}\left(\tilde{K}_2-\tilde{K}_1\right)_y+\tilde{K}_{1}+\tilde{K}_{2}\right]\p_y^{k+1}\mathbb{V}_{1}dy+\bar{c}\norm{\p_y^{k+1}V_2}^2. \nonumber
	\end{align}
All terms concerning $\tilde{K}_2-\tilde{K}_1$ in \cref{V1yy} can be estimated in a similar way as in \cref{999999}. We only need to estimate the last term concerning $\tilde{K}_2+\tilde{K}_1$. Note from \cref{K2+K1} that $\tilde{K}_1+\tilde{K}_2-\frac{E_2}{a\bar{c}}\lesssim\abs{\tilde{E}^{(1)}+\tilde{E}^{(2)}+\tilde{N}}$, one should take care of $ \frac{1}{a\bar{c}}\int_\R \p_y^{k} E_2 \p_y^{k+1}\mathbb{V}_1 dy$, since it contains the highest order of derivatives $\p_y^{k+2}\mathbb{V}_2$. Indeed,  we have from  \cref{E2} that 
	\begin{align}
		\frac{1}{a\bar{c}}&\int_{\R} \p_y^kE_2\p_y^{k+1}\mathbb{V}_1 dy=-a\bar{\mu}\int_\R \p_y^{k+1}\left[\frac{(\mathbb{V}_{2y}+\mathcal{C}_1+\mathcal{C}_2)(\mathbb{V}_{1y}-C_1+C_2)}{\mathbb{V}_{1y}+\mathcal{C}_2-\mathcal{C}_1+1}\right]\p_y^{k+1}\mathbb{V}_1dy \nonumber\\
		=&-a\bar{\mu}\int_\R \sum_{r=0}^k\left\{\p_y^{k-r}\left[(\mathbb{V}_{1y}+\mathcal{C}_2-\mathcal{C}_1)(\mathbb{V}_{2y}+\mathcal{C}_1+\mathcal{C}_2)\right]\p_y^{r+1}\left(\frac{1}{\mathbb{V}_{1y}+\mathcal{C}_2-\mathcal{C}_1+1}\right)\right\}\p_y^{k+1}\mathbb{V}_1 dy \notag\\
		&-a\bar{\mu}\int_\R \sum_{d=0}^{k+1} \left[\frac{\p_y^{k+1-d}(\mathbb{V}_{2y}+\mathcal{C}_1+\mathcal{C}_2)\p_y^d(\mathbb{V}_{1y}-\mathcal{C}_1+\mathcal{C}_2)}{\mathbb{V}_{1y}+\mathcal{C}_2-\mathcal{C}_1+1}\right]\p_y^{k+1}\mathbb{V}_1 dy =:\mathcal{H}_1 +\mathcal{H}_2. \label{zazaza}
	\end{align}
	For $\mathcal{H}_1$, we only need to consider $\mathcal{H}_{1q}:=-a\bar{\mu}\int_\R \frac{(\mathbb{V}_{1y}+\mathcal{C}_2-\mathcal{C}_1)(\mathbb{V}_{2y}+\mathcal{C}_1+\mathcal{C}_2)\p_y^{k+1}(\mathbb{V}_{1y}+\mathcal{C}_2-\mathcal{C}_1)}{(\mathbb{V}_{1y}+\mathcal{C}_2-\mathcal{C}_1+1)^2}\p_y^{k+1}\mathbb{V}_1 dy $ , since the other terms can be treated similarly.
	By  the a priori assumption \cref{aps}, one has 
	\begin{align*}
	\abs{\mathcal{H}_{1q}}&\lesssim \int_\R \sum_{i,j=1,2}\abs{\p_y^{k+1}\mathbb{V}_1}^2\left(\abs{\p_y^2\mathbb{V}_i}+\abs{\p_y \theta_j} \right)\left(\abs{\p_y\mathbb{V}_{3-i}}+ \abs{\theta_j}\right)+ \sum_{i,j,k=1,2}\abs{\p_y^{k+1}\mathbb{V}_1 \p_y^{k+1}\theta_k}(\mathbb{V}_{iy}+\theta_j)^2 dy\\
     & \lesssim \delta \mu \norm{\p_y^{k+1}\mathbb{V}_1}^2+\mu^{6\alpha-1-\varepsilon}[\mu(1+t)]^{-\frac{5}{2}-k}.
	\end{align*}
	For $\mathcal{H}_2$, the main term is $ -a\bar{\mu}\int_{\R}\frac{(\mathbb{V}_{1y}+\mathcal{C}_2-\mathcal{C}_1)\p_y^{k+2}\mathbb{V}_{2}\p_y^{k+1}\mathbb{V}_1}{\mathbb{V}_{1y}+\mathcal{C}_2-\mathcal{C}_1+1} dy $.  From \cref{hatv}$_1$, we have $\p_t \p_y^{k+1} \mathbb{V}_1 + \bar{c} \p_y^{k+2} \mathbb{V}_2 = \p_y^{k+1}(\tilde{K}_2-\tilde{K}_1)$. Then we get
	\begin{align}
		&{\int_{\R} -a\bar{\mu}\frac{\mathbb{V}_{1y}\p_y^{k+2}\mathbb{V}_{2}\p_y^{k+1}\mathbb{V}_1}{\mathbb{V}_{1y}+\mathcal{C}_2-\mathcal{C}_1+1}-\frac{a\bar{\mu}}{2\bar{c}}\p_t\left[\frac{\mathbb{V}_{1y}(\p_y^{k+1}\mathbb{V}_1)^2}{\mathbb{V}_{1y}+\mathcal{C}_2-\mathcal{C}_1+1}\right]dy}\\\nonumber
		=&\frac{a\bar{\mu}}{\bar{c}}\int_\R \frac{\p_y^{k+1}(\tilde{K}_1-\tilde{K}_2)\mathbb{V}_{1y}\p_y^{k+1}\mathbb{V}_1}{\mathbb{V}_{1y}+\mathcal{C}_2-\mathcal{C}_1+1}dy -\frac{a\bar{\mu}}{2\bar{c}}\int_\R (\p_y^{k+1}\mathbb{V}_1)^2 \p_t \left(\frac{\mathbb{V}_{1y}}{\mathbb{V}_{1y}+\mathcal{C}_2-\mathcal{C}_1+1}\right)dy \nonumber\\
		\lesssim& \mu^2 \norm{\p_y^{k+1}\mathbb{V}_1}^2+{\mu}^{3\alpha+2}\left[\bar{\mu}(1+t)\right]^{-\frac{7}{2}-k}\nonumber.
	\end{align}
	In the same way, we obtain
	\begin{align}\label{cvcvcv}
	&\int_{\R}-a\bar{\mu} \frac{\p_y^{k+2}\mathbb{V}_{2}\left(\mathcal{C}_2-\mathcal{C}_1\right)\p_y^{k+1}\mathbb{V}_1}{\mathbb{V}_{1y}+\mathcal{C}_2-\mathcal{C}_1+1}-\frac{a\bar{\mu}}{2\bar{c}}\p_t\left[\frac{\left(\mathcal{C}_2-\mathcal{C}_1\right)(\p_y^{k+1}\mathbb{V}_1)^2}{\mathbb{V}_{1y}+\mathcal{C}_2-\mathcal{C}_1+1}\right] dy \\ \nonumber
	 \lesssim &\mu^2 \norm{\p_y^{k+1}\mathbb{V}_1}^2+{\mu}^{3\alpha+2}\left[\bar{\mu}(1+t)\right]^{-\frac{5}{2}-k}.
	\end{align}
Obviously, $\int_\R \abs{\p_y^{k}( \tilde{E}^{(1)}+\tilde{N}) \p_y^{k+1}\mathbb{V}_1} dy$ can be treated as in \cref{nmnm} and $\int_\R \abs{\p_y^{k} \tilde{E}^{(2)} \p_y^{k+1}\mathbb{V}_1} dy$ can be treated as in 
 \cref{9999}. Thus combining \cref{V1yy}-\cref{cvcvcv}, we have
	\begin{align}\label{twotwo}
			&\frac{d}{dt}\int_{\R}\bar{\mu}\p^{k}_y\mathbb{V}_2\p_y^{k+1}\mathbb{V}_1+\frac{\bar{\mu}^2}{2\bar{c}}\abs{\p_y^{k+1}\mathbb{V}_1}^2 + \frac{a\bar{\mu}^2}{2\bar{c}}\frac{\mathcal{C}_1-\mathbb{V}_{1y}-\mathcal{C}_2}{\mathbb{V}_{1y}+\mathcal{C}_2-\mathcal{C}_1+1}(\p_y^{k+1}\mathbb{V}_1)^2dy+	\bar{c}\bar{\mu}\norm{\p_y^{k+1}\mathbb{V}_1}^2\\ \nonumber
			\lesssim& \bar{\mu}\norm{\p_y^{k+1}\mathbb{V}_2}^2+{\mu}^{2\alpha+3}\left[\bar{\mu}(1+t)\right]^{-\frac{3}{2}-k}+\mu^{4\alpha-\frac{13}{3}-2\varepsilon} e^{-\frac{1}{32}\mu^{\frac{1}{3}}t}. 
	\end{align}
	
Therefore,  we obtain from  \cref{oneone,twotwo} that
	\begin{align}\label{1111}
		\begin{aligned}
			&\frac{d}{dt}\mathcal{E}_k+\bar{\mu}\norm{\p_y^{k+1}\mathbb{V}_2,\p_y^{k+1}\mathbb{V}_1}^2
			\lesssim{\mu}^{2\alpha-2-\varepsilon}\left[\bar{\mu}(1+t)\right]^{-1}\norm{\p_y^k(\mathbb{V}_1,\mathbb{V}_2)}^2\\
			&+(k-1){\mu}^{2\alpha-1}\left[\bar{\mu}(1+t)\right]^{-2}\norm{\p_y^{k-1}(\mathbb{V}_{1},\mathbb{V}_2)}^2+{\mu}^{2\alpha+\frac{3}{2}-k}(1+t)^{-\frac{3}{2}-k},
		\end{aligned}
	\end{align}
	where
	\begin{align*}
\mathcal{E}_k:=\int_{\R}\bar{\mu}\p^{k}_y\mathbb{V}_2\p_y^{k+1}\mathbb{V}_1+\frac{\bar{\mu}^2}{2\bar{c}}\abs{\p_y^{k+1}\mathbb{V}_1}^2+\frac{a\bar{\mu}^2}{2\bar{c}}\frac{\mathcal{C}_1-\mathbb{V}_{1y}-\mathcal{C}_2}{\mathbb{V}_{1y}+\mathcal{C}_2-\mathcal{C}_1+1}(\p_y^{k+1}\mathbb{V}_1)^2+\hat{C}^2	\norm{\p^k_y\mathbb{V}_1(\cdot,t),\p^k_y\mathbb{V}_2(\cdot,t)}^2,
	\end{align*}
	and $\hat{C}>0$  is a suitably large constant.
	Integrating \cref{1111} over $(0,T)$ and using \cref{1or} for $i,k=1,2$, one has 
	\begin{align}
	\norm{\p_y^k \mathbb{V}_i(\cdot,t)}+\bar{\mu}\int_0^T \norm{\p_y^{k+1} \mathbb{V}_i} dt \lesssim \mu^{2\alpha+\frac{1-k}{2}}. 
	\end{align}
	Moreover, multiplying \cref{1111} by $1+t$ for $k=1$,   \cref{1111} by $\left[\bar{\mu}(1+t)\right]^2$ for $k=2$ and then integrating the resulting equations, one has
	\begin{align}
			&(1+t)\norm{\p_y\mathbb{V}_1(\cdot,t),\p_y\mathbb{V}_2(\cdot,t)}^2+\int_0^T\bar{\mu}(1+t)\norm{\p_y^{2}\mathbb{V}_1,\p_y^{2}\mathbb{V}_2}^2dt\\ \nonumber
			\lesssim&\norm{\p_y\mathbb{V}_1(\cdot,0),\p_y\mathbb{V}_2(\cdot,0)}^2+\int_0^{T}\norm{\p_y(\mathbb{V}_1,\mathbb{V}_2)}^2dt+\mu^{2\alpha}\bar{\mu}^\frac{1}{2},\\ 
			&\left[\bar{\mu}(1+t)\right]^2\norm{\p^2_y\mathbb{V}_1(\cdot,t),\p_y^2\mathbb{V}_2(\cdot,t)}^2+\bar{\mu}\int_0^T\left[\bar{\mu}(1+t)\right]^2\norm{\p_y^{3}\mathbb{V}_1,\p_y^{3}\mathbb{V}_2}^2dt\\ \nonumber
			\lesssim&\bar{\mu}^2\norm{\p_y\mathbb{V}_1(\cdot,0),\p_y\mathbb{V}_2(\cdot,0)}^2+{\mu}^{2\alpha-1}\int_0^T\norm{\p_y(\mathbb{V}_{1},\mathbb{V}_2)}^2dt\\ \nonumber
			&+\bar{\mu} \int_0^{T}\bar{\mu}(1+t)\norm{\p^2_y(\mathbb{V}_1,\mathbb{V}_2)}^2dt+{\mu}^{2\alpha+\frac{3}{2}}.
	\end{align}
	In terms of \cref{1or}, we have
	\begin{align}
		&\norm{\p_y^k\tilde{V}_1(\cdot,t),\p_y^k\tilde{V}_2(\cdot,t)}^2\lesssim \mu^{2\alpha}[\mu(1+t)]^{-k},  \quad k=1,2,	\\ \nonumber
		&\norm{\p_y^k\tilde{V}_1(\cdot,t),\p_y^k\tilde{V}_2(\cdot,t)}^2+  \bar{\mu}\int_{0}^T\norm{\p_y^{k+1}\tilde{V}_1,\p_y^{k+1}\tilde{V}_2}
		^2dt\lesssim \mu^{2\alpha+\frac{1-k}{2}}.
	\end{align}
	which completes the proof of \cref{2or}.
\end{proof}
\subsection{Estimate on $\theta_A$}
This subsection aims to estimate the main part of $\psim_1$ (i.e. $\theta_A$). By Duhamel's principle and \cref{thetaA}, one has 
\begin{align}\label{Theta-A}
		{\theta_A(y, t)}=\int_0^t \int_{-\infty}^{\infty} &H(0,\mu) \bigg(a\bar{c}(\Xi_1+\Xi_2-\mathcal{C}_1-\mathcal{C}_2)-\frac{a\lambda}{2}\left(\Xi_2-\Xi_1\right)_y+\frac{a}{2}(\theta_1^2-\theta_2^2)\\\nonumber
		&\qquad-a\bar{c}(\mathcal{C}_1+\mathcal{C}_2)\p_y\theta_A+a^2\bar{c}(\mathcal{C}_1+\mathcal{C}_2)(\mathcal{C}_2-\Xi_2-\mathcal{C}_1+\Xi_1)\\ \nonumber
		&\qquad-a\bar{c}(\tilde{v}_1+\tilde{v}_2)\p_y\theta_A+a^2\bar{c}(\tilde{v}_1+\tilde{v}_2)(\mathcal{C}_2-\mathcal{C}_1)\bigg) d s d \tau. \nonumber
\end{align}

\begin{Lem}\label{estimateonta}
	Under the same assumptions of \cref{MT},  it holds that
	\begin{align}
		\norm{\theta_A(y, t)}_{L^\infty} \lesssim \mu^{\alpha} ,\quad \norm{\p_y^k\theta_{A}(x, t)}\lesssim\mu^{\alpha} \left[\mu(1+t)\right]^{-\frac{k}{2}+\frac{1}{4}}, \quad(k=1,2).
	\end{align}
\end{Lem}
\begin{proof}
	Let
	\begin{align}\label{MMMM}
		\hat{M}(t):=\max_{0 \leq k \leq 2}\left\{\norm{\theta_A(y, t)}_{L^\infty}+\norm{\p_y^k\theta_{A}(x, t)}\left[\mu(1+t)\right]^{\frac{k}{2}-\frac{1}{4}}\right\}.
	\end{align}
A direct computation gives that 
	\begin{align}\label{thetaA1}
		\p^k_y\theta_A 
		=&\int_0^{\frac{t}{2}} \int_{-\infty}^{\infty} \p^{k}_yH(0,\mu)S_{\mathcal{A}}(s, \tau) d s d \tau+\int_{\frac{t}{2}}^t \int_{-\infty}^{\infty} H(0,\mu) \p_s^{k}S_{\mathcal{A}}d s d \tau\\ \nonumber
		&\qquad\qquad-\int_0^{t} \int_{-\infty}^{\infty} a\bar{c}H(0,\mu)\p^k_s(\theta_1+\theta_2)d s d \tau:=I_7+I_8+I_9,
\end{align}
	where $S_\mathcal{A}$ is defined as
	\begin{align*}
		S_{\mathcal{A}}:=&\frac{a\bar{\mu}}{4}(\theta_{2y}-\theta_{1y})-\frac{a}{4}(\lambda-2\mu)\left(\Xi_2-\Xi_1\right)_y+\frac{a}{2}(\theta_1^2-\theta_2^2)\\
		&-a\bar{c}(\mathcal{C}_1+\mathcal{C}_2)\p_y\theta_A+a^2\bar{c}(\mathcal{C}_1+\mathcal{C}_2)(\mathcal{C}_2-\Xi_2-\mathcal{C}_1+\Xi_1)\\
	&-a\bar{c}(\tilde{v}_1+\tilde{v}_2)\p_y\theta_A+a^2\bar{c}(\tilde{v}_1+\tilde{v}_2)(\mathcal{C}_2-\mathcal{C}_1) d s d \tau,\\
\mbox{and}\quad			\abs{S_{\mathcal{A}}}=&O(1)\sum_{i,j=1,2}\left[\bar{\mu}\abs{\theta_{iy}}+\theta_i^2+\left(\abs{\theta_i}+\abs{\p_y\tilde{V}_j}\right)\abs{\p_y\theta_A}+\abs{\p_y\tilde{V}_i\theta_j}\right].
	\end{align*}
	For $k=0,1$, by \cref{estimateontheta}-\ref{1or}, \ref{2or} and \cref{MMMM}, one has,
	\begin{align}
		\norm{\p_y^kS_{\mathcal{A}}}_{L^1 }\lesssim \mu^{\alpha}\left[\mu(1+t)\right]^{-\frac{1}{2}-\frac{k}{2}}\hat{M}(t)+\mu^{\alpha}\bar{\mu}\left[\bar{\mu}(1+t)\right]^{-\frac{1}{2}-\frac{k}{2}}.
	\end{align}
	Thus, for $k=0,1,2$, there holds 
	\begin{align}\label{I7I8}
			\norm{I_7}+\norm{I_8}\lesssim \mu^{\alpha-1}\left[\bar\mu(1+t)\right]^{-\frac{k}{2}+\frac{1}{4}}\hat{M}(t)+\mu^{\alpha}\left[\bar\mu(1+t)\right]^{-\frac{k}{2}+\frac{1}{4}},
	\end{align}
    where we have used the fact that
	\begin{align*}
		\int_{\frac{t}{2}}^t \int_{-\infty}^{\infty} H(0,\mu) \p_s^{k}S_{\mathcal{A}}d s d \tau=-\int_{\frac{t}{2}}^t \int_{-\infty}^{\infty} \p_s H(0,\mu) \p_s^{k-1}S_{\mathcal{A}}d s d \tau,\quad(k=1,2).
	\end{align*}
	Since the propagation speeds of $\theta_i$ and $H(0,\mu)$ are different, applying \cref{semi-group} and using the same method as $I_1$ in \cref{I123} and $I_6$ in  \cref{pk1xi}, we obtain
	\begin{align}\label{I9}
		\norm{I_9}_{L^\infty}\lesssim \mu^{\alpha}\quad(k=0), \qquad \norm{I_9}\lesssim  \mu^{\alpha}\left[\mu(1+t)\right]^{-\frac{k}{2}+\frac{1}{4}}\quad(k=1,2).
	\end{align}
Combining \cref{thetaA1,I7I8,I9}, we deduce that $\hat{M}(t)\lesssim \mu^{\alpha}+\mu^{\alpha-1}\hat{M}(t)$.  Note that $\alpha>11/3$, one has $\hat{M}(t)\lesssim \mu^{2\alpha}$ as $\mu$ is small. Substituting $\hat{M}(t)\lesssim \mu^{2\alpha}$ into \cref{MMMM} yields Lemma \ref{estimateonta}.
\end{proof}

\subsection{Estimate on $\mathcal{A}$}
This subsection is devoted to the decay part of $\psim_1$ (i.e. $\mathcal{A}$). Recall
\begin{align}\label{equ-AA}
\p_t{\mathcal{A}}- \mu \p_y^2{\mathcal{A}}-a\mu\p_y(\tilde{v}_2-\tilde{v}_1)=-\psim_{2}\p_y\mathcal{A}-\mu\frac{\phim}{\phim+1}\p_y^2\mathcal{A}+F_{A},
\end{align}
where
\begin{align*}
\abs{F_A}=&O(1)\sum_{i,j=1,2}\Big[\abs{\p_y\tilde{V}_i\Xi_j}+\mu\abs{\p_y\tilde{V}_i\p_y^2\tilde{V}_j}\\
+&\mu\abs{\p_y^2\theta_A}\left(\abs{\theta_j}+\abs{\p_y\tilde{V}_i}\right)+\abs{\Do(\nabla^2\psi_{i\neq}\phi_\neq)}+\bar{\mu}^2\abs{\p_y^2\theta_i} \Big],
\end{align*}
using \cref{estimateontheta}-\ref{1or}, \cref{2or} and \ref{estimateonta}, a direct calculation yields that 
\begin{align}	&\norm{F_A}_{L^1}\lesssim\left(\mu^{2\alpha}+\mu^{\alpha+2} \right)\left[\bar{\mu}(1+t)\right]^{-1}+\mu^{2\alpha-\frac{4}{3}-\varepsilon}e^{-\frac{1}{64}\mu^{\frac{1}{3}}t},\\[2mm]
	&\norm{F_A}\lesssim \left(\mu^{2\alpha}+\mu^{\alpha+2} \right)[\bar{\mu}(1+t)]^{-\frac{5}{4}}+\mu^{2\alpha-\frac{4}{3}-\varepsilon}e^{-\frac{1}{64}\mu^\frac{1}{3}t}.
\end{align}
Then, we have
\begin{Lem}\label{Aor}
	Under the same assumptions of \cref{MT}, it holds that
	\begin{align}
		&\norm{\mathcal{A}(\cdot,t)}^2\lesssim\mu^{2\alpha},\qquad \norm{\p_y\mathcal{A}(\cdot,t)}^2\lesssim\mu^{2\alpha} [\mu(1+t)]^{-1},\\ \nonumber
        &\norm{\p_y^k\mathcal{A}(\cdot,t)}^2+\mu \int_0^T \norm{\p_y^{k+1}\mathcal{A}}^2 dt \lesssim \mu^{2\alpha} , \quad k=0,1.
	\end{align}
\end{Lem}
\begin{proof}
	Since $\abs{\phim,\psim_2}=O(1)\sum_{i,j=1}^2\left[\abs{\tilde{V}_{iy}}+\abs{\theta_j}\right]$, multiplying \cref{equ-A} by $\mathcal{A}$ and integrating the resulting equation, by the apriori assumption \cref{aps}, \cref{estimateontheta}-\ref{1or}, \ref{2or} and \ref{estimateonta}, one has,
	\begin{align}
			&\norm{\mathcal{A}(\cdot,T)}^2+\mu\int_0^T\norm{\mathcal{A}_y}^2dt\\ \nonumber
			= &\norm{\mathcal{A}(\cdot,0)}^2+ \int_0^T\int_\R F_A \mathcal{A} - \psim_2 \mathcal{A}\mathcal{A}_y-\mu \frac{\phim}{\phim+1} \p_y^2 \mathcal{A} \mathcal{A} + a\mu(\tilde{v}_2-\tilde{v}_1)\p_y \mathcal{A}  dydt\\ \nonumber
			\lesssim&\norm{\mathcal{A}(\cdot,0)}^2+\mu^{2\alpha-1}\int_0^T\int_{\R}\left[\bar{\mu}(1+t)\right]^{-1}e^{-\frac{(y\pm\bar{c}(1+t))^2}{\bar{\mu}(1+t)}}\mathcal{A}^2dydt+\mu\int_0^T \norm{\tilde{v}_2-\tilde{v}_1}^2 dt+\delta\mu\int_0^T\norm{\mathcal{A}_y}^2dt\\ \nonumber
			&+\mu^{-1}\norm{\mathcal{A}}_{L^\infty L^2}\norm{\p_y\mathcal{A}}_{L^\infty L^2}\sum_{i,j,k=1}^2\int_0^T \norm{\p_y\theta_i}^2+\norm{\p_y^k\tilde{V}_j}^2dt+\int_0^T\norm{F_A}_{L^1}\norm{\mathcal{A}}^{\frac{1}{2}}\norm{\mathcal{A}_{y}}^{\frac{1}{2}}dt\\
			\lesssim&\mu^{2\alpha}+\int_0^T\delta\mu\norm{\mathcal{A}_y}^2dt+\mu^{2\alpha-1}\int_0^T\int_{\R}\left[\bar{\mu}(1+t)\right]^{-1}e^{-\frac{(y\pm\bar{c}(1+t))^2}{\bar{\mu}(1+t)}}\mathcal{A}^2dydt,\nonumber
	\end{align}
	where we have used the fact that
	\begin{align*}
	&\int_0^T \int_{\R} \psim_2 \mathcal{A} \mathcal{A}_y dydt \leq \delta \mu \int_0^T \norm{\mathcal{A}_y}^2 dt + C \mu^{-1} \sum_{i,j=1}^2 \int_0^T \int_{\R} \left(\theta_j^2 + \tilde{V}_{iy}^2 \right) \mathcal{A}^2 dydt\\
\quad\lesssim&\mu^{2\alpha-1}\int_0^T\int_{\R}\left[\bar{\mu}(1+t)\right]^{-1}e^{-\frac{(y\pm\bar{c}(1+t))^2}{\bar{\mu}(1+t)}}\mathcal{A}^2dydt+\delta \mu \int_0^T \norm{\mathcal{A}_y}^2 dt\\
&+\mu^{-1}\norm{\mathcal{A}}_{L^\infty L^2}\norm{\p_y\mathcal{A}}_{L^\infty L^2}\sum_{i=1}^2\int_0^T \norm{\p_y\tilde{V}_i}^2dt.
	\end{align*}
	Similar to \cref{q}-\cref{new1}, one has
	\begin{align}
			\bar{\mu}\int_0^T\int_{R}\left[\bar{\mu}(1+t)\right]^{-1}&e^{-\frac{(y\pm\bar{c}(1+t))^2}{\bar{\mu}(1+t)}}\mathcal{A}^2 dydt \lesssim  \bar{\mu}^{\frac{1}{2}}\int_0^T\int_\R\eta_1h\mathcal{A}^2dydt\\
			\lesssim& \mu^{\frac{1}{2}}\norm{\mathcal{A}(\cdot,0)}^2+{\mu}^{\frac{3}{2}}\int_0^T\norm{\mathcal{A}_y}^2dt\nonumber.
	\end{align}
Then we arrive at
\begin{align}
	&\norm{\mathcal{A}(\cdot,T)}^2+\mu\int_0^T\norm{\mathcal{A}_y}^2dt
	\lesssim \mu^{2\alpha},
\end{align}
due to $\norm{\mathcal{A}(\cdot,0)}\lesssim \mu^{\alpha}$. For the higher-order estimate, one has
	\begin{align}\label{asss}
			\frac{d}{dt}\norm{\mathcal{A}_y(\cdot,t)}^2&+\mu\norm{\mathcal{A}_{yy}}^2\lesssim \abs{\int_{\R}\psim_2 \mathcal{A}_y \mathcal{A}_{yy}+F_A\mathcal{A}_{yy}dy}+\mu\norm{\p_y(\tilde{v}_2-\tilde{v}_1)}^2+\delta\mu\norm{\mathcal{A}_{yy}}^2 \\ \nonumber
			\lesssim& \mu^{\alpha-1}(1+t)^{-1}\norm{\mathcal{A}_{y}}^2+\mu^{-1}\norm{F_A}^2+\mu\norm{\tilde{v}_{iy}}^2+\delta\mu\norm{\mathcal{A}_{yy}}^2.
	\end{align}
	 Multiplying \cref{asss} by $(1+t)$, and integrating the resulting equation on $(0,T)$, we arrive at
	\begin{align}
			(1+t)\norm{\mathcal{A}_y}^2+\mu\int_0^T(1+t)\norm{\mathcal{A}_{yy}}^2dt
			\leq \norm{\mathcal{A}_{y}(\cdot,0)}^2+\int_0^T\norm{\mathcal{A}_y}^2dt+\mu^{2\alpha-1},
		\end{align}
		where  we have used the fact from \cref{2or} that $\bar{\mu}^2\int_0^T (1+t)\norm{\tilde{v}_{iy}}^2dt\lesssim \mu^{2\alpha}$.  Integrating \cref{asss} over $(0,T)$, one has
	\begin{align}
		\norm{\mathcal{A}_y}^2+\mu\int_0^T\norm{\mathcal{A}_{yy}}^2dt \lesssim \mu^{2\alpha}.
	\end{align}
	Then it holds that 
	\begin{align}
		\norm{\mathcal{A}_y}^2\lesssim  \mu^{2\alpha}[\mu(1+t)]^{-1},\quad \norm{\mathcal{A}_y}^2+\mu\int_0^T\norm{\mathcal{A}_{yy}}^2dt \lesssim \mu^{2\alpha}.
	\end{align}
Therefore \cref{Aor} is completed.
\end{proof}

\section{Estimates on non-zero modes}

By the Helmholtz decomposition, the velocity $\psi$ can be decomposed as 
$$ \psi=(\psi_1,\psi_2)=\nabla\lap^{-1} \dv\psi +\nabla^{\bot}\lap^{-1}(\nabla^{\bot}\cdot \psi). $$

\begin{Thm}\label{ed}
	Under the same assumptions of \cref{MT},  it holds that
	\begin{align}
		\norm{\nabla^{1+n} \phi_\neq}^2+\norm{\nabla^{1+n} \psi_\neq}^2\lesssim  \mu^{2\alpha-\frac{2n+3}{3}} e^{-\frac{1}{64}\mu^{\frac{1}{3}}t}, \quad n=0,1.
	\end{align}
\end{Thm}

In order to remove the transport term from the equations \cref{perturbation1}, we use the following coordinates \begin{align*}
	X=x-yt,\qquad Y=y,
\end{align*}
then it holds that,
\begin{align}\label{natations-differential operators}
	\begin{cases}
		\p_x=\p_X,\qquad \pt_{XY}:=\p_X\pt_Y,& \nablat =(\p_X,\pt_Y),\\[1.5mm]
		\p_y=\pt_Y:=\p_Y-t\p_X,\qquad\qquad\qquad
		&\tilde{\dv}=\tilde{\nabla}\cdot,\\[1.5mm]
		\lap=\lapt:=\p_{XX}+(\p_Y-t\p_X)^2,&\tilde{\nabla}^{\bot}=(-\pt_Y,\p_X).
	\end{cases}
\end{align}
The functions in the new coordinates are denoted as,
\begin{align}\begin{cases}
		R(t,X,Y)=\phi(t,X+tY,Y),\\[1.5mm]
		A(t,X,Y)=\dv\psi(t,X+tY,Y),\\[1.5mm]
		\Omega(t,X,Y)=\nabla^{\bot}\cdot \psi(t,X+tY,Y).
	\end{cases}
\end{align}
We denote the symbol associated to $-\lapt$ as 
$$p(t,k,\eta)=k^2+(\eta-kt)^2.$$
Moreover 
$$\p_t p(t,k,\eta)=2k(kt-\eta)$$ 
is the symbol associated to the operator $2\pt_{XY}$. 

For convenience, we further define
\begin{align}
	\mathcal{B}:=(\mathcal{B}_1,\mathcal{B}_2)^{T}:=\bigg(\Xlap A- \Ylap \Omega ,  \Ylap A+\Xlap \Omega\bigg)^T.
\end{align}
A tedious computation on  \cref{perturbation1} by taking the Helmholtz decomposition and the above coordinate transformation yields the system of $(R_\neq,A_\neq,\Omega_\neq)$
\begin{align}\label{new variable}
	\left\{\begin{aligned}
		&\p_t R_\neq +A_\neq = -\left(AR\right)_\neq -\bigg\{ \B \cdot\nablat  R \bigg\}_\neq,\\
		&\p_t A_\neq +2\XYlap A_\neq +2\XXlap \Omega_\neq+\frac{1}{M^2}\lapt R_\neq -(\lambda+2\mu)\lapt A_\neq \\
		&\quad\qquad=-\left\{(\pt_Y \B_2)^2\right\}_\neq-\left\{(\p_X\B_1)^2\right\}_\neq -2\left\{\pt_Y\B_1 \p_X\B_{2} \right\}_\neq-\left\{ \B \cdot \nablat  A \right\}_\neq\\
		&\qquad\qquad -\tilde{\dv}\left\{\frac{R}{R+1}(\lambda+2\mu)\nablat  A+\frac{R}{R+1}\mu\nablat ^{\bot} \Omega\right\}_\neq\\
            &\qquad\qquad-\frac{1}{M^2}\tilde{\dv}\left(\frac{P'(R+1)-1}{R+1}\nablat R\right)_\neq+\frac{1}{M^2}\tilde{\dv}\left(\frac{R}{R+1}\nablat R\right)_\neq\\[1mm]
		&\p_t \Omega_\neq -\mu\lapt \Omega_\neq - A_\neq=-(A\Omega)_\neq -\left\{ \B \cdot \nablat \Omega  \right\}_\neq\\
            &\qquad\qquad-\nablat^{\bot}  \cdot \left\{\frac{R}{R+1}(\lambda+2\mu)\nablat  A + \frac{R}{R+1} \mu \nablat^{\bot}  \Omega \right\}_\neq.
	\end{aligned}\right.
\end{align}

\subsection{Fourier multipliers}
    Motivated by \cite{ADM2021,BGM2017ann,ZZZ2022}, we use $m_1$ to compensate for the transient slow-down of the enhanced dissipation near the critical times and $m_2$ to balance the growth due to the 
    linear coupling between $A_\neq$ and $\Omega_\neq$. That is, 
\begin{equation}\label{m1}
	\begin{cases}
		\p_t m_1(t,k,\eta)=2\mu^{\frac{1}{3}}\left(\mu^{\frac{2}{3}}(t-\frac{\eta}{k})^2+1\right)^{-1}m_1(t,k,\eta),\\[3mm]
		\p_t m_2(t,k,\eta)=N \frac{k^2}{p}m_2(t,k,\eta),\\[2mm]
		m_1(0,k,\eta)=\exp(2\arctan(-\mu^{\frac{1}{3}}\frac{\eta}{k})),\\[2mm]
		m_2(0,k,\eta)=\exp(N \arctan(-\frac{\eta}{k})),
	\end{cases}
\end{equation}
which are explicitly given by
\begin{equation}\label{m1ini}
	\begin{cases}
		m_1(t,k,\eta)=\exp(2\arctan(\mu^{\frac{1}{3}}(t-\frac{\eta}{k}))),\\[2mm]
		m_2(t,k,\eta)=\exp(N \arctan(t-\frac{\eta}{k})).
	\end{cases}
\end{equation}
where $N=C_0\max\{1,M^4\}$, $C_0>0$ is a suitable large constant independent of $\mu$. For convenience, we also define $\tilde{C}=C_0'\max \left\{1,M^4\right\}$ with $10\leq  C_0'\leq \frac{1}{100}C_0$.
These multipliers are introduced since they enjoy the following crucial property 
\begin{align}
	\mu p+ \frac{\p_t m_1}{m_1} +\frac{\p_t m_2}{m_2} \geqslant \mu^{\frac{1}{3}}, \qquad t\geq 0, \quad k\in\mathbb{Z} \textbackslash {\{0\}},\quad \eta\in\R.
\end{align}

We further construct the multiplier $\omega$ to balance the growth generated by the term $\frac{\p_t p}{p}$ in the equation for the divergence. 
For $t\leq\frac{\eta}{k}$, we observe that $\p_t p\leq0$.  When $|t-\frac{\eta}{k}|\geqslant \vartheta \mu^{-\frac{1}{3}}$, we have for any $\vartheta>0$ that 
\begin{align}
		\mu p(t,k,\eta)\geqslant \vartheta^2 \mu^{\frac{1}{3}},\quad 
		\frac{\p_t p}{p}(t,k,\eta)\leqslant \frac{2}{\sqrt {1+(t-\frac{\eta}{k})^2}}\leqslant 2\vartheta^{-1}\mu^{\frac{1}{3}}.
\end{align}
However, we can't expect  $\frac{\p_t p}{p}\lesssim \mu p$ for $t\in[\frac{\eta}{k},\frac{\eta}{k}+\vartheta \mu^{-\frac{1}{3}}]$.
So we construct the third Fourier multiplier for a fixed $\vartheta>1000$, that
\begin{align}\label{omega}
		\omega(t,k,\eta)=\begin{cases}
			1,\qquad \qquad \qquad  &\mathrm{if}\   \eta k > 0\   \text{and}\  0\leqslant t\leqslant \frac{\eta}{k},\\[2mm]
			1+\vartheta^2\mu^{-\frac{2}{3}}, &\mathrm{if}\ \  \eta k\leq0.\  \ |\frac{\eta}{k}|\geqslant \vartheta\mu^{-\frac{1}{3}}\ \  \mathrm{and}\ \  t\geqslant 0,\\[2mm]
			{p(t,k,\eta)}{k^{-2}}, &\mathrm{if}\ \   \frac{\eta}{k}\leqslant t\leqslant \frac{\eta}{k}+\vartheta\mu^{-\frac{1}{3}},\\[2mm]
			1+\vartheta^2\mu^{-\frac{2}{3}}, &\text{else},
		\end{cases}
\end{align}
which satisfies 
 \begin{align}\label{w}
		(\p_t \omega)(t,k,\eta)&=\begin{cases}
			0\qquad \qquad \qquad \quad  &if  \quad t\notin[\frac{\eta}{k}, \frac{\eta}{k}+\vartheta\mu^{-\frac{1}{3}}],\\[2mm]
			(\frac{\p_t p}{p}\omega)(t,k,\eta) &if \quad  t\in [\frac{\eta}{k}, \frac{\eta}{k}+\vartheta\mu^{-\frac{1}{3}}].
	\end{cases}
\end{align}
Then we have
\begin{Prop}\label{multiplier prop}
	Let $\omega$, $m_1$ and $m_2$ be the Fourier multipliers defined in \cref{m1,omega} respectively. Then, for any $t \geq 0, \eta \in \mathbb{R}$ and $k \in \mathbb{Z} \backslash\{0\}$, the following inequalities hold:
	$$
	\begin{aligned}
		1 \leq \omega(t, k, \eta) & \leq 1+\vartheta^2 \mu^{-\frac{2}{3}}, \qquad\left(\omega p^{-1}\right)(t, k, \eta)  \leq \frac{1}{k^2} .
	\end{aligned}
	$$
	In addition, for any $\vartheta$ satisfying $\vartheta_1:=\max \left\{2\left(\vartheta\left(\vartheta^2-1\right)\right)^{-1}, 4 \vartheta^{-1}\right\}< 1$, it holds that
	$$
	\begin{aligned}
		& \left(\delta_\vartheta\left(\frac{\partial_t m_1}{m_1}+\frac{\partial_t m_2}{m_2}+\mu p\right)+\frac{\partial_t \omega}{\omega}-\frac{\partial_t p}{p}\right)(t, k, \eta) \geq \delta_\vartheta \mu^{\frac{1}{3}}, \\
		& \left(\delta_\vartheta\left(\frac{\partial_t m_1}{m_1}+\frac{\partial_t m_2}{m_2}+\mu^{\frac{1}{3}}\right)+\frac{\partial_t \omega}{\omega}-\frac{\partial_t p}{p}\right)(t, k, \eta) \geq \frac{\delta_\vartheta}{2} \mu^{\frac{1}{3}},
	\end{aligned}
	$$
	where $\delta_\vartheta$ is any constant located in $(\vartheta_1,1)$.
\end{Prop}
\begin{proof}
The proof is omitted since it is similar to \cite{ADM2021}.
\end{proof}

In order to control nonlinear terms, we need to obtain the estimates of the commutator for multipliers. We use the following notations,
\begin{align}
	\mathcal{T}^{\alpha}:=&\bigg(\left(m_1^{-1}m_2^{-1} \omega^{-\frac{2n+3}{4}} p^{\frac{\alpha}{2}}\right)(t,k,\eta)-\left(m_1^{-1}m_2^{-1} \omega^{-\frac{2n+3}{4}} p^{\frac{\alpha}{2}}\right)(t,k,\eta-\xi)\bigg),\\
	\mathcal{P}^{\alpha}:=&\bigg(\left(\frac{\p_t p}{p}m_1^{-1}m_2^{-1} \omega^{-\frac{2n+3}{4}} p^{\frac{\alpha}{2}}\right)(t,k,\eta)-\left(\frac{\p_t p}{p}m_1^{-1}m_2^{-1} \omega^{-\frac{2n+3}{4}} p^{\frac{\alpha}{2}}\right)(t,k,\eta-\xi)\bigg),
\end{align}
then the following lemma holds.
\begin{Lem}\label{com}
	For the multipliers defined in \cref{m1}, \cref{m1ini}, \cref{omega}, it holds that,
	\begin{align*}
		\mathbf{1)}\quad& \abs{k\mathcal{T}^{\alpha}}\lesssim \begin{cases}
			\abs{\xi},&\alpha=0,\\
			\abs{\xi}\la\xi\ra^{\alpha-1}p^{\frac{\alpha}{2}}(t,k,\eta-\xi),&\alpha\geq 1,
		\end{cases}\\[2mm]
		\mathbf{2)}\quad&\abs{\eta-\xi-kt}\abs{\mathcal{T}^{n+1}}\lesssim\abs{\xi}\la\xi\ra^{n}p^{\frac{1+n}{2}}(t,k,\eta-\xi),\\[2mm]
		\mathbf{3)}\quad& \abs{k\mathcal{P}^{\alpha}}\lesssim\begin{cases}
			\abs{\xi},&\alpha=0,\\
			\abs{\xi}\la\xi\ra^{\alpha-1}p^{\frac{\alpha}{2}}(t,k,\eta-\xi),&\alpha\geq 1.
		\end{cases}
	\end{align*}
\end{Lem}
\begin{proof}
	We only calculate $\mathbf{2)}$ since the others are similar. We denote the commutator in $\mathbf{2)}$ as
	\begin{align*}
	    &|\eta-\xi-kt|\big(m_1^{-1}m_2^{-1} \omega^{-\frac{2n+3}{4}} p^{\frac{n+1}{2}}(t,k,\eta)-m_1^{-1}m_2^{-1}\omega^{-\frac{2n+3}{4}} p^{\frac{n+1}{2}}(t,k,\eta-\xi)\big)\\
     &\qquad\qquad=\mathcal{Q}_1+\mathcal{Q}_2+\mathcal{Q}_3+\mathcal{Q}_4, 
     \end{align*} 
	where
	\begin{align*}
		&\mathcal{Q}_1=|\eta-\xi-kt|(m_1^{-1}(t,k,\eta)-m_1^{-1}(t,k,\eta-\xi))(m_2^{-1}\omega^{-\frac{2n+3}{4}})(t,k,\eta)p^{\frac{n+1}{2}}(t,k,\eta-\xi),\\
		&\mathcal{Q}_2=|\eta-\xi-kt|(m_2^{-1}(t,k,\eta)-m_2^{-1}(t,k,\eta-\xi))(m_1^{-1}p^{\frac{n+1}{2}})(t,k,\eta-\xi)\omega^{-\frac{2n+3}{4}}(t,k,\eta),\\
		&\mathcal{Q}_3=|\eta-\xi-kt|(p^{\frac{n+1}{2}}(t,k,\eta)-p^{\frac{n+1}{2}}(t,k,\eta-\xi))(m_1^{-1}m_2^{-1}\omega^{-\frac{2n+3}{4}})(t,k,\eta),\\
		&\mathcal{Q}_4=|\eta-\xi-kt|(\omega^{-\frac{2n+3}{4}}(t,k,\eta)-\omega^{-\frac{2n+3}{4}}(t,k,\eta-\xi))(m_1^{-1}m_2^{-1}p^{\frac{n+1}{2}})(t,k,\eta-\xi).
	\end{align*}
	For $\mathcal{Q}_1$ and $\mathcal{Q}_2$,  we obtain directly from the mean value theorem that
	$$\abs{\mathcal{Q}_1}+\abs{\mathcal{Q}_2} \lesssim   (|\xi|+|\xi|^2) p^{\frac{n+1}{2}}(t,k,\eta-\xi). $$
	For $\mathcal{Q}_3$, we have 
	$$ p^{\frac{n+1}{2}}(t,k,\eta)-p^{\frac{n+1}{2}}(t,k,\eta-\xi)=\frac{n+1}{2}p^{\frac{n-1}{2}}(t,k,\eta-\theta_1\xi)p'_\eta(t,k,\eta-\theta_1\xi)\xi, $$ 
	where $0<\theta_1<1$. Since $p^{\frac{n-1}{2}}(t,k,\eta-\theta_1\xi)p'_\eta(t,k,\eta-\theta_1\xi)\lesssim \la \xi\ra^n p^{\frac{n}{2}}(t,k,\eta-\xi)$, we control $\mathcal{Q}_3$ as
	$$ \abs{\mathcal{Q}_3} \lesssim    |\xi| \la \xi\ra^n p^{\frac{n+1}{2}}(t,k,\eta-\xi). $$
	
	For $\mathcal{Q}_4$, we only calculate the case of $k>0,\xi>0$, and the other scenarios are similar. To estimate the commutator for $\omega(t,k,\eta)-\omega(t,k,\eta-\xi)$, we consider $\omega(t,k,\zeta)$ as a function of $\zeta$,
   then
   \begin{align}\label{omega11}
		\omega(t,k,\zeta)=\begin{cases}
			1,\qquad \qquad \qquad  &\mathrm{if}\   \zeta  > kt,\\[1mm]
			p(t,k,\zeta)k^{-2}, &\mathrm{if}\ kt-k\vartheta\mu^{-\frac{1}{3}}\leq \zeta \leq kt,\\[1mm]
			1+\vartheta^2\mu^{-\frac{2}{3}}, &\mathrm{if}\ \zeta \leq kt-k\vartheta \mu^{-\frac{1}{3}}.
		\end{cases}
\end{align}
	Since $\omega(t,k,\zeta)$ is not differentiable at $\zeta=\eta_0:= kt-k\vartheta \mu^{-\frac{1}{3}}$, the mean value theorem cannot be directly applied. We divide the proof into four cases by the values of $\eta$ and $\eta-\xi$.
	\begin{flushleft}
		\textbf{Case 1: $\eta>\eta-\xi>\eta_0$.}
	\end{flushleft}
	Because of $\eta>\zeta>\eta-\xi>\eta_0$, $\omega(t,k,\zeta)$ is differentiable. From the mean value theorem, there exists $0<\theta_2<1$ such that
\begin{align}
&\abs{\eta-kt-\xi} \abs{\omega^{-\frac{2n+3}{4}}(t,k,\eta)-\omega^{-\frac{2n+3}{4}}(t,k,\eta-\xi)} \nonumber\\ 
=&\abs{\eta-kt-\xi}\abs{\p_\eta \omega(t,k,\eta-\theta_2\xi) \omega^{-\frac{2n+7}{4}}(t,k,\eta-\theta_2\xi)\xi}\\ 
\lesssim& \big(\abs{\eta-kt-\theta_2\xi}+\abs{\xi} \big)\abs{\frac{\p_\eta p}{k^2}}(t,k,\eta-\theta_2\xi)\left(\frac{k^2}{p} \right)^{\frac{2n+7}{4}}(t,k,\eta-\theta_2\xi)\abs{\xi} \nonumber\\
\lesssim& \abs{\xi}+\abs{\xi}^2.\nonumber
\end{align}
\begin{flushleft}
	\textbf{Case 2: $\eta>kt>\eta_0>\eta-\xi$.}
\end{flushleft}
In this case, $\omega(t,k,\eta)=1$ and $\omega(t,k,\eta-\xi)=1+\vartheta^2 \mu^{-\frac{2}{3}}$.
Since $\eta-\xi-kt \leq -k\vartheta \mu^{-\frac{1}{3}}<0$ and $kt-\eta<0$, we deduce that 
$$\abs{\eta-\xi-kt}=kt-\eta+\xi \leq \xi. $$
Then, there holds
$$\abs{\eta-\xi-kt}\abs{\omega^{-\frac{2n+3}{4}}(t,k,\eta)-\omega^{-\frac{2n+3}{4}}(t,k,\eta-\xi)}\leq \abs{\eta-kt-\xi} \leq \abs{\xi}. $$ 
\begin{flushleft}
	\textbf{Case 3: $kt>\eta>\eta_0>\eta-\xi$.}
\end{flushleft}
In this case, $\omega(t,k,\eta)=\frac{p}{k^2}$ and $\omega(t,k,\eta-\xi)=1+\vartheta^2\mu^{-\frac{2}{3}}.$ By \cref{multiplier prop}, there holds $\omega(t,k,\eta)<1+\vartheta^2\mu^{-\frac{2}{3}}$ and $\omega(t,k,\eta-\xi)\leq\frac{p}{k^2}(t,k,\eta-\xi)$. So we obtain that
\begin{align}
&\abs{\omega^{-\frac{2n+3}{4}}(t,k,\eta)-\omega^{-\frac{2n+3}{4}}(t,k,\eta-\xi)}=\omega^{-\frac{2n+3}{4}}(t,k,\eta)-\omega^{-\frac{2n+3}{4}}(t,k,\eta-\xi)\\ \nonumber
&\qquad \qquad \qquad \lesssim \left(\frac{p}{k^2}\right)^{-\frac{2n+3}{4}}(t,k,\eta)-\left(\frac{p}{k^2}\right)^{-\frac{2n+3}{4}}(t,k,\eta-\xi),
\end{align}
which implies that
$$\abs{\eta-kt-\xi}\abs{\omega^{-\frac{2n+3}{4}}(t,k,\eta)-\omega^{-\frac{2n+3}{4}}(t,k,\eta-\xi)}\lesssim \abs{\xi}+\abs{\xi}^2. $$
\begin{flushleft}
	\textbf{Case 4: $\eta_0>\eta>\eta-\xi$.}
\end{flushleft}
It is straightforward to check that  $\abs{\omega^{-\frac{2n+3}{4}}(t,k,\eta)-\omega^{-\frac{2n+3}{4}}(t,k,\eta-\xi)}=0.$

Combining all of the above cases, we  prove $ \abs{\mathcal{Q}_4} \lesssim   (|\xi|+|\xi|^2) p^{\frac{n+1}{2}}(t,k,\eta-\xi). $
Then we can show $\mathbf{2)}$ from the estimates of $\mathcal{Q}_1$, $\mathcal{Q}_2$, $\mathcal{Q}_3$ and $\mathcal{Q}_4$. And the proof of \cref{com} is completed. 
\end{proof}

\subsection{Weighted variables}
In what follows we will use the following shorthand $Q^n:=m_1^{-1}m_2^{-1}\omega^{-\frac{3+2n}{4}}p^{\frac{n}{2}}$. Motivated by \cite{ADM2021}, we introduce the following new variables as,
\begin{align}
	\begin{aligned}
		&Z_1^n=\frac{1}{M}m_1^{-1}m_2^{-1}\omega^{-\frac{3+2n}{4}}p^{\frac{1+n}{2}}R_\neq,\qquad\qquad\qquad Z_2^n=m_1^{-1}m_2^{-1}\omega^{-\frac{3+2n}{4}}p^{\frac{n}{2}}A_\neq,\\
		&Z_3^n=m_1^{-1}m_2^{-1}\omega^{-\frac{3+2n}{4}}p^{\frac{n}{2}}(R_\neq+\Omega_\neq-\mu M^2 A_\neq).
	\end{aligned}
\end{align}
From \cref{new variable} and complicated computations, $Z_1^n, Z_2^n, Z_3^n$ satisfy
\begin{equation}\label{non-zero n-th}
	\begin{cases}
		\p_t Z_1^n=&-\Big(\frac{\p_t m_1}{m_1}+\frac{\p_t m_2}{m_2} +\frac{3+2n}{4}(\frac{\p_t \omega}{\omega}-\frac{\p_t p}{p})+\frac{1}{4}\frac{\p_t p}{p}\Big)Z_1^n -\frac{1}{M}p^{\frac{1}{2}}Z_2^n+\mathcal{Z}_1^n\\[2mm]
		\p_t Z_2^n=&-\Big(\frac{\p_t m_1}{m_1}+\frac{\p_t m_2}{m_2} +\frac{3+2n}{4}(\frac{\p_t \omega}{\omega}-\frac{\p_t p}{p})+(\lambda+2\mu)p\Big)Z_2^n+\frac{1}{4}\frac{\p_t p}{p}Z_2^n\\[2mm]
		&+\frac{1}{M} p^{\frac{1}{2}}Z_1^n-2M \p_X^2 p^{-\frac{3}{2}}Z_1^n + 2 \p_X^2 p^{-1} Z_3^n + 2\mu M^2 \p_X^2 p^{-1} Z_2^n+\mathcal{Z}_2^n\\[2mm]
		\p_t Z_3=&-\Big(\frac{\p_t m_1}{m_1}+\frac{\p_t m_2}{m_2}  +\mu p\Big) Z_3^n - \frac{3+2n}{4}\frac{\p_t\omega}{\omega}Z_3^n +\mu(\lambda+\mu)M^2 p Z_2^n+\mathcal{Z}_3^n\\[2mm]
		&-\mu M^2 \frac{\p_t p}{p}Z_2^n+2\mu M^3 \p_X^2 p^{-\frac{3}{2}}  Z_1^n - 2\mu M^2 \p_X^2 p^{-1}  (Z_3^n + \mu M^2  Z_2^n),
	\end{cases}
\end{equation}
where
\begin{align}
	\mathcal{Z}_1^n:=    &-\frac{1}{M}Q^{n} p^{\frac{1}{2}}\left\{ AR+ \B \cdot \nablat  R \right\}_\neq,\\[1mm] \nonumber
	\mathcal{Z}_2^n:=    &-Q^{n} \left\{(\p_X\B_1 )^2\right\}_\neq-Q^{n} \left\{(\pt_Y\B_2 )^2\right\}_\neq -2Q^{n} \left\{\pt_Y\B_1 \p_X\B_{2} \right\}_\neq -Q^{n} \left\{ \B \cdot \nablat  A \right\}_\neq\\\nonumber
	& -Q^{n} \tilde{\dv}\left\{\frac{R}{R+1}(\lambda+2\mu)\nablat  A+\frac{R}{R+1}\mu\nablat ^{\bot} \Omega\right\}_\neq+\frac{1}{M^2}Q^{n} \tilde{\dv}\bigg(\frac{R}{R+1}\nablat R\bigg)_\neq\\
	& -\frac{1}{M^2}Q^{n} \p_X\bigg(\frac{P'(R+1)-1}{R+1}\p_X R\bigg)_\neq -\frac{1}{M^2}Q^{n} \pt_Y\bigg(\frac{P'(R+1)-1}{R+1}\pt_YR\bigg)_\neq,\\\nonumber
	\mathcal{Z}_3^n:=& - Q^{n}  (AR)_\neq-Q^{n} (A\Omega)_\neq+2\mu M^2 Q^{n} \left\{\pt_Y\B_1 \p_X\B_{2} \right\}_\neq\\
	& +\mu M^2 Q^{n} \tilde{\dv}\left\{\frac{R}{R+1}(\lambda+2\mu)\nablat  A+\frac{R}{R+1}\mu\nablat ^{\bot} \Omega\right\}_\neq -\mu Q^{n} \tilde{\dv}\left(\frac{R}{R+1}\nablat R\right)_\neq\\[2mm] \nonumber
	&+\mu Q^{n} \tilde{\dv}\left(\frac{P'(R+1)-1}{R+1}\nablat R\right)_\neq-Q^{n} \nablat^{\bot}  \cdot \left\{\frac{R}{R+1}(\lambda+2\mu)\nablat  A + \frac{R}{R+1} \mu\nablat^{\bot}\Omega \right\}_\neq\\[2mm] \nonumber
	&+\mu M^2 Q^{n} \left\{(\pt_Y\B_2)^2\right\}_\neq + \mu M^2 Q^{n} \left\{(\p_X\B_1)^2\right\}_\neq- Q^{n} \left\{ \B \cdot \nablat  (R+\Omega-\mu M^2 A) \right\}_\neq. \nonumber
\end{align}
Let $\gamma=\frac{M\mu^{\frac{1}{3}}}{4} \leqslant \frac{1}{200}$ and consider the following energy functional,
\begin{align}
	E^n(t)=&\frac{1}{2}\norm{\left(1+M^2 \frac{(\p_t p)^2}{{p}^3}\right)^{\frac{1}{2}}Z_1^n}^2+\frac{1}{2}\norm{Z_2^n}^2+\frac{1}{2}\norm{Z_3^n}^2\nonumber\\
	&\qquad\qquad\qquad\qquad+\frac{M}{4}\bigg\langle \frac{\p_t p}{p^{\frac{3}{2}}}Z_1^n,Z_2^n\bigg\rangle-\gamma \bigg\langle p^{-\frac{1}{2}}Z_1^n,Z_2^n\bigg\rangle.
\end{align}
It is easy to verify that $E^n(t)$ is coercive, namely 
$$E^n(t)\thickapprox\norm{\left(1+M^2 \frac{(\p_t p)^2}{{p}^3}\right)^{\frac{1}{2}}Z_1^n}^2+\norm{Z_2^n}^2+\norm{Z_3^n}^2. $$

Before proving \cref{ed}, we will establish the following two inequalities.
\begin{Lem}\label{L1}
	The following inequalities hold for $f\in H^s(\Omega)$, $s\geq 2$,
	\begin{align}
		\norm{\hat{f}_\neq}_{L^1}\lesssim \norm{\lapt f_\neq} \text{\qquad and\qquad} \norm{\hat{\mathring{f}}}_{L^1}\lesssim \norm{\mathring{f}}_{H^1}.
	\end{align}
\end{Lem}
\begin{proof}
	We prove $\norm{\hat{f}_\neq}_{L^1}\lesssim \norm{\lapt f_\neq}$ at first.
	By direct calculations, one has
	\begin{align*}
		\norm{\hat{f}_\neq}_{L^1}&=\sum_{k\neq0}\int_{\R}\abs{\hat{f}_\neq}(t, k, \eta)d\eta\\
		&=\sum_{k\neq0}\int_{\R}\frac{1}{k^2+(kt-\eta)^2} (k^2+(kt-\eta)^2)\abs{\hat{f}_\neq}(t, k, \eta)d\eta\\
		&\lesssim \Big(\sum_{k\neq0}\frac{1}{k^4}\int_{\R}\frac{d\eta}{(1+(\frac{\eta}{k}-t)^2)^2}\Big)^{\frac{1}{2}}\Big(\sum_{k\neq0}\int_{\R}\abs{p \hat{f}_\neq}^2 (t, k, \eta)d\eta\Big)^{\frac{1}{2}}
		\lesssim \norm{\lapt f_\neq}.
	\end{align*}
	The second inequality can be proved similarly,
	\begin{align*}
		\norm{\hat{\mathring{f}}}_{L^1}&=\int_{\R}\abs{\hat{\mathring{f}}}(t, \eta)d\eta=\int_{\R}\la\eta\ra^{-1}\la\eta\ra\abs{\hat{\mathring{f}}}(t, \eta)d\eta\\
		&\lesssim (\int_{\R}\la\eta\ra^{-2}d\eta)^{\frac{1}{2}}\Big(\int_{\R}|\la\eta\ra \hat{\mathring{f}}|^2 d\eta\Big)^{\frac{1}{2}}\lesssim \norm{\mathring{f}}_{H^1}.
	\end{align*}
	Then  \cref{L1} is proved.
\end{proof}

\subsection{Estimate on $\frac{d}{dt}\norm{Z_1^n,Z_2^n,Z_3^n}^2$}
In this section, we give a basic estimate of $\norm{Z_1^n},\norm{Z_2^n}$ and $\norm{Z_3^n}$.
\begin{Lem}\label{Z_1^n+Z_2^n+Z_3^n}
	Under the same assumptions of \cref{MT}, for $n=0,1$, it holds that
	\begin{align*}
		&\frac{1}{2}\frac{d}{dt} \bigg(\norm{\left(1+M^2\frac{(\p_t p)^2}{{p}^3}\right)^{\frac{1}{2}}Z_1^n}^2+\norm{Z_2^n}^2+\norm{Z_3^n}^2 \bigg)\\
		\leq & -\norm{\sqrt{\frac{63}{64}\left(\frac{\p_t m_1}{m_1}+\frac{\p_t m_2}{m_2}\right)+\frac{2n+3}{4}\left(\frac{\p_t \omega}{\omega}-\frac{\p_t p}{p}\right)}Z_1^n}^2 \\
		&-\norm{\sqrt{\frac{31}{64}\left(\frac{\p_t m_1}{m_1}+\frac{\p_t m_2}{m_2}+(\lambda+2\mu)p\right)+\frac{2n+3}{4} \left(\frac{\p_t \omega}{\omega}-\frac{\p_t p}{p} \right)}Z_2^n}^2\\
		&-\frac{15}{64}\norm{\sqrt{\frac{\p_t m_1}{m_1}+\frac{\p_t m_2}{m_2}+\mu p}Z_3^n}^2+\frac{1}{4}\bigg\la \frac{\p_t p}{p} Z_2^n , Z_2^n\bigg\ra- \frac{1}{4}\bigg\la \frac{\p_t p}{p} Z_1^n , Z_1^n\bigg\ra\\
		&+C_M\mu^{-\frac{2n+3}{6}} \Big(\norm{\nabla\psi}_{H^{n+3}}+\norm{\psi_2}+\norm{\nabla\phi}_{H^{n+3}}\Big) \Big(\norm{Z_1^n}^2+\norm{Z_2^n}^2+\norm{Z_3^n}^2 \Big)\\
		&+C_M\mu^{-\frac{2n-3}{6}} \Big(\norm{\nabla\psi}_{H^{n+2}}+\norm{\phi}_{H^{n+3}}\Big) \Big(\norm{Z_1^n}^2+\norm{p^{\frac{1}{2}}Z_2^n}^2+\norm{p^{\frac{1}{2}}Z_3^n}^2 \Big)\\
		&+C_M \mu^{-2-\frac{2n}{3}}\left(\norm{\phi}_{H^{n+2}}^2+\norm{\nabla\psi}_{H^{n+1}}^2+\norm{\psi_2}^2\right) \left(\norm{Z_1^n}^2+\norm{Z_2^n}^2+\norm{Z_3^n}^2\right).      
	\end{align*}
\end{Lem}

\begin{proof}
	{\bf{$L^2$ estimate on $Z_1^n$}}. An energy estimate gives
	\begin{align}\label{Z_1^n}
			\frac{1}{2}\frac{d}{dt}\norm{Z_1^n}^2=&-\norm{\sqrt{\frac{\p_t m_1}{m_1}+\frac{\p_t m_2}{m_2} +\frac{3+2n}{4}\left(\frac{\p_t \omega}{\omega}-\frac{\p_t p}{p} \right)}Z_1^n}^2\\
			&-\frac{1}{M}\bigg\la p^{\frac{1}{2}} Z_1^n , Z_2^n \bigg\ra - \frac{1}{4}\bigg\la \frac{\p_t p}{p}Z_1^n , Z_1^n\bigg\ra  +\bigg\la \mathcal{Z}_1^n, Z_1^n\bigg\ra,\nonumber
	\end{align}
	where
	\begin{align*}
		\bigg\la \mathcal{Z}_1^n ,  Z_1^n\bigg\ra=&-\frac{1}{M} \bigg\la Q^n p^{\frac{1}{2}}(A R)_\neq ,  Z_1^n \bigg\ra- \frac{1}{M^2} \bigg\la Q^n p^{\frac{1}{2}}\big(\p_X R \B_1\big)_\neq ,  Q^n p^{\frac{1}{2}} R_\neq \bigg\ra\\
		&-\frac{1}{M^2} \bigg\la Q^n p^{\frac{1}{2}}\big(\pt_Y R \B_2\big)_\neq ,  Q^n p^{\frac{1}{2}} R_\neq \bigg\ra\\
		:=&NL_{1a}+NL_{1b}+NL_{1c}.
	\end{align*}
	\cref{L1} implies
	\begin{align}\label{NL_11}
			&\abs{NL_{1a}}=\frac{1}{M}\abs{\bigg\la Z_1^n , Q^n p^{\frac{1}{2}}(A_\neq R_\neq)_\neq +Q^n p^{\frac{1}{2}}(\mathring{A}R_\neq)+ Q^n p^{\frac{1}{2}}(A_\neq \mathring{R})\bigg\ra} \notag\\
			&\;\lesssim \Big(\left\|Q^n p^{\frac{1}{2}}(A_\neq R_\neq)_\neq\right\|+\norm{Q^n p^{\frac{1}{2}}(\mathring{A}R_\neq)}+\norm{Q^n p^{\frac{1}{2}}(A_\neq \mathring{R})}\Big)\norm{Z_1^n}  \notag \\
			&\;\lesssim \Big(\norm{p^{\frac{n+1}{2}}\hat{A}_\neq \ast p^{\frac{n+1}{2}}\hat{R}_\neq} +\norm{\la\eta\ra^{n+1} \hat{\mathring{A}} \ast p^{\frac{n+1}{2}}\hat{R}_\neq}+\norm{\la\eta\ra^{n+1} \hat{\mathring{R}} \ast p^{\frac{n+1}{2}} \hat{A}_\neq }\Big)\norm{Z_1^n}   \\
			&\;\lesssim \mu^{-\frac{2n+3}{6}}\Big( \norm{p^{\frac{n+1}{2}}\hat{A}_\neq}_{L^1}+\norm{\la \eta\ra^{n+1} \hat{\mathring{A}}}_{L^1}\Big)\norm{Z_1^n}^2+\mu^{-\frac{2n+3}{6}}\norm{\la \eta\ra^{n+1} \hat{\mathring{R}}}_{L^1}\norm{Z_1^n}\norm{p^{\frac{1}{2}}Z_2^n}\notag\\
			&\;\lesssim \mu^{-\frac{2n+3}{6}}\Big( \norm{p^{\frac{n+3}{2}}\hat{A}_\neq}+\norm{\hat{\mathring{A}}}_{H^{n+2}}\Big)\norm{Z_1^n}^2+\mu^{-\frac{2n+3}{6}}\norm{\hat{\mathring{\phi}}}_{H^{n+2}}\norm{Z_1^n}\norm{p^{\frac{1}{2}}Z_2^n}  \notag\\
			&\;\leq C_M\left(\mu^{-\frac{2n+3}{6}}\norm{\nabla \psi}_{H^{n+3}}+\mu^{-2-\frac{2n}{3}}\norm{\phi}^2_{H^{n+2}} \right)\norm{Z_1^n}^2+\frac{\mu}{10000}\norm{p^{\frac{1}{2}}Z_2^n}^2,\notag
	\end{align}
	where we have used the fact that
	$$ p^{\frac{n+1}{2}}(t,k,\eta)\lesssim \min\left\{p^{\frac{n+1}{2}}(t,\;k-l,\;\eta-\xi)p^{\frac{n+1}{2}}(t,\;l,\;\xi)\;,\;\la\xi\ra^{n+1} p^{\frac{n+1}{2}}(t,k,\eta-\xi) \right\}.$$
	$NL_{1b}$ can be split into
	\begin{align*}
		NL_{1b}=-\frac{1}{M^2}\bigg\la Q^n p^{\frac{1}{2}}(\p_X R_\neq \mathring{\psi}_1) , Q^n p^{\frac{1}{2}}R_\neq \bigg\ra-\frac{1}{M} \bigg\la Q^n p^{\frac{1}{2}} \big(\p_X R_\neq\mathcal{B}_{1\neq}\big)_\neq ,  Z_1^n \bigg\ra.
	\end{align*}
	Then we estimate these terms item by item. For  $NL_{1b}$, by \cref{com} and \ref{L1}, we deduce that
	\begin{align}\label{NL_1211}
			& \frac{1}{M^2}\bigg\la Q^n p^{\frac{1}{2}}(\p_X R_\neq \mathring{\psi} _1) , Q^n p^{\frac{1}{2}}R_\neq\bigg\ra\\ \nonumber
			=& \frac{1}{M^2}\bigg\la \p_X(Q^n p^{\frac{1}{2}} R_\neq) \mathring{\psi} _1 , Q^n p^{\frac{1}{2}}R_\neq \bigg\ra+ \frac{1}{M^2}\bigg\la \mathcal{T}^{n+1}(\p_X R_\neq \mathring{\psi} _1) , Q^n p^{\frac{1}{2}}R_\neq\bigg\ra \\ \nonumber
			\lesssim  &   \bigg\la \abs{p^{\frac{n+1}{2}} \hat{R}_\neq} \ast \abs{\eta \la\eta\ra^{n} \hat{\mathring{\psi}} _1} , \abs{Q^n p^{\frac{1}{2}}\hat{R}_\neq} \bigg\ra \lesssim  \mu^{-\frac{2n+3}{6}}\norm{\p_y \mathring{\psi} _1}_{H^{n+1}}\norm{Z_1^n}^2, \nonumber
	\end{align}	
	where we have used the fact that $\bigg\la \p_X(Q^n p^{\frac{1}{2}} R_\neq) \mathring{\psi} _1 , Q^n p^{\frac{1}{2}}R_\neq \bigg\ra=0$.
	The other term in $NL_{1b}$ can be treated by using \cref{L1} as
	\begin{align}\label{NL_122}
			\frac{1}{M}\bigg\la Q^n p^{\frac{1}{2}} \big(\p_X R_\neq \mathcal{B}_{1\neq}\big)_\neq ,  Z_1^n\bigg\ra \lesssim \mu^{-\frac{2n+3}{6}}\|\nabla \phi\|_{H^{n+3}}\left(\norm{Z_1^n}^2+\norm{Z_2^n}^2+\norm{Z_3^n}^2\right).
	\end{align}
	Decompose $NL_{1c}$ by frequency:
	\begin{align*}
		NL_{1c}=&-\frac{1}{M^2}\bigg\la Q^n p^{\frac{1}{2}}(\pt_Y R_\neq \mathring{\psi} _2) , Q^n p^{\frac{1}{2}}R_\neq\bigg\ra -  \frac{1}{M}\bigg\la Q^n p^{\frac{1}{2}}\big(\p_y\mathring{\phi} \mathcal{B}_{2\neq} \big),  Z_1^n\bigg\ra\\
		 &- \frac{1}{M}\bigg\la Q^n p^{\frac{1}{2}}\big(\pt_Y R_\neq \mathcal{B}_{2\neq} \big)_\neq , Z_1^n\bigg\ra.
	\end{align*}
	We handle these terms one by one. We have from \cref{com} and \ref{L1} that 
	\begin{align}\label{NL_1311}
			& \frac{1}{M^2}\bigg\la Q^n p^{\frac{1}{2}}(\pt_Y R_\neq \mathring{\psi} _2) , Q^n p^{\frac{1}{2}}R_\neq\bigg\ra\\\nonumber
			=& \frac{1}{M^2}\bigg\la \pt_Y(Q^n p^{\frac{1}{2}} R_\neq) \mathring{\psi} _2 , Q^n p^{\frac{1}{2}}R_\neq \bigg\ra+ \frac{1}{M^2}\bigg\la \mathcal{T}^{n+1}(\pt_Y R_\neq \mathring{\psi} _2) , Q^n p^{\frac{1}{2}}R_\neq\bigg\ra \\\nonumber
			\lesssim  & \bigg\la \abs{ Q^n p^{\frac{1}{2}} R_\neq \p_y\mathring{\psi} _2} , \abs{Q^n p^{\frac{1}{2}}R_\neq} \bigg\ra+  \bigg\la \abs{p^{\frac{n+1}{2}} \hat{R}_\neq} \ast \abs{\eta \la\eta\ra^{n} \hat{\mathring{\psi}} _2} , \abs{Q^n p^{\frac{1}{2}}\hat{R}_\neq} \bigg\ra\\\nonumber
			\lesssim & \mu^{-\frac{2n+3}{6}}\norm{\p_y \mathring{\psi} _2}_{H^{n+1}}\norm{Z_1^n}^2.\nonumber
	\end{align}
	The latter two terms can be directly bounded as           
	\begin{align}\label{NL_132}
			&\frac{1}{M} \bigg\la Q^n p^{\frac{1}{2}}\big(\p_y \mathring{\phi} \mathcal{B}_{2\neq}\big) ,  Z_1^n\bigg\ra \lesssim \mu^{-\frac{2n+3}{6}}\norm{\p_y \mathring{\phi} }_{H^{n+2}}\Big(\norm{Z_1^n}^2+\norm{Z_2^n}^2+\norm{Z_3^n}^2\Big),\\
			&\frac{1}{M} \bigg\la Q^n p^{\frac{1}{2}}\big(\pt_Y R_\neq \mathcal{B}_{2\neq} \big)_\neq ,  Z_1^n\bigg\ra \lesssim \mu^{-\frac{2n+3}{6}}\norm{ \nabla \phi}_{H^{n+3}}\Big(\norm{Z_1^n}^2+\norm{Z_2^n}^2+\norm{Z_3^n}^2\Big). \nonumber
	\end{align}
	Combining \cref{NL_11}-\cref{NL_132}, one can verify
	\begin{align}\label{NL_1}
		\bigg\la \mathcal{Z}_1^n ,  Z_1^n\bigg\ra\leq &C_M \mu^{-\frac{2n+3}{6}}\big(\norm{ \nabla \phi}_{H^{n+3}}+\norm{\nabla \psi}_{H^{n+3}}\big)\Big(\norm{Z_1^n}^2+\norm{Z_2^n}^2+\norm{Z_3^n}^2\Big)\\
         &\qquad+ C_M\mu^{-2-\frac{2n}{3}}\norm{\phi}^2_{H^{n+2}} \norm{Z_1^n}^2+\frac{\mu}{10000}\norm{p^{\frac{1}{2}}Z_2^n}^2. \nonumber
	\end{align}
	
	\noindent {\bf{$L^2$ estimate on $Z_2^n$}}. An energy estimate gives
	\begin{align}\label{Z_2^n}
			\frac{1}{2}&\frac{d}{dt}\norm{Z_2^n}^2=-\norm{ \sqrt{ \frac{\p_t m_1}{m_1}+\frac{\p_t m_2}{m_2} +\frac{3+2n}{4}\left(\frac{\p_t \omega}{\omega}-\frac{\p_t p}{p}\right)+(\lambda+2\mu)p}Z_2^n}^2+NL_{2a}+NL_{2b}\\
			&+NL_{2c}+\bigg\la \left(\frac{1}{M}p^{\frac{1}{2}}+\frac{2M k^2}{p^{\frac{3}{2}}} \right)Z_1^n , Z_2^n\bigg\ra -2\bigg\la\frac{k^2}{p}Z_3^n , Z_2^n\bigg\ra -2\mu M^2 \norm{\frac{k}{p^{\frac{1}{2}}}Z_2^n}^2+\frac{1}{4}\bigg\la \frac{\p_t p}{p}Z_2^n , Z_2^n \bigg\ra ,\nonumber
	\end{align}
	where
	\begin{align*}
		NL_{2a}=&-\bigg\la Q^{n}  \left\{(\p_X\B_{1})^2\right\}_\neq ,  Z_2^n\bigg\ra-\bigg\la Q^{n}  \left\{(\pt_Y\B_{2})^2\right\}_\neq ,  Z_2^n\bigg\ra-2\bigg\la Q^{n}  \left(\pt_Y\B_1 \p_X \B_2\right)_\neq ,  Z_2^n\bigg\ra,\\
		:=&NL_{2a1}+NL_{2a2}+NL_{2a3}\\
		NL_{2b}=&\frac{1}{M^2}\bigg\la Q^{n} \tilde{\dv}\left(\frac{R}{R+1} \nablat R\right)_\neq ,  Z_2^n\bigg\ra-\frac{1}{M^2}\bigg\la Q^{n} \tilde{\dv}\left(\frac{P'(R+1)-1}{R+1} \nablat  R\right)_\neq ,  Z_2^n\bigg\ra\\
		&-\bigg\la Q^{n} \tilde{\dv}\left((\lambda+2\mu)\frac{R}{R+1}\nablat A+\mu\frac{R}{R+1}\nablat ^{\bot}\Omega\right)_\neq ,  Z_2^n\bigg\ra,\\
		:=&NL_{2b1}+NL_{2b2}+NL_{2b3}\\
		NL_{2c}=&-\bigg\la Q^{n}  \left\{\B \cdot \nablat  A \right\}_\neq ,  Q^{n}  A_\neq\bigg\ra.
	\end{align*}
	 We consider $NL_{2a}$ at first. Using \cref{L1}, one has 
	\begin{align}\label{NL_211}
			NL_{2a3}\lesssim &\abs{\bigg\la Q^{n}  \left\{\pt_Y\mathcal{B}_{1\neq} \p_X \B_{2\neq}\right\}_\neq ,  Z_2^n\bigg\ra}  +\abs{\bigg\la Q^{n}  \left\{ \p_X \B_{2\neq} \p_y\mathring{\psi_{1}}\right\} , Z_2^n\bigg\ra}\\
			\lesssim& \mu^{-\frac{2n+3}{6}}\big(\norm{\nabla\psi}_{H^{n+2}}+\norm{\nabla\phi}_{H^{n+2}}\big)\Big(\norm{Z_1^n}^2+\norm{Z_2^n}^2+\norm{Z_3^n}^2\Big).\nonumber
	\end{align}
	$NL_{2a1}$ and $NL_{2a2}$ can be treated in the same way as in  \cref{NL_211}, then we obtain
	\begin{align}\label{NL_21}
		NL_{2a}\lesssim \mu^{-\frac{2n+3}{6}}\big(\norm{\nabla\psi}_{H^{n+2}}+\norm{\nabla\phi}_{H^{n+2}}\big)\Big(\norm{Z_1^n}^2+\norm{Z_2^n}^2+\norm{Z_3^n}^2\Big). 
	\end{align}
	Turning to $NL_{2b}$,  \cref{L1} yields 
	\begin{align}\label{NL_221}
			&NL_{2b1}\lesssim  \bigg|\bigg\la Q^{n} \left(\nablat R_\neq \Do \Big(\frac{R}{R+1}\Big)  \right) , \nablat Z_2^n\bigg\ra\bigg|+ \bigg|\bigg\la Q^{n}\pt_Y \left(\p_y \mathring{\phi} \Big(\frac{R}{R+1}\Big)_\neq\right) ,  Z_2^n\bigg\ra\bigg| \nonumber\\ 
			&\qquad\qquad+ \bigg|\bigg\la Q^{n} \left(\nablat R_\neq\Big(\frac{R}{R+1}\Big)_\neq\right)_\neq , \nablat Z_2^n\bigg\ra\bigg|\\ \nonumber
			& \quad\leq  C_M\left[\mu^{-2-\frac{2n}{3}}\norm{\phi}^2_{H^{n+2}} \norm{Z_1^n}^2+ \mu^{-\frac{2n+3}{6}}\norm{\nabla\phi}_{H^{n+2}}\left( \norm{Z_{1}^n}^2+\norm{Z_{2}^n}^2 \right)\right]+\frac{\mu}{60000}\norm{p^{\frac{1}{2}}Z_2^n}^2.
	\end{align}       
	$NL_{2b2}$ is treated similarly. $NL_{2b3}$ can be bounded by  
	\begin{align}\label{NL2211}
    NL_{2b3}\lesssim \mu^{-\frac{2n-3}{6}} \left( \norm{\phi}_{H^{n+2}} + \norm{\nabla \psi}_{H^{n+2}} \right)\left( \norm{Z_1^n}^2+\norm{p^{\frac{1}{2}}Z_2^n}^2+\norm{p^{\frac{1}{2}}Z_3^n}^2\right).
    \end{align}
	 Combining \cref{NL_221,NL2211}, we obtain
	\begin{align}\label{NL_22}
		&NL_{2b}\leq C_M\mu^{-\frac{2n-3}{6}}\big(\norm{\nabla\psi}_{H^{n+2}}+\norm{\phi}_{H^{n+3}}\big)\Big(\norm{Z_1^n}^2+\norm{p^{\frac{1}{2}}Z_2^n}^2+\norm{p^{\frac{1}{2}}Z_3^n}^2\Big)\\
        &\quad+C_M\left[\mu^{-2-\frac{2n}{3}}\norm{\phi}^2_{H^{n+2}} \norm{Z_1^n}^2+\mu^{-\frac{2n+3}{6}}\norm{\nabla\phi}_{H^{n+2}}\left( \norm{Z_{1}^n}^2+\norm{Z_{2}^n}^2 \right)\right]+\frac{\mu}{30000}\norm{p^{\frac{1}{2}}Z_2^n}^2.\nonumber
	\end{align}
	Finally we consider $NL_{2c}$. $NL_{2c}$ can be split into
	\begin{align*}
		NL_{2c}&=-\bigg\la Q^{n} (\p_X A_\neq \mathring{\psi_{1}}) , Q^{n} A_\neq \bigg\ra-\bigg\la Q^{n} (\p_X A_\neq \B_{1\neq})_\neq , Q^{n} A_\neq \bigg\ra-\bigg\la Q^{n} (\pt_YA \B_2)_\neq , Q^{n} A_\neq \bigg\ra\\
		&:=NL_{2c1}+NL_{2c2}+NL_{2c3}.
	\end{align*}
	Because of $\|\mathring{\psi_{1}}\|\lesssim \mu^{\alpha} (1+t)^{\frac{1}{4}}$, we consider the bad term $NL_{2c1}$ at first.
	 We have
	\begin{align}\label{NL_2311}
		\bigg\la Q^{n} (\p_X A_\neq \mathring{\psi_{1}}) , Q^{n} A_\neq \bigg\ra&=\bigg\la \p_X \left(Q^{n} A_\neq\right) \mathring{\psi_{1}} , Q^{n} A_\neq \bigg\ra+\bigg\la \mathcal{T}^n (\p_X A_\neq \mathring{\psi_{1}}) , Q^{n} A_\neq \bigg\ra \\ \nonumber
		&\lesssim \mu^{-\frac{2n+3}{6}}\norm{\nabla\psi}_{H^{n+1}}\norm{Z_2^n}^2.
	\end{align}
    And $NL_{2c2}$ can be bounded by
    \begin{align}\label{NL_23111}
    NL_{2c2}\leq\frac{\mu}{30000}\norm{p^{\frac{1}{2}}Z_2^n}^2+C_M \mu^{-\frac{2n}{3}-2}\norm{\nabla\psi}^2_{H^{n+1}}\norm{Z_2^n}^2.
	\end{align}
	$NL_{2c3}$ can be estimated in the same way as in  \cref{NL_1311,NL_211,NL_221,NL_23111}, thus we obtain
	\begin{align}\label{NL_23}
		NL_{2c}\leq& C_M\mu^{-\frac{2n+3}{6}}\big(\norm{\nabla\psi}_{H^{n+2}}+\norm{\nabla\phi}_{H^{n+2}}\big)\Big(\norm{Z_1^n}^2+\norm{Z_2^n}^2+\norm{Z_3^n}^2\Big)\\
        & +\frac{\mu}{15000}\norm{p^{\frac{1}{2}}Z_2^n}^2+C_M \mu^{-\frac{2n}{3}-2}\left(\norm{\nabla\psi}^2_{H^{n+1}}+\norm{\psi_2}^2\right)\norm{Z_2^n}^2.\nonumber
	\end{align}
	Combining  \cref{NL_21,NL_22,NL_23}, one can verify
	\begin{align}\label{NL_2}
		 N&L_{2a}+NL_{2b}+NL_{2c}\leq C_M \mu^{-\frac{2n+3}{6}} \big(\norm{\nabla\psi}_{H^{n+2}}+\norm{\nabla\phi}_{H^{n+2}}\big)\Big(\norm{Z_1^n}^2+\norm{Z_2^n}^2+\norm{Z_3^n}^2\Big)\\ \nonumber
		 &+C_M\mu^{-\frac{2n-3}{6}}\big(\norm{\nabla\psi}_{H^{n+2}}+\norm{\phi}_{H^{n+3}}\big)\Big(\norm{Z_1^n}^2+\norm{p^{\frac{1}{2}}Z_2^n}^2+\norm{p^{\frac{1}{2}}Z_3^n}^2\Big)\\\nonumber
		 &+C_M\mu^{-2-\frac{2n}{3}}\left[\norm{\phi}^2_{H^{n+2}}\norm{Z_1^n}^2+\left(\norm{\nabla\psi}_{H^{n+1}}^2+\norm{\psi_2}^2\right)\norm{Z_2^n}^2\right] +\frac{\mu}{10000}\norm{p^{\frac{1}{2}}Z_2^n}^2.\nonumber
	\end{align}
	
\noindent	{\bf{$L^2$ estimate on $Z_3^n$}}. An energy estimate gives
	\begin{align}\label{Z_3^n} 
		\frac{1}{2}\frac{d}{dt}&\|Z_3^n\|^2=-\norm{\sqrt{\frac{\p_t m_1}{m_1}+\frac{\p_t m_2}{m_2}  +\mu p} Z_3^n}^2 - \frac{3+2n}{4}\norm{\sqrt{\frac{\p_t\omega}{\omega}}Z_3^n}^2 \\ \nonumber
		&+\mu(\lambda+\mu) M^2 \bigg\la p  Z_2^n , Z_3^n\bigg\ra  
		-\mu M^2 \bigg\la \frac{\p_t p}{p}\ Z_2^n , Z_3^n\bigg\ra -2\mu M^3 \bigg\la \frac{k^2}{p^{\frac{3}{2}}} Z_1^n , Z_3^n\bigg\ra \\  \nonumber
		&+2\mu M^2 \norm{\frac{k}{p^{\frac{1}{2}}}Z_3^n}^2 +2\mu^2 M^4 \bigg\la \frac{k^2}{p} Z_2^n , Z_3^n\bigg\ra +NL_{3},  \nonumber
	\end{align}
	where
	\begin{align*}
		NL_{3}=&2\mu M^2 \bigg\la Q^{n} \left\{\pt_{Y}\B_1\p_X\B_2 \right\}_\neq ,  Z_3^n\bigg\ra -\bigg\la Q^{n} (A\Omega)_\neq ,  Z_3^n\bigg\ra-\bigg\la Q^{n} (AR)_\neq ,  Z_3^n\bigg\ra\\
		&+\mu M^2 \bigg\la Q^{n} \left\{(\pt_Y\B_2)^2\right\}_\neq ,  Z_3^n\bigg\ra+\mu M^2 \bigg\la Q^{n} \left\{(\p_X\B_1)^2\right\}_\neq ,  Z_3^n\bigg\ra,\\
		&+\mu M^2 \bigg\la Q^{n} \tilde{\dv}\left\{\frac{R}{R+1}(\lambda+2\mu)\nablat  A+\frac{R}{R+1}\mu\nablat ^{\bot} \Omega\right\}_\neq ,  Z_3^n\bigg\ra\\
		&-\mu \bigg\la Q^{n} \tilde{\dv}\left(\frac{R}{R+1}\nablat R\right)_\neq ,  Z_3^n\bigg\ra+\mu\bigg\la Q^{n} \tilde{\dv}\left(\frac{P'(R+1)-1}{R+1}\nablat  R\right)_\neq ,  Z_3^n \bigg\ra\\
		&-\bigg\la Q^{n} \nablat^{\bot}  \cdot \left\{\frac{R}{R+1}(\lambda+2\mu)\nablat  A + \frac{R}{R+1} \mu\nablat^{\bot}\Omega \right\}_\neq ,  Z_3^n\bigg\ra,\\
		&-\bigg\la Q^{n} \left\{ \B \cdot \nablat  (R+\Omega-\mu M^2 A) \right\}_\neq , Q^{n} (R+\Omega-\mu M^2 A)_\neq \bigg\ra.
	\end{align*}        
	$NL_{3}$ can be treated in the same way as in  \cref{NL_211,NL_221,NL_2311}, so we have
	\begin{align}\label{NL_3}
			N&L_{3}\leq C_M \mu^{-\frac{2n+3}{6}} \big(\norm{\nabla\psi}_{H^{n+2}}+\norm{\nabla\phi}_{H^{n+2}}\big)\Big(\norm{Z_1^n}^2+\norm{Z_2^n}^2+\norm{Z_3^n}^2\Big)\nonumber\\ 
			&+C_M\mu^{-\frac{2n-3}{6}}\big(\norm{\nabla\psi}_{H^{n+2}}+\norm{\phi}_{H^{n+3}}\big)\Big(\norm{Z_1^n}^2+\norm{p^{\frac{1}{2}}Z_2^n}^2+\norm{p^{\frac{1}{2}}Z_3^n}^2\Big)\\
			&+C_M\mu^{-2-\frac{2n}{3}}\left[\norm{\phi}^2_{H^{n+2}}\norm{Z_1^n}^2+\left(\norm{\nabla\psi}_{H^{n+1}}^2+\norm{\psi_2}^2\right)\norm{Z_3^n}^2\right] +\frac{\mu}{10000}\norm{p^{\frac{1}{2}}Z_3^n}^2.\nonumber
	\end{align}
	
\noindent	{\bf{$L^2$ estimate on $\frac{\p_t p}{p^{\frac{3}{2}}} Z_1^n$ }}. An energy estimate gives
	\begin{align}\label{lower Z_1^n}
			\frac{M^2}{2}\frac{d}{dt}\norm{\frac{\p_t p}{p^{\frac{3}{2}}}Z_1^n}^2&=-M^2\norm{\sqrt{\left(\frac{\p_t m_1}{m_1}+\frac{\p_t m_2}{m_2}+\frac{3+2n}{4}\frac{\p_t \omega}{\omega}\right)\frac{(\p_t p)^2}{p^3}}Z_1^n}^2 \nonumber\\ 
			+&M^2(\frac{n}{2}-1)\bigg\la \frac{(\p_t p)^3}{p^4} Z_1^n , Z_1^n\bigg\ra-M \bigg\la \frac{(\p_t p)^2}{p^{\frac{5}{2}}}Z_1^n , Z_2^n\bigg\ra\\ \nonumber
			+&2M^2 \bigg\la k^2\frac{\p_t p}{p^3}Z_1^n , Z_1^n \bigg\ra+M^2 \bigg\la \frac{\p_t p}{p^{\frac{3}{2}}} \mathcal{Z}_1^n, \frac{\p_t p}{p^{\frac{3}{2}}} Z_1^n\bigg\ra,
	\end{align}
	where
	\begin{align*}
		M^2\bigg\la \frac{\p_t p}{p^{\frac{3}{2}}} \mathcal{Z}_1^n,  \frac{\p_t p}{p^{\frac{3}{2}}}Z_1^n\bigg\ra=&-M \bigg\la Q^{n} \frac{(\p_t p)^2}{p^\frac{5}{2}}(A R)_\neq ,  Z_1^n \bigg\ra- \bigg\la Q^{n}  \frac{\p_t p}{p}\big(\pt_Y R \B_2\big)_\neq ,  Q^{n}  \frac{\p_t p}{p} R_\neq \bigg\ra\\
		&- \bigg\la Q^{n}  \frac{\p_t p}{p}\big(\p_X R \B_1\big)_\neq ,  Q^{n}  \frac{\p_t p}{p} R_\neq \bigg\ra\\
		:=&NL_{4a}+NL_{4b}.
	\end{align*}
	A direct calculation gives
	\begin{align}\label{NL_41-3}
		&NL_{4a}\lesssim \mu^{-\frac{2n+3}{6}}\big(\norm{\nabla\psi}_{H^{n+2}}+\norm{\psim_2}+\norm{\nabla\phi}_{H^{n+2}}\big)\Big(\norm{Z_1^n}^2+\norm{Z_2^n}^2+\norm{Z_3^n}^2\Big). 
	\end{align}
We calculate that 
	\begin{align*}
		NL_{4b}&=-\bigg\la Q^{n}  \frac{\p_t p}{p}\big(\p_X R_\neq \mathring{\psi_1}\big) ,  Q^{n}  \frac{\p_t p}{p} R_\neq \bigg\ra - \bigg\la Q^{n}  \frac{\p_t p}{p}\big(\p_X R_\neq \B_{1\neq}\big)_\neq ,  Q^{n}  \frac{\p_t p}{p} R_\neq \bigg\ra,\\
		&:=NL_{4b1}+NL_{4b2}.
	\end{align*}
	Due to $\norm{\mathring{\psi_{1}}}\lesssim \mu^{\alpha} (1+t)^{\frac{1}{4}}$, we consider $NL_{4b1}$ at first.
	Using \cref{com} and \ref{L1}, we have
	\begin{align}\label{NL_4211}
			& NL_{4b1}=  -\bigg\la \p_X \left(Q^{n} \frac{\p_t p}{p} R_\neq \right) \mathring{\psi} _1 , Q^{n}  \frac{\p_t p}{p}R_\neq\bigg\ra - \bigg\la \mathcal{P}^{n} (\p_X R_\neq \mathring{\psi} _1) , Q^{n}  \frac{\p_t p}{p}R_\neq\bigg\ra \\
			&\quad\lesssim      \bigg\la \abs{\eta \la\eta\ra^{n} \hat{\mathring{\psi}}_1} \ast \abs{p^{\frac{n}{2}} \hat{R}_\neq} , \abs{Q^{n}  \frac{\p_t p}{p}\hat{R}_\neq}\bigg\ra \lesssim  \mu^{-\frac{2n+3}{6}}\norm{\p_y \mathring{\psi} _1}_{H^{n+1}}\norm{Z_1^n}^2.\notag
	\end{align}
	The latter term can be treated in a straightforward manner by using \cref{L1}, namely, 
	\begin{align}\label{NL_422}
		NL_{4b2}\lesssim \mu^{-\frac{2n+3}{6}} \norm{\nabla \psi}_{H^{n+1}}\norm{Z_1^n}^2.
	\end{align}
	Thus, we have from \cref{NL_41-3}-\cref{NL_422} that 
	\begin{align}\label{NL_4}
		 NL_{4a}+NL_{4b}\lesssim  \mu^{-\frac{2n+3}{6}}\big(\norm{\nabla\psi}_{H^{n+2}}+\norm{\psim_2}+\norm{\nabla\phi}_{H^{n+2}}\big)\Big(\norm{Z_1^n}^2+\norm{Z_2^n}^2+\norm{Z_3^n}^2\Big). 
	\end{align}
	Adding \cref{Z_1^n,Z_2^n,Z_3^n,lower Z_1^n} up and using \cref{NL_1,NL_2,NL_3,NL_4}, we have
	\begin{align}\label{ongoing1}\nonumber
		&\frac{1}{2}\frac{d}{dt}\left(\norm{\left(1+M^2\frac{(\p_t p)^2}{{p}^3}\right)^{\frac{1}{2}}Z_1^n}^2+\|Z_2^n\|^2+\|Z_3^n\|^2\right)\\ \nonumber
		\leq & -\norm{\sqrt{\frac{\p_t m_1}{m_1}+\frac{\p_t m_2}{m_2}+\frac{2n+3}{4}\left(\frac{\p_t \omega}{\omega}-\frac{\p_t p}{p}\right)}Z_1^n}^2+C_M \mu^{-\frac{2n}{3}-2}\norm{\phi}_{H^{n+2}}^2\norm{Z_1^n}^2\\ 
		&-\norm{\sqrt{\frac{\p_t m_1}{m_1}+\frac{\p_t m_2}{m_2}+(\lambda+2\mu)p+\frac{2n+3}{4}\left(\frac{\p_t \omega}{\omega}-\frac{\p_t p}{p}\right)}Z_2^n}^2+\sum_{i=0}^2\mathcal{L}_{i} \\ \nonumber
		&-\norm{\sqrt{\frac{\p_t m_1}{m_1}+\frac{\p_t m_2}{m_2}+\mu p}Z_3^n}^2 +C_M \mu^{-2-\frac{2n}{3}}\left(\norm{\nabla\psi}_{H^{n+1}}^2+\norm{\psi_2}^2\right)\left(\norm{Z_2^n}^2+\norm{Z_3^n}^2\right)\\ \nonumber
		&+C_M \mu^{-\frac{2n+3}{6}}\big(\norm{\nabla\psi}_{H^{n+3}}+\norm{\mathring{\psi} _2}+\norm{\nabla\phi}_{H^{n+3}}\big)\Big(\norm{Z_1^n}^2+\norm{Z_2^n}^2+\norm{Z_3^n}^2\Big) \\\nonumber  
		&+C_M\mu^{-\frac{2n-3}{6}} \big(\norm{\nabla\psi}_{H^{n+2}}+\norm{\phi}_{H^{n+3}}\big) \Big(\norm{Z_1^n}^2+\norm{p^{\frac{1}{2}}Z_2^n}^2+\norm{p^{\frac{1}{2}}Z_3^n}^2 \Big),   
	\end{align}
	where
	\begin{align*}
		\mathcal{L}_0=&\frac{1}{4}\bigg\la \frac{\p_t p}{p} Z_2^n , Z_2^n\bigg\ra - \frac{1}{4}\bigg\la \frac{\p_t p}{p} Z_1^n , Z_1^n\bigg\ra, \\
		\mathcal{L}_1=& M^2 (\frac{n}{2}-1)\bigg\la \frac{(\p_t p)^3}{p^4} Z_1^n , Z_1^n\bigg\ra +2M^2\bigg\la k^2\frac{\p_t p}{p^3}Z_1^n , Z_1^n \bigg\ra\\
		&+ M \bigg\la \left(2\frac{k^2}{p^{\frac{3}{2}}}-\frac{(\p_t p)^2}{p^{\frac{5}{2}}}\right)Z_1^n , Z_2^n\bigg\ra  -\frac{2n+3}{4}\norm{\sqrt{\frac{\p_t \omega}{\omega}}Z_3^n}^2+2 M^2 \mu\norm{\frac{k}{p^{\frac{1}{2}}}Z_3^n}^2\\
		&-2 M^3 \mu\bigg\la \frac{k^2}{p^{\frac{3}{2}}}Z_1^n , Z_3^n\bigg\ra-2M^2\mu\norm{\frac{k}{p^{\frac{1}{2}}}Z_2^n}^2-M^2 \norm{\sqrt{\left(\frac{\p_t m_1}{m_1}+\frac{\p_t m_2}{m_2}+\frac{3+2n}{4}\frac{\p_t \omega}{\omega}\right)\frac{(\p_t p)^2}{p^3}}Z_1^n}^2,\\
		\mathcal{L}_2=&2(M^4\mu^2-1)\bigg\la \frac{k^2}{p}Z_2^n , Z_3^n\bigg\ra -\mu M^2 \bigg\la \frac{\p_t p}{p} Z_2^n , Z_3^n\bigg\ra +\mu (\lambda+\mu) M^2 \bigg\la p Z_2^n , Z_3^n\bigg\ra.
	\end{align*}
	Using $|\p_t p|\leqslant 2|k| p^{\frac{1}{2}}$ and  \cref{m1ini},we obtain
	\begin{align}\label{L_1-6}
		\mathcal{L}_1\leqslant \frac{\tilde{C}}{4N}\bigg(\norm{\sqrt{\frac{\p_t m_2}{m_2}}Z_1^n}^2+\norm{\sqrt{\frac{\p_t m_2}{m_2}}Z_2^n}^2+\norm{\sqrt{\frac{\p_t m_2}{m_2}}Z_3^n}^2\bigg). 
	\end{align}
	In terms of  $(\lambda+\mu)M^2<\frac{1}{2500}$, $\mu<\delta_0$, $\lambda+\mu<\delta_0$ and $|\p_t p|\leqslant 2|k| p^{\frac{1}{2}}$, we bound $\mathcal{L}_2$ as
	\begin{align}\label{L_5}
		\mathcal{L}_2\leqslant \frac{\lambda+\mu}{2}\norm{p^{\frac{1}{2}}Z_2^n}^2+\frac{3\mu}{4}\norm{p^{\frac{1}{2}}Z_3^n}^2 + \frac{\tilde{C}}{4N} \bigg(\norm{\sqrt{\frac{\p_t m_2}{m_2}}Z_2^n}^2+\norm{\sqrt{\frac{\p_t m_2}{m_2}}Z_3^n}^2\bigg).  
	\end{align}
	Substituting \cref{L_1-6,L_5} into \cref{ongoing1}, there holds
	\begin{align*}
		&\frac{1}{2}\frac{d}{dt} \bigg(\norm{\left(1+M^2\frac{(\p_t p)^2}{{p}^3}\right)^{\frac{1}{2}}Z_1^n}^2+\norm{Z_2^n}^2+\norm{Z_3^n}^2 \bigg)\\
		\leq & -\norm{\sqrt{\frac{63}{64}\left(\frac{\p_t m_1}{m_1}+\frac{\p_t m_2}{m_2}\right)+\frac{2n+3}{4}\left(\frac{\p_t \omega}{\omega}-\frac{\p_t p}{p}\right)}Z_1^n}^2 \\
		&-\norm{\sqrt{\frac{31}{64}\left(\frac{\p_t m_1}{m_1}+\frac{\p_t m_2}{m_2}+(\lambda+2\mu)p\right)+\frac{2n+3}{4} \left(\frac{\p_t \omega}{\omega}-\frac{\p_t p}{p} \right)}Z_2^n}^2\\
		&-\frac{15}{64}\norm{\sqrt{\frac{\p_t m_1}{m_1}+\frac{\p_t m_2}{m_2}+\mu p}Z_3^n}^2+\frac{1}{4}\bigg\la \frac{\p_t p}{p} Z_2^n , Z_2^n\bigg\ra- \frac{1}{4}\bigg\la \frac{\p_t p}{p} Z_1^n , Z_1^n\bigg\ra\\
		&+C_M\mu^{-\frac{2n+3}{6}} \Big(\norm{\nabla\psi}_{H^{n+3}}+\norm{\psi_2}+\norm{\nabla\phi}_{H^{n+3}}\Big) \Big(\norm{Z_1^n}^2+\norm{Z_2^n}^2+\norm{Z_3^n}^2 \Big)\\
		&+C_M\mu^{-\frac{2n-3}{6}} \Big(\norm{\nabla\psi}_{H^{n+2}}+\norm{\phi}_{H^{n+3}}\Big) \Big(\norm{Z_1^n}^2+\norm{p^{\frac{1}{2}}Z_2^n}^2+\norm{p^{\frac{1}{2}}Z_3^n}^2 \Big)\\
		&+C_M \mu^{-2-\frac{2n}{3}}\left(\norm{\phi}_{H^{n+2}}^2+\norm{\psi_2}^2+\norm{\nabla\psi}_{H^{n+1}}^2\right) \left(\norm{Z_1^n}^2+\norm{Z_2^n}^2+\norm{Z_3^n}^2\right).      
	\end{align*}     
This finishes the proof of \cref{Z_1^n+Z_2^n+Z_3^n}.
\end{proof}

\subsection{Estimate on $\frac{M}{4}\frac{d}{dt}\bigg\la \frac{\p_t p}{p^{\frac{3}{2}}}Z_1^n , Z_2^n\bigg\ra-\gamma\frac{d}{dt}\bigg\la p^{-\frac{1}{2}}Z_1^n , Z_2^n\bigg\ra$}
This section is devoted to the enhanced dissipation for density.
\begin{Lem}\label{damping of Z_1^n}
	Under the same assumptions of \cref{MT},  it holds for $n=0,1$ and  $\gamma=\frac{M\mu^{\frac{1}{3}}}{4}$ that
	\begin{align*} 
		&\frac{M}{4}\frac{d}{dt}\bigg\la \frac{\p_t p}{p^{\frac{3}{2}}}Z_1^n , Z_2^n\bigg\ra-\gamma\frac{d}{dt}\bigg\la p^{-\frac{1}{2}}Z_1^n , Z_2^n\bigg\ra + \frac{\gamma}{2M} \norm{Z_1^n}^2 + 2\gamma M\norm{\frac{k}{p}Z_1^n}^2\\
		\leq  &\frac{1}{4}\norm{\sqrt{\frac{\p_t m_1}{m_1}+\frac{\p_t m_2}{m_2}+\frac{3+2n}{4}\left(\frac{\p_t \omega}{\omega}-\frac{\p_t p}{p} \right)}Z_1^n}^2 +\frac{\tilde{C}}{2N}\sum_{j=1}^3\norm{\sqrt{\frac{\p_t m_2}{m_2}}Z_j^n}^2\\
		&-\frac{1}{4}\bigg\la \frac{\p_t p}{p} Z_2^n , Z_2^n\bigg\ra + \frac{1}{4}\bigg\la \frac{\p_t p}{p} Z_1^n , Z_1^n\bigg\ra+\frac{\mu^{\frac{1}{3}}}{4}\norm{Z_2^n}^2+\frac{3(\lambda+2\mu)}{16}\norm{p^{\frac{1}{2}}Z_2^n}^2\\
		&+C_M \mu^{-\frac{2n+3}{6}}\Big(\norm{\nabla\psi}_{H^{n+2}}+\norm{\psi_2}+\norm{\nabla\phi}_{H^{n+2}}\Big)\Big(\norm{Z_1^n}^2+\norm{Z_2^n}^2+\norm{Z_3^n}^2\Big)\\
	    &+C_M\mu^{-\frac{2n-3}{6}} \Big(\norm{\nabla\psi}_{H^{n+2}}+\norm{\phi}_{H^{n+3}}\Big) \Big(\norm{Z_1^n}^2+\norm{p^{\frac{1}{2}}Z_2^n}^2+\norm{p^{\frac{1}{2}}Z_3^n}^2 \Big).
	\end{align*}
\end{Lem}

\begin{proof} 
	An energy estimate gives
	\begin{align}\label{ongoing2}
		\frac{M}{4}\frac{d}{dt}\bigg\la \frac{\p_t p}{p^{\frac{3}{2}}}Z_1^n , Z_2^n\bigg\ra-\gamma\frac{d}{dt}\bigg\la p^{-\frac{1}{2}}Z_1 , Z_2\bigg\ra=\sum_{i=0}^2\mathcal{I}_i+NL_{5a}+NL_{5b},
	\end{align}
	where  
	\begin{align*}
		\mathcal{I}_0=&-\frac{1}{4}\bigg\la \frac{\p_t p}{p} Z_2^n , Z_2^n\bigg\ra + \frac{1}{4}\bigg\la \frac{\p_t p}{p} Z_1^n , Z_1^n\bigg\ra,\\
		\mathcal{I}_1=&-\frac{M}{2}\bigg\la \frac{k^2 \p_t p}{p^{\frac{5}{2}}} Z_1^n , Z_3^n\bigg\ra +2\gamma \bigg\la \frac{k^2}{p^{\frac{3}{2}}} Z_1^n , Z_3^n\bigg\ra\\
		&+\frac{M^2}{2} \bigg\la \frac{k^2 \p_t p}{p^3} Z_1^n , Z_1^n\bigg\ra  -\frac{\gamma}{M}\norm{Z_1^n}^2-2\gamma M \bigg\la\frac{k^2}{p^2}Z_1^n , Z_1^n\bigg\ra\\
		&+\frac{M}{4}\bigg\la \left(\frac{2k^2}{p^{\frac{3}{2}}}-\frac{3}{2}\frac{(\p_t p)^2}{p^{\frac{5}{2}}}\right) Z_1^n , Z_2^n\bigg\ra -\mu M^3 \bigg\la \frac{k^2 \p_t p}{2p^{\frac{5}{2}}} Z_1^n , Z_2^n\bigg\ra  \\
		&-\frac{M}{2}\bigg\la \left(\frac{\p_t m_1}{m_1}+\frac{\p_t m_2}{m_2} +\frac{2n+3}{4}\left(\frac{\p_t \omega}{\omega}-\frac{\p_t p}{p}\right)\right)\frac{\p_t p}{p^{\frac{3}{2}}} Z_1^n , Z_2^n\bigg\ra+ \frac{\gamma}{2}\bigg\la \frac{\p_t p}{p^{\frac{3}{2}}} Z_1^n , Z_2^n\bigg\ra\\
		&+2\gamma\bigg\la \left(\frac{\p_t m_1}{m_1}+\frac{\p_t m_2}{m_2} +\frac{2n+3}{4}\left(\frac{\p_t \omega}{\omega}-\frac{\p_t p}{p}\right)\right)p^{-\frac{1}{2}} Z_1^n , Z_2^n\bigg\ra+2\gamma\mu M^2 \bigg\la \frac{k^2}{p^{\frac{3}{2}}} Z_1^n , Z_2^n\bigg\ra,\\
		\mathcal{I}_2=&-\frac{M}{4}(\lambda+2\mu)\bigg\la \frac{\p_t p}{p^{\frac{1}{2}}} Z_1^n , Z_2^n\bigg\ra+\gamma(\lambda+2\mu)\bigg\la p^{\frac{1}{2}} Z_1^n , Z_2^n\bigg\ra+\frac{\gamma}{M} \|Z_2^n\|^2,\\
		NL_{5a}=&-\frac{1}{4}\bigg\la Q^{n} \frac{\p_t p}{p}(A R)_\neq ,  Z_2^n\bigg\ra+\frac{\gamma}{M}\bigg\la Q^{n} (A R)_\neq ,  Z_2^n\bigg\ra \notag\\
		&-\frac{M}{4}\bigg\la \frac{\p_t p}{p^{\frac{3}{2}}}Q^{n} \left[ \left\{(\pt_Y\B_2)^2\right\}_\neq+\left\{(\p_X\B_1)^2\right\}_\neq+2\left(\pt_{Y}\B_1\p_X \B_2 \right)_\neq\right] ,  Z_1^n\bigg\ra \notag\\   
		&+\gamma\bigg\la p^{-\frac{1}{2}}Q^{n}  \left[ \left\{(\pt_Y\B_2)^2\right\}_\neq+\left\{(\p_X\B_1)^2\right\}_\neq+2\left(\pt_{Y}\B_1\p_X \B_2 \right)_\neq\right] ,  Z_1^n\bigg\ra \notag\\    
		&+\frac{1}{4M}\bigg\la \frac{\p_t p}{p^{\frac{3}{2}}}Q^{n}  \tilde{\dv}\left(\frac{R}{R+1} \nablat R\right)_\neq ,  Z_1^n\bigg\ra -\frac{1}{4M}\bigg\la \frac{\p_t p}{p^{\frac{3}{2}}}Q^{n}  \tilde{\dv}\left(\frac{P'(R+1)-1}{R+1} \nablat  R\right)_\neq ,  Z_1^n\bigg\ra \notag\\
		&-\frac{M}{4}\bigg\la \frac{\p_t p}{p^{\frac{3}{2}}}Q^{n}  \tilde{\dv}\left((\lambda+2\mu)\frac{R}{R+1}\nablat A+\mu\frac{R}{R+1}\nablat ^{\bot}\Omega\right)_\neq ,  Z_1^n\bigg\ra \notag\\   
		&-\frac{\gamma}{M^2}\bigg\la p^{-\frac{1}{2}}Q^{n}  \tilde{\dv}\left(\frac{R}{R+1} \nablat R\right)_\neq ,  Z_1^n\bigg\ra +\frac{\gamma}{M^2} \bigg\la p^{-\frac{1}{2}}Q^{n}  \tilde{\dv}\left(\frac{P'(R+1)-1}{R+1} \nablat  R\right)_\neq ,  Z_1^n\bigg\ra \notag\\
		&+\gamma\bigg\la p^{-\frac{1}{2}}Q^{n}  \tilde{\dv}\left((\lambda+2\mu)\frac{R}{R+1}\nablat A+\mu\frac{R}{R+1}\nablat ^{\bot}\Omega\right)_\neq ,  Z_1^n\bigg\ra,\notag\\   
		NL_{5b}=&-\frac{1}{4}\bigg\la Q^{n}  \frac{\p_t p}{p}\left\{ \B \cdot \nablat  R\right\}_\neq ,  Q^{n}  A_\neq\bigg\ra
		-\frac{1}{4}\bigg\la Q^{n}  \left\{\B \cdot \nablat  A\right\}_\neq ,  Q^{n}  \frac{\p_t p}{p} R_\neq\bigg\ra \\
		&+\frac{\gamma}{M}\bigg\la Q^{n}  \left\{\B \cdot \nablat  R\right\}_\neq ,  Q^{n}  A_\neq\bigg\ra
		+\frac{\gamma}{M}\bigg\la Q^{n}  \left\{\B \cdot \nablat  A\right\}_\neq ,  Q^{n}  R_\neq\bigg\ra:=NL_{5b1}+NL_{5b2}.\notag
	\end{align*}
	A direct calculation gives
	\begin{align}\label{I_1-3}
			\mathcal{I}_1\leqslant& \frac{1}{4}\norm{\sqrt{\frac{\p_t m_1}{m_1}+\frac{\p_t m_2}{m_2}+\frac{3+2n}{4}\left(\frac{\p_t \omega}{\omega}-\frac{\p_t p}{p}\right)}Z_1^n}^2 \\ \nonumber
			&\qquad\qquad+\frac{\tilde{C}}{4N}\sum_{j=1}^3\norm{\sqrt{\frac{\p_t m_2}{m_2}}Z_j^n}^2-\frac{\gamma}{M} \norm{Z_1^n}^2 - 2\gamma M\norm{\frac{k}{p}Z_1^n}^2. 
	\end{align}
 According to $\mu<\delta_0$, $\lambda+\mu<\delta_0$ and the restrition of $M$, we bound $\mathcal{I}_2$ as
	\begin{align}\label{I_4}
		\mathcal{I}_2\leqslant \frac{\mu^{\frac{1}{3}}}{8}\norm{Z_1^n}^2+\frac{\mu^{\frac{1}{3}}}{4}\|Z_2^n\|^2+\frac{3(\lambda+2\mu)}{16}\norm{p^{\frac{1}{2}}Z_2^n}^2+\frac{\tilde{C}}{4N}\norm{\sqrt{\frac{\p_t m_2}{m_2}}Z_3^n}^2. 
	\end{align}
	Using \cref{L1}, one has
	\begin{align}\label{NL_51-2}
		NL_{5a}&\lesssim \mu^{-\frac{2n+3}{6}}\big(\norm{\nabla\psi}_{H^{n+2}}+\norm{\nabla\phi}_{H^{n+2}}\big)\Big(\norm{Z_1^n}^2+\norm{Z_2^n}^2+\norm{Z_3^n}^2\Big)\\
		&\quad+\mu^{-\frac{2n-3}{6}} \Big(\norm{\nabla\psi}_{H^{n+2}}+\norm{\phi}_{H^{n+3}}\Big) \Big(\norm{Z_1^n}^2+\norm{p^{\frac{1}{2}}Z_2^n}^2+\norm{p^{\frac{1}{2}}Z_3^n}^2 \Big).\nonumber
	\end{align}
	Turning to $NL_{5b}$, we obtain
	\begin{align}\label{NL_531}
		\begin{aligned}
			NL_{5b1}=&-\frac{1}{4}\bigg\la Q^{n}  \frac{\p_t p}{p}\big(\pt_{Y}R\B_2 \big)_\neq , Q^{n}  A_\neq\bigg\ra-\frac{1}{4}\bigg\la  Q^{n}  \big(\pt_{Y}A\B_2 \big)_\neq , Q^{n} \frac{\p_t p}{p} R_\neq\bigg\ra\\
			&-\frac{1}{4}\bigg\la \frac{\p_t p}{p} Q^{n} \big(\p_X R_\neq \mathcal{B}_{1\neq}\big)_\neq , Q^{n}  A_\neq \bigg\ra
			-\frac{1}{4}\bigg\la Q^{n}  \big(\p_X A_\neq \mathcal{B}_{1\neq}\big)_\neq , Q^{n} \frac{\p_t p}{p} R_\neq \bigg\ra\\
			&-\frac{1}{4}\bigg\la \frac{\p_t p}{p} Q^{n}  (\p_X R_\neq \mathring{\psi_{1}}) , Q^{n}  A_\neq\bigg\ra-\frac{1}{4}\bigg\la Q^{n}  (\p_X A_\neq \mathring{\psi_{1}}) , \frac{\p_t p}{p} Q^{n}  R_\neq\bigg\ra\\
			\lesssim &\mu^{-\frac{2n+3}{6}}\left(\norm{\nabla\psi}_{H^{n+2}}+\norm{\nabla\phi}_{H^{n+2}}+\norm{\psim_2}\right)\Big(\norm{Z_1^n}^2+\norm{Z_2^n}^2+\norm{Z_3^n}^2\Big),
		\end{aligned}
	\end{align}
	where we have used that
	\begin{align*}
	&\bigg\la \frac{\p_t p}{p} Q^{n}  (\p_X R_\neq \mathring{\psi_{1}}) , Q^{n}  A_\neq\bigg\ra+\bigg\la Q^{n}  (\p_X A_\neq \mathring{\psi_{1}}) , \frac{\p_t p}{p} Q^{n}  R_\neq\bigg\ra\\
    &\;=\bigg\la  \p_X \left(\frac{\p_t p}{p} Q^{n} R_\neq\right) \mathring{\psi_{1}} , Q^{n}  A_\neq\bigg\ra+  \bigg\la \p_X \left(Q^{n} A_\neq\right) \mathring{\psi_{1}} , \frac{\p_t p}{p} Q^{n}  R_\neq\bigg\ra\\
	&\quad+\bigg\la \mathcal{P}^n (\p_X R_\neq \mathring{\psi_{1}}) , Q^{n}  A_\neq\bigg\ra +  \bigg\la \mathcal{T}^n  (\p_X A_\neq \mathring{\psi_{1}}) , \frac{\p_t p}{p} Q^{n}  R_\neq\bigg\ra\\
    &\lesssim \mu^{-\frac{2n+3}{6}}\left(\norm{\nabla\psi}_{H^{n+2}}+\norm{\nabla\phi}_{H^{n+2}}\right)\Big(\norm{Z_1^n}^2+\norm{Z_2^n}^2+\norm{Z_3^n}^2\Big).
\end{align*}
$NL_{5b2}$ can be treated in a similar way.
	Combining \cref{ongoing2}-\cref{NL_531}, it holds that
	\begin{align}
			&\frac{M}{4}\frac{d}{dt}\bigg\la \frac{\p_t p}{p^{\frac{3}{2}}}Z_1^n , Z_2^n\bigg\ra-\gamma\frac{d}{dt}\bigg\la p^{-\frac{1}{2}}Z_1 , Z_2\bigg\ra + \frac{\gamma}{2M} \norm{Z_1^n}^2 + 2\gamma M\norm{\frac{k}{p}Z_1^n}^2 \nonumber\\ 
			\leq  &\frac{1}{4}\norm{\sqrt{\frac{\p_t m_1}{m_1}+\frac{\p_t m_2}{m_2}+\frac{3+2n}{4}\left(\frac{\p_t \omega}{\omega}-\frac{\p_t p}{p}\right)}Z_1^n}^2 +\frac{\tilde{C}}{2N}\sum_{j=1}^3\norm{\sqrt{\frac{\p_t m_2}{m_2}}Z_j^n}^2\nonumber\\
			&-\frac{1}{4}\bigg\la \frac{\p_t p}{p} Z_2^n , Z_2^n\bigg\ra + \frac{1}{4}\bigg\la \frac{\p_t p}{p} Z_1^n , Z_1^n\bigg\ra+\frac{\mu^{\frac{1}{3}}}{4}\norm{Z_2^n}^2+\frac{3(\lambda+2\mu)}{16}\norm{p^{\frac{1}{2}}Z_2^n}^2\\
			&+C_M \mu^{-\frac{2n+3}{6}}\big(\norm{\nabla\psi}_{H^{n+2}}+\norm{\psi_2}+\norm{\nabla\phi}_{H^{n+2}}\big)\Big(\norm{Z_1^n}^2+\norm{Z_2^n}^2+\norm{Z_3^n}^2\Big)\nonumber\\
	        &+C_M\mu^{-\frac{2n-3}{6}} \Big(\norm{\nabla\psi}_{H^{n+2}}+\norm{\phi}_{H^{n+3}}\Big) \Big(\norm{Z_1^n}^2+\norm{p^{\frac{1}{2}}Z_2^n}^2+\norm{p^{\frac{1}{2}}Z_3^n}^2 \Big).\nonumber
		\end{align}
	This completes the proof of \cref{damping of Z_1^n}.
\end{proof}

\subsection{Enhanced dissipation for non-zero modes}
We are ready to prove \cref{ed}. Note that $\alpha>\frac{11}{3}$,   \cref{aps} gives 
\begin{align}\label{0909090909}
\norm{\nabla \psi}_{H^4}+\norm{\nabla \phi}_{H^4} \lesssim \mu^{\frac{7}{6}+},\qquad  \norm{\nabla\psi}_{H^{2}}+ \norm{\phi}_{H^3}\lesssim \mu^{\frac{13}{6}+},
\end{align}	
where $\frac76+$ means $\frac76+\varepsilon_0$ for some $\varepsilon_0>0$.
Combining \cref{multiplier prop}, \cref{Z_1^n+Z_2^n+Z_3^n}, \ref{damping of Z_1^n} and \cref{0909090909}, we obtain  for $n=0,1$ that 
\begin{align*}
	\frac{d}{dt}E^n(t)&\leq  -\norm{\sqrt{\frac{3}{64}\left(\frac{\p_t m_1}{m_1}+\frac{\p_t m_2}{m_2}+\mu^{\frac{1}{3}}\right)+\frac{2n+3}{4}\left(\frac{\p_t \omega}{\omega}-\frac{\p_t p}{p}\right)}Z_1^n}^2\\
	&-\norm{\sqrt{\frac{3}{64}\left(\frac{\p_t m_1}{m_1}+\frac{\p_t m_2}{m_2}+(\lambda+2\mu)p\right)+\frac{2n+3}{4}\left(\frac{\p_t \omega}{\omega}-\frac{\p_t p}{p}\right)}Z_2^n}^2\\
	&-\frac{3}{64}\norm{\sqrt{\frac{\p_t m_1}{m_1}+\frac{\p_t m_2}{m_2}+\mu p}Z_3^n}^2-2\gamma M \bigg\la \frac{k^2}{p^2}Z_1^n , Z_1^n\bigg\ra \\
	&+C_M \mu^{-\frac{2n+3}{6}}\Big(\norm{\nabla\psi}_{H^{n+3}}+\norm{\psi_2}+\norm{\nabla\phi}_{H^{n+3}}\Big)\Big(\norm{Z_1^n}^2+\norm{Z_2^n}^2+\norm{Z_3^n}^2\Big)\\      
	&+C_M\mu^{-\frac{2n-3}{6}} \Big(\norm{\nabla\psi}_{H^{n+2}}+\norm{\phi}_{H^{n+3}}\Big) \Big(\norm{Z_1^n}^2+\norm{p^{\frac{1}{2}}Z_2^n}^2+\norm{p^{\frac{1}{2}}Z_3^n}^2 \Big)\\
	&+C_M \mu^{-2-\frac{2n}{3}}\left(\norm{\phi}_{H^{n+2}}^2+\norm{\psi_2}^2+\norm{\nabla\psi}_{H^{n+1}}^2\right) \left(\norm{Z_1^n}^2+\norm{Z_2^n}^2+\norm{Z_3^n}^2\right)\\
	\leq & -\frac{1}{64}\mu^{\frac{1}{3}}E^n(t),
\end{align*}
where we have used the fact that $-\gamma M \bigg\la \frac{k^2}{p^2}Z_1^n , Z_1^n \bigg\ra \lesssim -\mu^{\frac{1}{3}}M^2 \bigg\la \frac{(\p_t p)^2}{p^3} Z_1^n,Z_1^n \bigg\ra$. 
Hence, the Gr{\"o}nwall's inequality implies
\begin{align*}
	E^n(t)\leq E^n(0) e^{-\frac{1}{64}\mu^{\frac{1}{3}}t} \lesssim \mu^{2\alpha}e^{-\frac{1}{64}\mu^{\frac{1}{3}}t}. 
\end{align*}
By the Helmholtz decomposition and the change of variable, we find 
\begin{align*}
	\frac{1}{M^2}\|\nabla^{1+n} &\phi_\neq\|^2+\norm{\nabla^{1+n} \psi_\neq}^2\lesssim  \norm{p^{\frac{n}{2}}A_\neq}^2+\norm{p^{\frac{n}{2}}\Omega_\neq}^2+\norm{p^{\frac{n+1}{2}}R_\neq}^2\\
	&\lesssim \mu^{-\frac{2n+3}{3}}\Big(\norm{\omega^{-\frac{2n+3}{4}} p^{\frac{n}{2}}A_\neq}^2+\norm{\omega^{-\frac{2n+3}{4}} p^{\frac{n}{2}}\Omega_\neq}^2+\norm{\omega^{-\frac{2n+3}{4}} p^{\frac{n+1}{2}} R_\neq}^2\Big)\\
	&\lesssim \mu^{2\alpha-\frac{2n+3}{3}} e^{-\frac{1}{64}\mu^{\frac{1}{3}}t},
\end{align*}
where we have used $\norm{\Omega_\neq}\leqslant \norm{\Omega_\neq+ R_\neq - \mu M^2 A_\neq}+\norm{R_\neq} + \mu M^2 \norm{A_\neq}$.  Therefore \cref{ed} is completed. 


\section{Estimates on higher-order derivatives}
We return back the original perturbation system,
\begin{equation}\label{perturbation}
	\begin{cases}
		&\p_t \phi +y\p_x \phi + \dv \psi=N_1,\\
		&\p_t \psi+y\p_x\psi+\psi_{2}\mathbf{e}_1-\mu\lap\psi-(\lambda+\mu) \nabla \dv  \psi +\frac{1}{M^2} \nabla\phi=N_2.
	\end{cases}
\end{equation}
where
\begin{equation}
	\begin{split}
		&N_1=-\psi\cdot\nabla\phi-\phi\dv\psi,\\
		&N_2=-\psi\cdot\nabla\psi-\frac{\phi}{\phi+1}\big(\mu\lap\psi+(\lambda+\mu)\nabla\dv\psi\big)+\frac{1}{M^2}\frac{\phi}{\phi+1}\nabla\phi-\frac{1}{M^2}\frac{(P'(\rho)-1)\nabla\phi}{\phi+1}.
	\end{split}
\end{equation}      
Firstly, we give a basic energy estimate of $\norm{\p_y\phim(\cdot,t),\p_y \psim_2(\cdot,t)}$.
\begin{Lem}\label{77777}
Under the same assumptions of \cref{MT}, it holds that 
\begin{align}
\frac{1}{M^2}\norm{\p_y \phim}^2+\norm{\p_y \psim_2}^2+\mu \int_0^T \norm{\p_y^2 \psim_2}^2 dt \lesssim \mu^{2\alpha}.
\end{align}
\end{Lem}
\begin{proof}
	Applying $\Do$ to \cref{perturbation}$_1$ and \cref{perturbation}$_3$, one has
\begin{align}\label{777777}
\left\{\begin{aligned}
&\p_t \phim+ \p_y \psim_2=N_1^{(1)}+N_1^{(2)}\\
&\p_t \psim_2 -(\lambda+2\mu)\p_y^2 \psim_2 + \frac{1}{M^2}\p_y \phim = N_2^{(1)}+N_2^{(2)},
\end{aligned} \right.
\end{align}
where
\begin{align}
&N_1^{(2)}=-\Do(\psi_{2\neq}\p_y \phi_{\neq}+\phi_{\neq}\p_y \psi_{2\neq}),\\ \nonumber
&N_2^{(2)}=-\Do\left(\psi_\neq \cdot \nabla \psi_\neq + \mu \phi_\neq \lap \psi_{2\neq}+(\lambda+\mu)\phi_\neq \p_y \dv \psi_\neq + \phi_\neq \p_y\phi_\neq \right),\\
&N_1^{(1)}=-\psim_2\p_y\phim-\phim\p_y\psim_2, \quad N_2^{(1)}=-(\lambda+2\mu)\phim\p_y^2\psim_2-\psim_2\p_y\psim_2-\phim\p_y\phim.\nonumber
\end{align}
 By \cref{ed}, there holds 
\begin{align}\label{7777777}
\norm{(\p_y N_1^{(2)},N_2^{(2)})}\lesssim \mu^{2\alpha-\frac{5}{3}-\varepsilon}e^{-\frac{1}{64}\mu^{\frac{1}{3}}t} . 
\end{align}
Multiplying $\p_y$\cref{777777}$_1$ by $\frac{1}{M^2}\p_y \phim$ and \cref{777777}$_2$ by $-\p_y^2 \psim_2$, adding them up and integrating the results on $\R$, one has
\begin{align}\label{777}
\frac{1}{2} \frac{d}{dt}& \left( \frac{1}{M^2}\norm{\p_y \phim}^2+\norm{\p_y\psim_2}^2 \right) + (\lambda+2\mu)\norm{\p_y^2 \psim_2}^2 \\ \nonumber
&= \frac{1}{M^2} \int_\R \p_y \phim (\p_y N_1^{(1)}+\p_y N_1^{(2)}) dy+ \int_\R \p_y^2\psim_2 ( N_2^{(1)}+N_2^{(2)})dy.
\end{align}
By the a priori assumption \cref{aps}, \cref{estimateontheta}-\ref{1or}, \ref{2or} and \ref{7777777}, the nonlinear terms can be controlled as
\begin{align}\label{7777}
&\int_\R \frac{1}{M^2} \p_y \phim \p_y N_1^{(1)}+ \p_y^2 \psim_2 N_2^{(1)}dy \lesssim \delta \mu \norm{\p_y^2 \psim_2}^2 + \mu^{2\alpha+\varepsilon}(1+t)^{-\frac{3}{2}},\\
&\int_\R \frac{1}{M^2} \p_y \phim \p_y N_1^{(2)} + \p_y^2 \psim_2 N_2^{(2)} dy \lesssim \delta\mu \norm{\p_y^2 \psim_2}^2  +\mu^{2\alpha+\varepsilon}(1+t)^{-\frac{3}{2}}+\mu^{4\alpha-5-\varepsilon}e^{-\frac{1}{64}\mu^{\frac{1}{3}}t}. \nonumber
\end{align}
Substituting \cref{7777} into \cref{777} and integrating over $(0,T)$, we have
\begin{align}
	\frac{1}{M^2}\norm{\p_y \phim}^2+\norm{\p_y \psim_2}^2+\mu \int_0^T \norm{\p_y^2 \psim_2}^2 dt \lesssim \mu^{2\alpha}.
	\end{align}
This completes the proof of \cref{77777}.
\end{proof}
Thus, from \cref{estimateonxi}, \ref{2or}-\ref{Aor}, \cref{ed} and \cref{77777},  we deduce that 
\begin{align}\label{stepone}
	\frac{1}{M^2}\norm{\nabla \phi}^2+\norm{\nabla\psi}^2+\frac{\mu}{M^2}\int_{0}^T \norm{\nabla \phi}^2 dt+ \mu\int_0^T\norm{\nabla^2 \psi}^2 dt \lesssim \mu^{2\alpha-1}.
\end{align}
Furthermore, we have
\begin{Thm}\label{000000}
	Under the same assumptions of \cref{MT}, it holds that 
	\begin{align}\label{energy}
			\frac{1}{M^2}\norm{\nabla^{|\beta|+1} \phi}^2+\norm{\nabla^{|\beta|+1} \psi}^2+\frac{1}{M^2}\int_0^T \norm{\nabla^{|\beta|+1} \phi}^2dt+\mu\int_0^T\norm{\nabla^{|\beta|+2} \psi}^2 dt \lesssim  \mu^{2\alpha-\abs{\beta}-1},
	\end{align}
 where $|\beta|=|(\beta_1 , \beta_2)|=k\geqslant 1 .$
\end{Thm}
\begin{proof}
	We estimate $\norm{\nabla^{|\beta|+1}\phi}$ at first.
	Applying $\p_i\p^\beta(i=1,2)$ to the mass equation \cref{perturbation}$_1$ and $\p^\beta$ to the  momentum equation \cref{perturbation}$_2$ with $\p^\beta=\p^{\beta_1}_x\p^{\beta_2}_y$, we have
	\begin{equation}\label{ori}
		\begin{cases}
			&\p_t \p_i\p^\beta\phi +\p_i\p^\beta(y\p_x \phi) + \p_i\p^\beta\dv \psi=\p_i\p^\beta N_1,\\[2mm]
			&\p_t \p^\beta\psi+\p^\beta (y\p_x\psi) +\p^\beta\psi_{2}e_1-\mu\lap\p^\beta\psi-(\lambda+\mu) \nabla \dv \p^\beta \psi +\frac{1}{M^2}\nabla\p^\beta\phi=\p^\beta N_2.
		\end{cases}
	\end{equation}
Multiplying the equation \cref{ori}$_2$ by $\nabla\p^\beta\phi$ and integrating the results on $\mathbb{T} \times\mathbb{R}$ lead to
	\begin{align}\label{phi1}
			&\frac{d}{dt}\int_{\mathbb{T} \times\mathbb{R}} \p^\beta\psi_i \p_i\p^\beta\phi   dxdy+ \frac{1}{M^2} \int_{\mathbb{T} \times\mathbb{R}} \abs{\p_i \p^\beta \phi}^2   dxdy\nonumber\\ 
			=&-\int_{\mathbb{T} \times\mathbb{R}} \left[\p_i\p^\beta \phi \p^\beta (y\p_x \psi_i)  + \p^\beta \psi_i \p_i\p^\beta (y\p_x \phi)  + \p^\beta \psi_i \p_i \p^\beta \dv\psi  + \p_x \p^\beta \phi \p^\beta \psi_2 \right] dxdy\\
			&+\int_{\mathbb{T} \times\mathbb{R}} \p^\beta N_{2i} \p_i \p^\beta \phi   +  \p^\beta \psi_i \p_i \p^\beta N_1    + (\lambda+2\mu) \p_i \p^\beta \phi \p_i \p^\beta \dv \psi   dxdy. \nonumber
	\end{align}
	where we have used
	\begin{align*}
		&\int_{\mathbb{T} \times\mathbb{R}}\p_t \p^\beta \psi_i \p_i \p^\beta \phi dxdy=\frac{d}{dt} \int_{\mathbb{T} \times\mathbb{R}} \p^\beta \psi_i \p_i \p^\beta \phi   dxdy - \int_{\mathbb{T} \times\mathbb{R}} \p^\beta \psi_i \p_i \p^\beta \p_t \phi   dxdy\\
		=&\frac{d}{dt} \int_{\mathbb{T} \times\mathbb{R}} \p^\beta \psi_i \p_i \p^\beta \phi   dxdy + \int_{\mathbb{T} \times\mathbb{R}} \p^\beta \psi_i \p_i\p^\beta (y\p_x \phi)   +  \p^\beta \psi_i \p_i \p^\beta \dv\psi   -  \p^\beta \psi_i \p_i \p^\beta N_1   dxdy,
	\end{align*}
	and
	$$ \mu \int_{\mathbb{T} \times\mathbb{R}}  \p_i \p^\beta \phi \lap \p^\beta \psi_i   dxdy= \mu \int_{\mathbb{T} \times\mathbb{R}} \p_i \p^\beta \phi \p_i \p^\beta \dv\psi   dxdy. $$
	Next, multiplying \cref{ori}$_1$ by $\p_i \p^\beta\phi$ and integrating respect to $(x,y)$ yield
	\begin{align}\label{phi}
		\frac{d}{dt}\int_{\mathbb{T} \times\mathbb{R}} |\p_i \p^\beta \phi|^2   dxdy=\int_{\mathbb{T} \times\mathbb{R}} \p_i \p^\beta \phi \p_i \p^\beta N_1-\p_i \p^\beta \phi \p_i \p^\beta (y \p_x \phi)  - \p_i \p^\beta \phi \p_i \p^\beta \dv\psi dxdy .
	\end{align}
	Then, adding \cref{phi1}$\times(\lambda+2\mu)$ and \cref{phi}$\times(\lambda+2\mu)^2$ up, summing $i$  from 1 to 2 gives
	\begin{align}\label{a11}
			\frac{d}{dt}&\Big(\int_{\mathbb{T} \times\mathbb{R}}  \bar{\mu}^2 |\nabla \p^\beta\phi|^2 +\bar{\mu} \nabla \p^\beta \phi \cdot \p^\beta \psi   dxdy \Big) +\frac{\bar{\mu}}{M^2} \norm{\nabla \p^\beta \phi}^2\\
			&\leq  100\bar{\mu}^2  \norm{\nabla^{|\beta|+1} \phi}^2  + C_M\bar{\mu} \left\{\norm{\nabla^{|\beta|}\phi}^2+ \norm{\nabla^{|\beta|+1}\psi}^2+\sum_{l=1}^3 \mathcal{R}^\beta_l\right\},\nonumber
	\end{align}
where
	\begin{align*}
		\mathcal{R}_1^\beta=&\int_{\mathbb{T} \times\mathbb{R}} \p^\beta N_{2i} \p_i \p^\beta \phi dxdy:=\mathcal{R}_{1,1}^\beta+\mathcal{R}_{1,2}^\beta=-\int_{\mathbb{T} \times\mathbb{R}}\p^{\beta}(\psi\cdot\nabla\psi_i)\p_i\p^\beta\phi dxdy\\
		&-\int_{\mathbb{T} \times\mathbb{R}} \p^\beta\left\{\frac{\phi}{\phi+1}\big(\mu\lap\psi_i+(\lambda+\mu)\p_i\dv\psi\big)-\frac{1}{M^2}\frac{\phi}{\phi+1}\p_i\phi+\frac{1}{M^2}\frac{(P'(\rho)-1)\p_i\phi}{\phi+1}\right\}\p_i\p^\beta \phi dxdy,\\
		\mathcal{R}_2^\beta=&\int_{\mathbb{T} \times\mathbb{R}} \p_i \p^\beta N_1  \p^\beta \psi_i dxdy:=\mathcal{R}_{2,1}^\beta+\mathcal{R}_{2,2}^\beta\\
		=&-\int_{\mathbb{T} \times\mathbb{R}}\p_i\p^\beta(\psi\cdot\nabla\phi)\p^\beta\psi_idxdy-\int_{\mathbb{T} \times\mathbb{R}}\p_i\p^\beta(\phi\dv\psi)\p^\beta\psi_i dxdy,\\     
		\mathcal{R}_3^\beta=&(\lambda+2\mu)\int_{\mathbb{T} \times\mathbb{R}} \p_i \p^\beta N_1 \p_i \p^\beta \phi dxdy:=\mathcal{R}_{3,1}^\beta+\mathcal{R}_{3,2}^\beta\\
		=&-\bar{\mu}\int_{\mathbb{T} \times\mathbb{R}}\p_i\p^\beta(\psi\cdot\nabla\phi)\p_i\p^\beta\phi dxdy-\bar{\mu}\int_{\mathbb{T} \times\mathbb{R}}\p_i\p^\beta(\phi\dv\psi)\p_i\p^\beta\phi dxdy.
	\end{align*}
	Multiplying \cref{ori}$_1$ by $\frac{1}{M^2}(\p_i \p^\beta \phi)$ and \cref{ori}$_2$ by $(-\lap\p^\beta \psi)$, adding them up, one has
	\begin{align}\label{a2}
			\frac{d}{dt} \Big( \frac{1}{M^2}&\norm{\nabla \p^\beta \phi}^2+\norm{\nabla \p^\beta \psi}^2 \Big) + \mu\norm{\lap \p^\beta \psi}^2 + (\lambda+\mu) \norm{\nabla \p^\beta \dv\psi}^2  \\
			&\quad\;\leq \frac{100}{M^2}\norm{\nabla^{|\beta|+1}\phi}^2+C_M\norm{\nabla^{|\beta|+1}\psi}^2+C_M\sum_{l=4}^5\mathcal{R}^\beta_l,\nonumber
	\end{align}
where
	\begin{align*}
		\mathcal{R}^\beta_4=&\frac{1}{M^2}\int_{\mathbb{T} \times\mathbb{R}} \p_i\p^\beta N_1 \p_i \p^\beta \phi dxdy:=\mathcal{R}_{4,1}^\beta+\mathcal{R}_{4,2}^\beta\\
		=&-\frac{1}{M^2}\int_{\mathbb{T} \times\mathbb{R}}\p_i\p^\beta(\psi\cdot\nabla\phi)\p_i\p^\beta\phi dxdy-\frac{1}{M^2}\int_{\mathbb{T} \times\mathbb{R}}\p_i\p^\beta(\phi\dv\psi)\p_i\p^\beta\phi dxdy,\\
		\mathcal{R}^\beta_5=&-\int_{\mathbb{T} \times\mathbb{R}}  \p^\beta N_{2i}  \lap\p^{\beta}\psi_idxdy=\int_{\mathbb{T} \times\mathbb{R}}\p^{\beta}(\psi\cdot\nabla\psi_i)\lap\p^{\beta}\psi_i dxdy\\
		&+\int_{\mathbb{T} \times\mathbb{R}} \p^\beta\left\{\frac{\phi}{\phi+1}\big(\mu\lap\psi_i+(\lambda+\mu)\p_i\dv\psi\big)-\frac{1}{M^2}\frac{\phi}{\phi+1}\p_i\phi+\frac{1}{M^2}\frac{(P'(\rho)-1)\p_i\phi}{\phi+1}\right\}\lap\p^{\beta}\psi_i dxdy.
	\end{align*}
	For $\abs{\beta}=1$,  we deduce that
	\begin{align}
	&\abs{\mathcal{R}_{1,1}^\beta+\mathcal{R}_{2,1}^\beta}\lesssim \norm{\nabla\psi}_{H^1}\norm{\nabla\phi}_{H^1}\norm{\nabla^2\psi}, \nonumber\\[0.3mm] 
	&\abs{\mathcal{R}_{3,1}^\beta}\lesssim\bar{\mu}\norm{\nabla\psi}_{L^\infty}\norm{\nabla^2\phi}^2+\bar{\mu}\norm{\nabla\phi}_{H^1}\norm{\nabla^2\phi}\norm{\nabla^2\psi}_{H^1},\nonumber \\[0.3mm] 
	&\abs{\mathcal{R}_{1,2}^\beta}\lesssim\norm{\nabla\phi}_{H^1}\norm{\nabla^2\phi}(\norm{\nabla\phi}+\norm{\nabla^2\psi})+\norm{\phi}_{H^2}\norm{\nabla^2\phi}(\norm{\nabla^2\phi}+\norm{\nabla^2\psi}),\notag\\[0.3mm]  
	&\abs{\mathcal{R}_{2,2}^\beta}\lesssim\norm{\nabla\phi}_{H^1}\norm{\nabla\psi}\norm{\nabla^2\psi}+\norm{\phi}_{L^\infty}\norm{\nabla^2\psi}^2, \notag\\[0.3mm]
    &\abs{\mathcal{R}_{3,2}^\beta}\lesssim\bar{\mu}(\norm{\nabla\psi}_{L^\infty}\norm{\nabla^2\phi}^2+\norm{\phi}_{L^\infty}\norm{\nabla^3\psi}\norm{\nabla^2\phi}+\norm{\nabla\phi}_{H^1}\norm{\nabla^2\psi}_{H^1}\norm{\nabla^2\phi}),\\[0.3mm]
	&\abs{\mathcal{R}_{4,1}^\beta}\lesssim\norm{\nabla\psi}_{L^\infty}\norm{\nabla^2\phi}^2+\norm{\nabla\phi}_{H^1}\norm{\nabla^2\phi}\norm{\nabla^2\psi}_{H^1},\nonumber\\[0.3mm]  
	&\abs{\mathcal{R}_{4,2}^\beta}\lesssim\norm{\nabla\psi}_{L^\infty}\norm{\nabla^2\phi}^2+\norm{\phi}_{L^\infty}\norm{\nabla^3\psi}\norm{\nabla^2\phi}+\norm{\nabla\phi}_{H^1}\norm{\nabla^2\psi}_{H^1}\norm{\nabla^2\phi},\nonumber\\[0.3mm] 
	&\abs{\mathcal{R}_5^\beta}\lesssim\norm{\nabla^3\psi}^2(\norm{\nabla\psi}_{H^1}+\norm{\phi}_{H^2})+\norm{\nabla^3\psi}\norm{\nabla\phi}_{H^1}\norm{\phi}_{H^2} \notag \\[0.3mm]
	&\qquad\quad+\norm{\nabla^3\psi}\norm{\nabla^2\psi}(\norm{\nabla\phi}_{H^1}+\norm{\psi}_{L^\infty}+\norm{\nabla\psi}_{H^1}),\nonumber
	\end{align}
	where we have used the fact that
	\begin{align*}
	\int_{\mathbb{T} \times\mathbb{R}} \psi \cdot \nabla \p_i \p^\beta \phi \p_i \p^\beta \phi dxdy = -\int_{\mathbb{T} \times\mathbb{R}} \dv\psi \abs{\p_i\p^\beta \phi}^2dxdy, 
	\end{align*}
	and
	\begin{align*}
	\int_{\mathbb{T} \times\mathbb{R}} \psi \cdot \nabla \p^\beta \psi_i \p_i \p^\beta \phi + \psi \cdot \nabla \p_i \p^\beta \phi \p^\beta \psi_i dxdy = -\int_{\mathbb{T} \times\mathbb{R}} \dv \psi \p^\beta \psi \cdot \nabla \p^\beta \phi dxdy. 
	\end{align*}
	By the a priori assupmtions \cref{aps} and \cref{77777},  we have
    \begin{align}\label{R}
	&\sum_{l=1}^3\bar{\mu}\int_0^T \abs{\mathcal{R}_l^\beta}dt\lesssim \mu^{3\alpha-2}+\mu^{\alpha-\frac{1}{2}}\int_0^T \norm{\nabla^3\psi}^2dt,\\ \nonumber
    &\sum_{l=4}^5\int_0^T \abs{\mathcal{R}_l^\beta}dt \lesssim \mu^{3\alpha-3}+\mu^{\alpha-\frac{3}{2}}\int_0^T \norm{\nabla^2\phi}^2dt. \nonumber
	\end{align}
	Since $M<\frac{1}{50}(\lambda+2\mu)^{-\frac{1}{2}}$, integrating \cref{a11,a2} over $(0,t)$ and combining \cref{estimateonxi}, \ref{2or}-\ref{Aor}, \cref{ed}, \cref{77777} and \cref{R}, we have
	\begin{align*}
			&\mu^2\norm{\nabla^2\phi}^2+\frac{\mu}{M^2}\int_0^T \norm{\nabla^2\phi}^2dt \lesssim\mu^{2\alpha+1}+\norm{\nabla\psi}^2+\bar{\mu}\int_0^T \norm{\nabla\phi}^2+\norm{\nabla^2\psi}^2dt+\mu^{\alpha-\frac{1}{2}}\int_0^T \norm{\nabla^3\psi}^2dt\\
			&\qquad\qquad\qquad\qquad\qquad\lesssim \mu^{2\alpha-1}+\mu^{\alpha-\frac{1}{2}}\int_0^T \norm{\nabla^3\psi}^2 dt,\\
            &\frac{1}{M^2}\norm{\nabla^2\phi}^2+\norm{\nabla^2\psi}^2+\mu\int_0^T \norm{\nabla^3\psi}^2dt \leq C_M \left(\mu^{2\alpha} + \int_0^T \norm{\nabla^2\psi}^2dt\right)+\frac{500}{M^2}\int_0^T\norm{\nabla^2\phi}^2dt\\
			&\qquad\qquad\qquad\qquad\qquad\leq C_M\mu^{2\alpha-2}+\frac{500}{M^2}\int_0^T \norm{\nabla^2\phi}^2dt. 
	\end{align*}
	Then we obtain
	\begin{align*}
	\frac{1}{M^2}\norm{\nabla^2\phi}^2+\norm{\nabla^2\psi}^2+\frac{1}{M^2}\int_0^T\norm{\nabla^2\phi}dt+\mu\int_0^T\norm{\nabla^3\psi}^2dt\lesssim\mu^{2\alpha-2}.
	\end{align*}
	For $\abs{\beta}\geqslant 2$, the same argument gives
	\begin{align*}
			\frac{1}{M^2}\norm{\nabla^{|\beta|+1} \phi}^2+\norm{\nabla^{|\beta|+1} \psi}^2+\frac{1}{M^2}\int_0^T \norm{\nabla^{|\beta|+1} \phi}^2dt+\mu\int_0^T\norm{\nabla^{|\beta|+2} \psi}^2 dt \lesssim  \mu^{2\alpha-\abs{\beta}-1}.
	\end{align*}
Therefore the proof of \cref{000000} is completed. 
\end{proof}
Finally, combining \cref{1or}, \ref{2or}, \ref{Aor} and \cref{ed}-\ref{000000}, we prove \cref{ape}.

\

\noindent{\bf Acknowledgments.}
Feimin Huang is partially supported by the National Key R\&D Program of China, grant No. 2021YFA1000800, and the National Natural Sciences Foundation of China, grant No. 12288201.

\end{document}